\documentclass[10pt]{article}
\usepackage[utf8]{inputenc}
\usepackage{lmodern}
\usepackage{amsmath}
\usepackage{amsfonts}
\usepackage{amsthm}
\usepackage{amssymb}
\usepackage{dsfont}
\usepackage{blindtext}
\usepackage{titlesec}
\usepackage{mathtools}
\usepackage{epigraph}
\usepackage{enumitem}
\usepackage{geometry}
\usepackage[colorlinks=true, linkcolor=red, citecolor=blue]{hyperref}
\usepackage{mathtools}
\usepackage{hyperref}
\usepackage{tabularx}
\usepackage{comment}

\usepackage{tikz}
\usetikzlibrary{shapes,positioning,intersections,quotes,patterns}
\usepackage{graphics}
\usepackage{graphicx}
\usetikzlibrary[topaths]

\usepackage{caption}
\usepackage{subcaption}

\usepackage{mathrsfs} 
\usepackage{textcomp}
\usepackage{newunicodechar}
\numberwithin{equation}{section}

\theoremstyle{plain}
\newtheorem{thm}{Theorem}[section]
\newtheorem{lem}[thm]{Lemma}
\newtheorem{prop}[thm]{Proposition}
\newtheorem{cor}[thm]{Corollary}

\theoremstyle{definition}
\newtheorem{remark}[thm]{Remark}
\newtheorem{definition}{Definition}

\title{Quantitative unique continuation for wave operators with a jump discontinuity across an interface and applications to approximate control}
\author{ Spyridon Filippas}
\date{October 2022}

\DeclareMathOperator\supp{supp}

\makeatletter
\def\keywords{
    \vspace{1ex}
    \noindent
    \if@twocolumn
      \small{\bf  Keywords}\/---$\!$    \else
      \begin{center}\small\ {\bf Keywords}\end{center}\quotation\small
    \fi}
\def\endkeywords{\vspace{0.6em}\par\if@twocolumn\else\endquotation\fi
    \normalsize\rm}
\makeatother

\begin{document}
	
	\maketitle
	
	\begin{abstract}
   In this article we prove quantitative unique continuation results for wave operators of the form $\partial_t^2-\textnormal{div}(c(x)\nabla \cdot)$ where the scalar coefficient $c$ is discontinuous across an interface of codimension one in a bounded domain or on a compact Riemannian manifold. We do not make any assumptions on the geometry of the interface or on the sign of the jumps of the coefficient $c$. The key ingredient is a local Carleman estimate for a wave operator with discontinuous coefficients. We then combine this estimate with the recent techniques of Laurent-Léautaud~\cite{Laurent_2018} to propagate local unique continuation estimates and obtain a global stability inequality. As a consequence, we deduce the cost of the approximate controllability for waves propagating in this geometry.
\end{abstract}

\begin{keywords}
  \noindent
Unique continuation, Carleman estimate, wave equation, jumps across an interface, approximate control, stability estimates
\medskip

\noindent
\textbf{2010 Mathematics Subject Classification:}
35B60, 
47F05,       
35L05, 
93B07, 
93B05, 
 35Q93 
\end{keywords}

	\newcommand{\rel}{\operatorname{Re} \lambda}
	\newcommand{\iml}{\operatorname{Im} \lambda}
	\newcommand{\norm}[3]{\left\Vert #1 \right\Vert_{#2}^{#3}}
	\newcommand{\normsurf}[3]{\left\vert #1 \right\vert_{#2}^{#3}}
	\newcommand{\R}{\mathbb{R}}
	\newcommand{\Rp}{\R_+}
	\newcommand{\Rm}{\R_-}
	\newcommand{\lr}{L^2(\R)}
	\newcommand{\lrp}{L^2(\R^+)}
	\newcommand{\lrm}{L^2(\R^-)}
	\newcommand{\lrn}{L^2(\R^{n+1})}
	\newcommand{\norvp}{\norm{v_+}{\lrp}{2} }
	\newcommand{\norvm}{\norm{v_-}{\lrm}{2} }
	\newcommand{\xip}{\xi^\prime}
	\newcommand{\lrpn}{L^2(\R^{n+1}_+)}
	\newcommand{\lrmn}{L^2(\R^{n+1}_-)}
	\newcommand{\ls}{L^2(\Sigma)}
	
	\newcommand{\lefp}{\left(}
	\newcommand{\rightp}{\right)}
	
	\newcommand{\cm}{c_{-}}
	\newcommand{\cp}{c_{+}}
	\newcommand{\invcp}{\cp^{-1}}
	\newcommand{\invcm}{\cm^{-1}}
	\newcommand{\invc}{c^{-1}}
	\newcommand{\xipt}{|\xip|^2+|\xi_t|^2}
	\newcommand{\xipc}{|\xip|^2}
	\newcommand{\xitc}{|\xi_t|^2}
	
	\newcommand{\errt}{E_t (v) }
	
	\newcommand{\fh}{l}
	\newcommand{\oh}{L}
	
	\newcommand{\poids }{Q_{\delta, \tau}^{\phi}}
	
		\newcommand{\poidsps }{Q_{\epsilon, \tau}^{\psi}}
	
	\newcommand{\cnop}{P_\phi}
	\newcommand{\cnp}{P^{+}_\phi}
	\newcommand{\cnm}{P^{-}_\phi}
	
	\newcommand{\phip}{\phi^{\prime}}
	
	\newcommand{\X}{L^2(X)}
	\newcommand{\Y}{L^2(Y)}
	
	\newcommand{\qr}{Q_2}
	\newcommand{\qi}{Q_1}
	
	\newcommand{\bjk}{\sum_{1\leq j,k \leq n-1}b_{jk}(x)}
	
	\newcommand{\op}{\textnormal{op}^w}
	
	\newcommand{\tchi}{\tilde{\chi}}
	\newcommand{\psis}{\psi_{\sigma}}
	
	\newcommand{\ths}{\theta_{\sigma}}
	\newcommand{\tths}{\check{\theta}_{\sigma}}
	
	\newcommand{\tphi}{\tilde{\phi}}
	\newcommand{\tphip}{\tilde{\phi}^\prime}
	
	\newcommand{\lkp}{\norm{\kappa^\prime}{L^\infty}{}}
	\newcommand{\lkps}{\norm{\kappa^\prime}{L^\infty}{2}}
	
	\newcommand{\Mc}{\mathcal{M}}
	\newcommand{\Lc}{\mathcal{L}}

	\newcommand{\psit}{\Psi}
	
	\newcommand{\psitt}{\Psi_{\tau}}
	
	\newcommand{\tran}{\mathcal{W}^{\theta,\Theta}}
	
	\newcommand{\tranphi}{\mathcal{W}^{\theta,\Theta}_{\phi}}

	\newcommand{\xmo}{\Xi_0}
	\newcommand{\xm}{\Xi_1}
	\tableofcontents

\section{Introduction}

For a wave operator $P$ the question of unique continuation consists in asking whether a partial observation of a wave on a small set $\omega$
is sufficient to determine \textit{the whole} wave. If this property holds, then the next natural question is if we can \textit{quantify} it. This is expressed via a stability estimate of the form
\begin{equation}
\label{quantitative phi}
\norm{u}{\Omega}{} \lesssim \phi\left(\norm{u}{\omega}{},\norm{Pu}{\Omega}{},\norm{u}{\Omega}{}\right),    
\end{equation}
with $\phi$ satisfying
$$
\phi(a,b,c) \stackrel{a,b \rightarrow0}{\longrightarrow} 0, \quad \textnormal{with }c\textnormal{ bounded.}
$$
Such estimates have numerous applications in control theory, spectral geometry and inverse problems. Concerning the wave operator a seminal unique continuation result was obtained by Robbiano in~\cite{Robbiano:91} and refined by Hörmander in~\cite{Hormander:92}. The optimal version of this \textit{qualitative} result was finally attained in the so called Tataru, Hörmander, Robbiano-Zuily Theorem \cite{Tataru:95, Tataru:99, Hor:97, RZ:98}. This theorem deals in fact with the more general case of operators with partially analytic coefficients and, in the particular case of a wave operator with coefficients independent of time, gives uniqueness across any non characteristic hypersurface. Recently, in \cite{Laurent_2018} the authors proved a quantitative version of the latter theorem which, for the wave equation, is optimal with respect to the observation time and the stability estimate obtained. Note that, a \textit{qualitative} uniqueness result is equivalent to an approximate controllability result, and a \textit{quantified} version of it gives an estimate of the control cost. The quantitative unique continuation result of~\cite{Laurent_2018} applies to (variants of) the operator $\partial_t^2-\Delta_g$ where $\Delta_g$ is an elliptic operator with $C^\infty$ coefficients. See also~\cite{BKL:16} for a related set of estimates concerning the wave operator.

However, in many contexts, waves propagate through singular media and therefore in the presence of non smooth coefficients. E.g. in the case of seismic waves \cite{Symes:83} or acoustic waves \cite{CDHCH:17,AdHG:17,CHKVU:19} propagating through the Earth's crust. Models proposing to describe such phenomena use discontinuous metrics and more precisely metrics which are piece-wise regular but presenting jumps along some hypersurfaces. See for instance the Mohorovičić discontinuity between the Earth's crust and the mantle. Another example arises in medical imaging. The human brain \cite{F:01,FCCP:17} has two main components: white and grey matter. These two have very different electric conductivities and models describing the situation are very similar to the preceding example.

The question of quantitative unique continuation across a jump discontinuity seems to be well understood in the elliptic/parabolic context. One of the first results (in the parabolic case) is \cite{DOP:02} where the operator $\partial^2_t-\textnormal{div}(c \nabla \cdot)$ is studied with a monotonicity assumption imposed on the scalar coefficient $c=c(x)$: the observation should take place in the region where the coefficient $c$ is smaller. In this article a global Carleman estimate was proved. Later, in the elliptic case in \cite{LRR:10} a similar result was obtained but without any restriction on the sign of the jump of the coefficient. These techniques were extended to the parabolic context in  \cite{LRR:11}.  The most recent (and general) result to the best of our knowledge is proved by Le Rousseau and Lerner \cite{le2013carleman} where the anisotropic case ($-\textnormal{div}(A(x) \nabla \cdot)$, with $A$ a matrix jumping across an interface) is treated.

The question of \textit{exact} control for waves with jumps at an interface has already been addressed in the book \cite{Lio:88}. A controllability result is proved for the operator $\partial_t^2-\textnormal{div}(c \nabla \cdot)$ with $c$ a piece-wise constant coefficient under a geometric assumption on the jump hypersurface and a sign condition on the jump. One of the first Carleman estimates was proved in the discontinuous setting in \cite{baudouin2007global}. With the same assumption on the coefficient and assuming that the interface is convex the authors prove linear quantitative stability estimates. Recently, in \cite{baudouin:hal-03211176} quantitative results were proved as well for interfaces that interpolate between star-shaped and convex.  Other related works are \cite{gagnon:hal-01958161} and \cite{buffe:hal-02473477}. 

However, to our knowledge the question of stability estimates without any particular geometric assumption on the interface has not been studied yet. This is the main object of this article.

\subsection{Setting and statement of main results}
\label{setting and main}
Let $(\mathcal{M},g)$ be a smooth connected compact $n$-dimensional Riemannian manifold with or without boundary. We consider $S$ an $(n-1)$-dimensional submanifold of $\mathcal{M}$ without boundary. We assume that $\mathcal{M}\backslash S= \Omega_{-}\cup \Omega_{+}$ with $\Omega_{-}\cap \Omega_{+}= \emptyset$. 

We consider a scalar coefficient $c(x)=\mathds{1}_{\Omega_{-}}c_-(x)+\mathds{1}_{\Omega_{+}}c_+(x)$ with $c_{\pm} \in C^{\infty}(\overline{\Omega}_{\pm})$ satisfying $0<c_{\min} < c(x)<c_{\max}$ uniformly on $\Omega_{-}\cup \Omega_{+}$ to ensure ellipticity. We shall work with the wave operator $P$ defined as 

\begin{equation}
 \label{definition of P }
 P=\partial^2_t-\textnormal{div}_g(c(x) \nabla_g ),\quad \textnormal{on }\R_t \times \Omega_-\cup \Omega_+.
\end{equation}

We consider for $(u_0,u_1) \in H^1_0( \mathcal{M}) \times L^2(\mathcal{M})$ the following evolution problem:
\begin{equation}
\label{system 1}
\begin{cases}
Pu=0 & \textnormal{in} \:(0,T) \times \Omega_{-}\cup\Omega_{+}\\
u_{|S_-}=u_{|S_+} & \textnormal{in}\: (0,T)\times S \\
(c\partial_\nu u)_{|S_-}=(c \partial_\nu u)_{|S_+} & \textnormal{in}\: (0,T)\times S \\
u=0 & \textnormal{in}\: (0,T)\times \partial\mathcal{M} \\
\left(u,\partial_t u \right)_{|t=0}=(u_0,u_1) & \textnormal{in} \: \mathcal{M},
\end{cases}
\end{equation}
where we denote by $\partial_{\nu}$ a nonvanishing vector field defined in a neighborhood of $S$, normal to $S$ (for the metric $g$), pointing into $\Omega_{+}$ and normalized for $g$. We denote as well by $u_{|S_{\pm}}$ the traces of $u_{|\Omega_{\pm}}$ on the hypersurface $S$. 

Notice that there are two extra equations in our system. These are some natural transmission conditions that we impose in the interface. These conditions imply that the underlying elliptic operator is self-adjoint on its domain and one can show using classical methods (for instance with the Hill-Yosida Theorem) that the system~\eqref{system 1} is well posed. For more details on this we refer to Section~\ref{the cerleman estimate}.

Our first result provides a quantitative unique continuation result from an observation region $\omega$ for the discontinuous wave operator $P$.

In section~\ref{not and def} we introduce $\mathcal{L}(\mathcal{M}, \omega)= \sup_{x \in \mathcal{M}}\textnormal{dist}(x,\omega),$ the “largest distance” of the subset $\omega$ to a point of $\mathcal{M}$, where dist is a distance function adapted to $(\mathcal{M},g,c)$.

\begin{thm}
\label{theorem quant}
Consider $(\mathcal{M},g), S, \Omega_{\pm}$ and $P$ as defined in \eqref{definition of P }. Then for any nonempty open subset $\omega$ of $\mathcal{M}$ and any $T>2\mathcal{L}(\mathcal{M},\omega)$, there exist $C, \kappa, \mu_0$ such that for any $(u_0,u_1) \in H^1_0( \mathcal{M}) \times L^2(\mathcal{M})$ and $u$ solving \eqref{system 1}
one has, for any $\mu\geq \mu_0$,

$$
\norm{(u_0,u_1)}{L^2\times H^{-1}}{}\leq Ce^{\kappa \mu} \norm{u}{L^2((0,T)\times \omega)}{}+\frac{C}{\mu}\norm{(u_0,u_1)}{H^1\times L^2}{}.
$$
If moreover $\partial \mathcal{M}\neq \emptyset$ and $\Gamma$ is a non empty open subset of $\partial \mathcal{M}$, for any $T>2 \mathcal{L}(\mathcal{M},\Gamma)$, there exist $C, \kappa, \mu_0>0$ such that for any $(u_0,u_1) \in H^1_0( \mathcal{M}) \times L^2(\mathcal{M})$ and $u$ solving \eqref{system 1}, we have
$$
\norm{(u_0,u_1)}{L^2\times H^{-1}}{}\leq Ce^{\kappa \mu} \norm{\partial_{\nu_{\Gamma}}u}{L^2((0,T) \times \Gamma)}{}+\frac{C}{\mu}\norm{(u_0,u_1)}{H^1\times L^2}{}.
$$
\end{thm}

\begin{remark}
In fact one can take $\mu >0$ in the statement of the above theorem. However we preferred to state it in this way in order to stress out the fact that this estimate is interesting only when $\mu$ is large.
\end{remark}

With the above one can recover the following qualitative result: “If we do not see anything from $\omega$ during a time $T$ strictly larger than $2 \mathcal{L}(\mathcal{M}, \omega)$, then there is no wave at all.” Indeed, if $ \norm{u}{L^2((0,T)\times \omega)}{}=0$, then letting $\mu \rightarrow +\infty$ in the above inequality implies that $(u_0,u_1)=0$. 

An important aspect of this theorem is that there is no assumption on the sign of the jump of the coefficient $c$ and consequently the observation region $\omega$ can be chosen indifferently on $\Omega_{-}$ or $\Omega_{+}$. Let us explain why this is quite surprising. Suppose, to fix ideas, that $c_-<c_+$ are two constants. We can then interpret $c_-$ and $c_+$ as the the speed of propagation of a wave travelling through two isotropic media $\Omega_{-}$ and $\Omega_{+}$ with different refractive indices, $n_-$ and $n_+$ respectively (recall that $n_\pm=1/c_\pm$). Imagine that a wave starts travelling from a region that is inside $\Omega_{-}$. One has $\frac{c_-}{c_+}=\frac{n_+}{n_-}$ and therefore the assumption $c_-<c_+$ translates to $n_->n_+$. Then Snell-Descartes law states that when a wave travels from a medium with a higher refractive index to one with a lower refractive index there is a critical angle from which there is \textit{total internal reflection}, that is no refraction at all. At the level of geometric optics, that is to say, in the high frequency regime such a wave stays trapped inside $\Omega_-$. Therefore one expects that, at least at high frequency, no information propagates from $\Omega_-$ to $\Omega_+$, following the laws of geometric optics. Our result (see Theorem~\ref{corol of theorem with log}) states that the intensity of waves in $\Omega_+$ is at least exponentially small in terms of the typical frequency $\Lambda$ of the wave.

\begin{figure}
    \centering   
         \begin{subfigure}{.5\textwidth}
  \centering
  \includegraphics[scale=0.5]{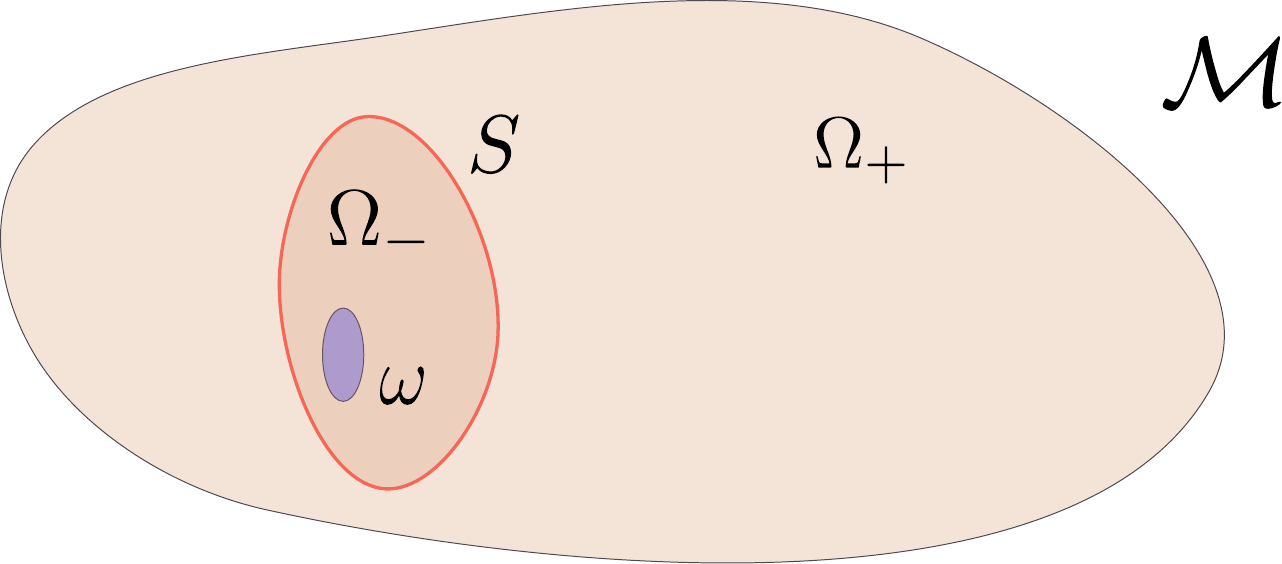}
  \caption{The observation takes place inside $\Omega_{-}$.}
  \label{case 2}
\end{subfigure}%
    \begin{subfigure}{.5\textwidth}
  \centering
  \includegraphics[scale=0.5]{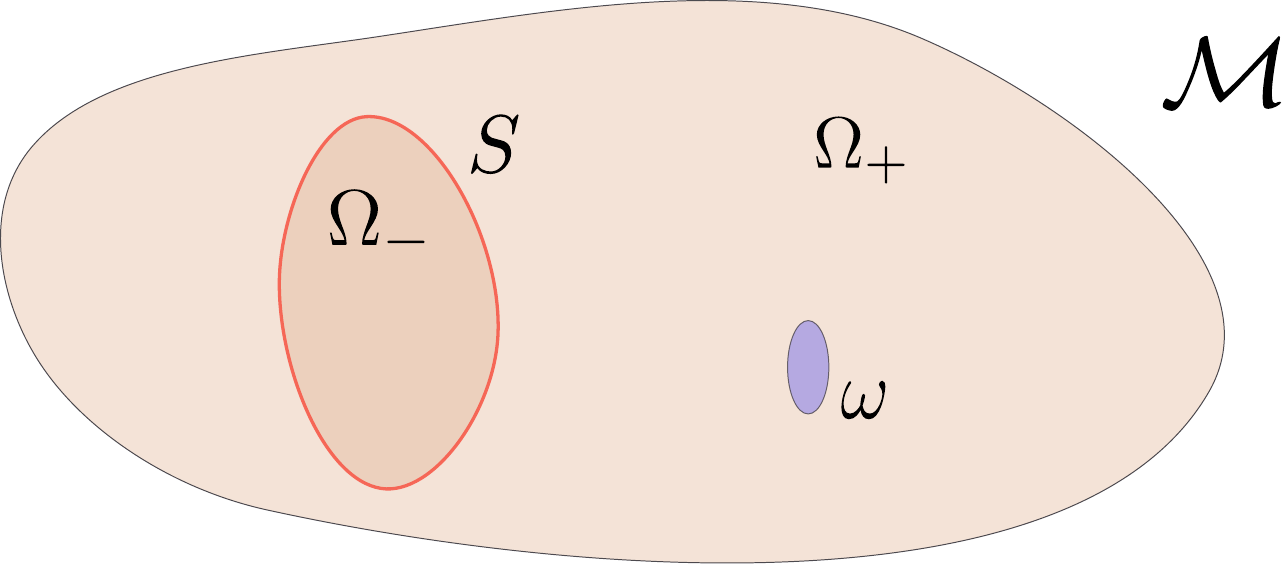}
  \caption{The observation takes place inside $\Omega_{+}$. If $c_-<c_+$ then a part of the wave may be trapped inside $\Omega_{-}$. Nevertheless, the quantitative unique continuation and its consequences still hold.}
  \label{case 1}
\end{subfigure}%
\end{figure}

\bigskip

We can reformulate Theorem~\ref{theorem quant} in a way closer to quantitative estimates such as \eqref{quantitative phi}.  Indeed, optimizing the inequalities of Theorem~\ref{theorem quant} with respect to $\mu$ yields the following result (which we state only in the interior observation case):

\begin{thm}
\label{corol of theorem with log}
Under the assumptions of Theorem~\ref{theorem quant} there exists $C>0$ such that for all $(u_0,u_1) \in H^1_0( \mathcal{M}) \times L^2(\mathcal{M})$ with $(u_0,u_1)\neq(0,0)$ one has:
\begin{align}
\label{fonction log}
 \norm{(u_0,u_1)}{H^1\times L^2}{}&\leq Ce^{C \Lambda}\norm{u}{L^2((0,T)\times \omega)}{}, \nonumber \\
\norm{(u_0,u_1)}{L^2 \times H^{-1}}{}& \leq C \frac{\norm{(u_0,u_1)}{H^1\times L^2}{}}{\log \left(1+ \frac{\norm{(u_0,u_1)}{H^1\times L^2}{}}{ \norm{u}{L^2((0,T)\times \omega)}{}}\right)}.
\end{align}
where $\Lambda= \frac{\norm{(u_0,u_1)}{H^1\times L^2}{}}{\norm{(u_0,u_1)}{L^2 \times H^{-1}}{}}$. 
\end{thm}

Note that $\Lambda$ represents the typical frequency of the initial data. Theorem~\ref{corol of theorem with log} is a direct consequence of Theorem~\ref{theorem quant} and Lemma A.3 in \cite{Laurent_2018}. Notice that the function 
$$
x \mapsto \frac{1}{\log(1+1/x)},
$$
appearing in the right hand side of~\eqref{fonction log} has been tacitly extended by continuity by $0$ when $x=0$.

In \cite [proof of Theorem 2, Section 3]{Robbiano:95} it is shown that such a quantitative information can lead to an estimate for the cost of the approximate controllability. We state the case of internal control, a similar result holds for approximate boundary controllability as well.

\begin{thm}[Cost of approximate interior control]
\label{cost of control}
Consider $\mathcal{M}$, $S$ and $\omega \subset \mathcal{M}$ as before. Then for any $T>2\mathcal{L}(\mathcal{M}, \omega)$ there exist $C,c>0$ such that for any $\epsilon>0$ and any $(u_0,u_1) \in H^1_0(\mathcal{M})\times L^2(\mathcal{M})$, there exists $f \in L^2((0,T)\times \omega)$ with 
$$
\norm{f}{L^2((0,T)\times \omega)}{}\leq Ce^{c/\epsilon}\norm{(u_0,u_1)}{H^1_0(\mathcal{M})\times L^2(\mathcal{M})}{}
$$
such that the solution of 
\begin{equation*}
    \begin{cases}
    Pu=\mathds{1}_{\omega}f & \textnormal{in} \:(0,T) \times \Omega_{-}\cup\Omega_{+}\\
    u_{|S_-}=u_{|S_+} & \textnormal{in}\: (0,T)\times S \\
    (c\partial_\nu u)_{|S_-}=(c \partial_\nu u)_{|S_+} & \textnormal{in}\: (0,T)\times S \\
    u=0 & \textnormal{in}\: (0,T)\times \partial\mathcal{M} \\
    \left(u,\partial_t u \right)_{|t=0}=(u_0,u_1) & \textnormal{in} \: \mathcal{M},
    \end{cases}
\end{equation*}
satisfies 
$$
\norm{(u,\partial_t u)_{|t=T}}{L^2\times H^{-1}}{} \leq \epsilon \norm{(u_0,u_1)}{H^1_0\times L^2}{}.
$$
\end{thm}

In other words, if we act on the region $\omega$ during a time $T>2 \mathcal{L}(\mathcal{M}, \omega)$ we can drive our solution from energy $1$ (in $H^1 \times L^2$) to $\epsilon$ close to $0$ (in $L^2 \times H^{-1}$). Additionally, this comes with an estimate of the energy of the control which is of the order of $ e^{c/\epsilon}$. In the analytic context and without the presence of an interface it was shown in \cite{Leb:Analytic} that this form of exponential cost is optimal in the absence of Geometric Control Condition~\cite{BLR:92}.

In the more general hypoelliptic context of \cite{LL:20dampedhypo} the result of Theorem~\ref{cost of control} is stated as \textit{approximate observability} for the wave equation. It is shown by the authors of this article that such a property implies some resolvent estimates (Proposition 1.11 in \cite{LL:20dampedhypo}) which in turn give a logarithmic energy decay estimate for the damped wave equation (see Theorem 1.5 in \cite{LL:20dampedhypo}). Consequently, Theorem~\ref{cost of control} combined with the results of \cite{LL:20dampedhypo} provides a different proof for theorems that were already obtained using Carleman estimates for elliptic/parabolic operators (see~\cite{LRR:10} or \cite{le2013carleman}).

\begin{remark}
    We have assumed that the interface $S$ decomposes $\mathcal{M}$ in two disjoint parts $\Omega_+$ and $\Omega_-$. However the same results can be obtained for other more general geometric situations as well. This comes from the fact that the key ingredient is a \textit{local} quantitative estimate (see Theorem~\ref{local quant estimate}). See also the figure in~\cite[Section 1.3.2]{LRLR:13}
\end{remark}
 
\subsection{Strategy of the proof and organization of the paper}

\subsubsection{The Carleman estimate}

One of the main tools for dealing with problems of local unique continuation across a hypersurface $\{\phi=0\}$  is \textit{Carleman estimates}. The idea, introduced by Carleman in \cite{Carleman:39}, is to prove an inequality involving a weight function $\psi$ and a large parameter $\tau$, of the form
$$
\norm{e^{\tau \psi}Pu}{L^2}{} \gtrsim \norm{e^{\tau \psi }u}{L^2}{},\quad \tau \geq \tau_0,
$$
uniform in $\tau$. The weight function $\psi$ is closely related to the level sets of the function $\phi$ which defines implicitly the hypersurface. In an heuristic way, the chosen weight re-enforces the sets where $u$ is zero and propagates smallness from sets where $\psi$ is big to sets where $\psi$ is small. Since Carleman estimates are already quantitative in nature they provide a good starting point for results of the form \eqref{quantitative phi}. We point out the fact that this is a \textit{local} problem. In order to obtain a \textit{global} result one needs in general to propagate the local one by passing through an appropriate family of hypersurfaces.

The core of this article is to prove a local Carleman estimate \textit{in a neighborhood of the interface}, containing a microlocal weight in the spirit of \cite{Tataru:95, Hor:97, RZ:98}. The presence of discontinuous coefficients complicates significantly this task. In general, for a Carleman estimate to hold a condition involving the principal symbol of the operator and the hypersurface needs to hold, the so-called pseudoconvexity condition (see for instance \cite{Hor67}). These results are based on microlocal analysis arguments and some regularity is necessary for the estimate to hold. In our case we explicitly construct an appropriate weight function and show our estimate for this particular weight. Our proof is inspired by that of Lerner-Le Rousseau in the elliptic case \cite{le2013carleman} and relies in a factorization argument. Even though the behavior of our (hyperbolic) operator may be very different we consider in our context the wave operator as a “perturbation” of a Laplacian, in the spirit of~\cite{Tataru:95, Hor:97, RZ:98, Laurent_2018}. Let us explain why. For the sake of exposition, consider that $\mathcal{M}=\R^n$, $g$ is the Euclidean metric and $c$ is piecewise constant with a jump across the hypersurface $S=\{x_n=0\}$. That is $c=c_+\mathds{1}_{\{x_n>0\}}+c_-\mathds{1}_{\{x_n<0\}}$. Then the principal symbol of the wave operator $\Box=\partial_t^2-c \Delta$ is 
$$
\sigma (\Box)=-\xi^2_t+c |\xi|^2.
$$
The inequality we want to prove contains the weight $e^{-\epsilon \frac{D^2_t}{2 \tau}}$ which heuristically localizes close to $\{\xi_t=0\}$, in other words localizes in a microlocal region where our operator is \textit{close to being elliptic}.

 In order to obtain a Carleman estimate with a weight $\psi$ one obtains an inequality for the conjugated operator $P_\psi$ defined by $P_\psi:=e^{\tau \psi}Pe^{-\tau \psi}$. In \cite{le2013carleman} the authors use an idea which can be traced back to Calderón \cite{Calderon:58}.  They factorize the conjugated operator $P_\psi$ as a product of two first-order operators and prove estimates for each first order factor. In the elliptic case $P=-c \Delta$ and for a weight $\psi$ depending only on $x_n$ one has the following factorization:
	\begin{align*}
		P_{\psi}&= c_{+}\left(D_n+i(\tau \psi^\prime +|\xip|)\right)\left(D_n+i(\tau \psi^\prime -|\xip|)\right)\\
		&\hspace{4mm}+c_{-}\left(D_n+i(\tau \psi^\prime -|\xip|)\right)\left(D_n+i(\tau \psi^\prime +|\xip|\right)\\
		&=c_{+}\left(D_n+ie_+\right)\left(D_n+if_+\right)+c_{-}\left(D_n+if_-\right)\left(D_n+ie_-\right).
	\end{align*}
Here the operator has been identified with its symbol in the tangential variables. This allows to work on the normal variable $x_n$ and treat the other ones as parameters. In more technical terms, we can use the tools of tangential symbolic calculus in all the variables but $x_n$ and try to obtain good one-dimensional estimates for the first order operators $D_n+ie_\pm$ and $D_n+if_\pm$. The general principle is that the sign of $e_\pm,f_\pm$ determines the quality of the one dimensional estimates. The choice of the weight function $\psi$ is made so that the following key property is satisfied: we can cover the tangential dual space by $\Gamma_1$, $\Gamma_2$ such that $(x^\prime,\tau, \xip)\in \mathbb{R}^{n-1} \times \mathbb{R}  \times \mathbb{R}^{n-1}=\Gamma_1 \cup \Gamma_2 $ with $f_+ \geq 0$ on $\Gamma_1$ and $f_- \leq 0 $ on $\Gamma_2$.

In the proof of the present article, in which $P$ is a wave operator, we work essentially in two microlocal zones. In the first zone, the operator $P$ is not microlocally elliptic. In this zone, we consider terms involving $D_t$ as admissible errors. Using the operator $e^{-\epsilon\frac{D^2_t}{2 \tau}}$ allows us then to obtain the desired estimate. In the second zone, the operator $P$ is microlocally elliptic one can follow the proof provided in the elliptic case in \cite{le2013carleman}.

In Section~\ref{the cerleman estimate} we give the precise statement of the local Carleman estimate that we prove and we describe adapted local coordinates in a small neighborhood of the interface to prepare the proof. In Section~\ref{proof for the general case} we prove the Carleman estimate. We considered helpful to give first a proof for a toy model (constant coefficients case) in Section~\ref{proof of toy model}. In this case one can simply work on the Fourier domain without having to work with pseudodifferential operators (see ~\cite[Chapter 10]{lerner2019carleman}). At the same time it allows to understand the core of the arguments which will be used in the general case too.

\bigskip
\subsubsection{Using the Carleman estimate}

The next step is to use the Carleman estimate. To do this, one needs to obtain a local quantitative version which can be iterated. This has been done in the smooth case by Laurent and Léautaud in \cite{Laurent_2018}. There, the estimate has the same form as that obtained in \cite{Laurent_2018} in a (arbitrarily small) neighborhood of the interface $S$. Thus, we are able to use it “once” to pass on the other side of $S$ and then combine it with the results of \cite{Laurent_2018}. In Section~\ref{local quant estimate} we show how one can use the techniques of \cite{Laurent_2018} to obtain such a local quantitative estimate in a neighborhood of the interface where the coefficients jump. Finally in Section~\ref{propagation of info} we prove that indeed, one can propagate the quantitative estimate by combining our local estimate with the analogous results in the smooth case.

\bigskip

\subsection{Some notations and definitions }
\label{not and def}
We recall in this Section some elementary geometric facts and we give the precise definition of the distance which is used in the statement of Theorem~\ref{theorem quant} and Theorem~\ref{cost of control}.

\bigskip

Let us recall that the geometric setting is given at the beginning of Section~\ref{setting and main}. The interface $S$ has a natural metric given by the restriction of $g$ to $TS$. In local coordinates we have:
$$
\partial_{\nu}=\sum_j \nu^j \partial_{x_j}, \quad \nu^j=\lambda\sum_k n_k g^{jk}, \quad |\nu_g|=1,
$$
with $g^{ij}g_{jk}=\delta_{ik},$ $\lambda^2=\left(\sum_{i,j}g^{ij}n_i n_j\right)^{-1}$ and $n$ is the normal to $S$ for the Euclidean metric in the chosen local coordinates pointing into $\Omega_{+}$. We denote by $(\cdot,\cdot)_g= g(\cdot,\cdot)$ the inner product on  $T \mathcal{M}$. The Riemannian gradient of a function $f$ is defined in an intrinsic manner by
$$
( \nabla_g f , X )_g=df(X), \quad \textnormal{for any smooth vector field } X.
$$
The integral of a function $f$ is defined by $\int f:= \int_\mathcal{M} f(x) d\textnormal{Vol}_g(x)$, where $ d\textnormal{Vol}_g(x)$ is the Riemannian volume element. The divergence operator, acting on a vector field $X$ is defined by the relation
$$
\int u \textnormal{ div}_g X=-\int (\nabla_g u, X  )_g, \quad \textnormal{for all } u\in C^\infty_0(\textnormal{Int}\mathcal{M}).
$$
Let us recall the expression of these objects in local coordinates. We consider $f$ a smooth function on $\mathcal{M}$ and $X=\sum_i X^i \partial_{x_i}$, $Y=\sum_i Y^i \partial_{x_i}$ two smooth vector fields on $\mathcal{M}$. We have:
\begin{align*}
    ( X, Y )_g &=\sum_{i,j} g_{ij}X^iY^j, \quad \nabla_g f =\sum_{i,j}g^{ij}(\partial_j f)\partial_{x_i}, \quad
     \textnormal{div}_g(X)= \frac{1}{\sqrt{\det g}}\sum_i  \partial_i \left(\sqrt{\det g}X_i\right).
\end{align*}
We finally have:
$$
\textnormal{div}_g(c(x) \nabla_g f)=\frac{1}{\sqrt{\det g}} \sum_{i,j}\partial_{x_i}\left(c g^{ij} \sqrt{\det g} \partial_{x_j}f\right), \quad x\in \Omega_{-}\cup \Omega_{+}.
$$

We want to define the natural distance associated to the operator $P$ appearing in Theorem~\ref{theorem quant}. Consider the piecewise smooth metric $g_c$ by $g_c:=\mathds{1}_{\Omega_-}c^{-1}_-g +\mathds{1}_{\Omega_+}c^{-1}_+g$.

\begin{definition}
	An admissible path $  C^{\infty}([0,1];\mathcal{M}) \ni\gamma: [0,1] \rightarrow \mathcal{M}$ is a path satisfying the conditions:
 \begin{itemize}
     \item $\gamma$ does not have self-intersections
     
     \item $\gamma$ intersects the interface $S$ a finite number of times 
     
     \item $\gamma(t) \in S \Longrightarrow \gamma^\prime(t) \perp T_{\gamma(t)}S $
 \end{itemize}
\end{definition} 

In particular, if $\gamma$ is an admissible path then the map $t \mapsto |\gamma^\prime(t)|_{g_c(\gamma(t))} \in L^{\infty}([0,1]; \mathcal{M})$ is bounded and consequently we can define its length by the usual formula:
$$
\textnormal{length}(\gamma)=\int_{0}^{1}|\gamma^{\prime}(t)|_{g_c(\gamma(t))}dt.
$$

We now define the distance we will be working with:

\begin{definition}
The distance of two points $x_0, x_1 \in \mathcal{M}$ is defined as:
$$
\textnormal{dist}(x_0,x_1)=\inf \{ \textnormal{length}(\gamma)|\:\: \gamma\textnormal{ admissible path, }\gamma(0)=x_0, \gamma(1)=x_1 \}.
$$
\end{definition}

We can now as well define the \textit{largest distance of a subset} $E \subset \mathcal{M}$ to $\mathcal{M}$ by
\begin{equation}
    \label{greatest distance}
    \mathcal{L}(\mathcal{M},E):= \sup_{x \in \mathcal{M}}\textnormal{dist}(x,E),
\end{equation}
where $$
\textnormal{dist}(x,E)= \inf_{y\in E} \textnormal{dist}(x,y).$$

\begin{remark}
    Notice that the conditions imposed on the family of admissible paths do not pose any important restriction since any Lipschitz path can be replaced by an admissible one up to increasing its length by an $\epsilon>0$ arbitrarily small. Since we take the infimum of these lengths the distance remains the same.
\end{remark}

\bigskip

\noindent \textbf{Acknowledgements} The author would like to thank C. Laurent and M. Léautaud for introducing him to the problem as well as for discussions, encouragements, and patient guidance.

	\section{The Carleman estimate}
	\label{the cerleman estimate}

The key ingredient for the proof of Theorem~\ref{theorem quant} is a local Carleman estimate. Since we will work in space-time it is convenient to consider $\Sigma:=\R_t \times S$ with $S$ the smooth interface
of the manifold $\mathcal{M}$ defined in Section~\ref{setting and main}. Therefore, $\Sigma$ is a smooth hypersurface of $\R_t \times \mathcal{M}.$ We define as well $\Omega_{t,\pm}:= \R_t \times \Omega_{\pm}$.

\subsection{General transmission conditions}
\label{general trans cond }
We want to derive a Carleman estimate for the wave operator
$$
\Box = P:=\partial_t^2-\textnormal{div}_g(c(x) \nabla_g),
$$
where the scalar coefficient $c$ satisfies $0<c_{\min}<c(x)<c_{\max}<+ \infty$ uniformly for $x\in \Omega_{-} \cup \Omega_{+}$ to make sure that the ellipticity property is satisfied. We recall as well that $c_{|\Omega_{t,\pm}} \in C^\infty(\overline{\Omega}_{t,\pm})$ but it jumps across the interface $S$.  Since the operator $P$ has discontinuous coefficients one needs to be careful with its domain. Indeed, given a function $u=1_{\Omega_{t,-}} u_-+1_{\Omega_{t,+}}u_+$ with $u_\pm \in C^\infty(\R_t \times \mathcal{M}) $ one has in the distributional sense 
$$
\nabla_g u = 1_{\Omega_{t,-}} \nabla_g u_- +1_{\Omega_{t,+}}\nabla_g u_+ +(u_--u_+)\delta_\Sigma \nu,
$$
where $\delta_\Sigma$ is the surface measure on $\Sigma$ and $\nu$ is the unit normal vector field pointing into $\Omega_{t,+}$. We impose then that
\begin{equation}
\label{trans cond general 1}
u_-{_{|\Sigma}}=u_+{_{|\Sigma}},    
\end{equation}
and the singular term is removed. Similarly, calculating 
$$
\textnormal{div}(c(x)\nabla_g u),
$$
we see that the condition 
\begin{equation}
\label{trans cond general 2}
\cp \partial_{\nu} {u_{+}}_{|\Sigma}=\cm\partial_{\nu} {u_{-}}_{|\Sigma}
\end{equation}
combined with \eqref{trans cond general 1} gives the equality
$$
\textnormal{div}_g(c(x) \nabla_g u)=1_{\Omega_{t,-}}\textnormal{div}(c_- \nabla_g u_- )+1_{\Omega_{t,+}}\textnormal{div}( c_+ \nabla_g u_+).
$$
We define then $\mathcal{W}$ as the space containing functions of the form 
\begin{equation}
\label{form of the functions}
u=1_{\Omega_{t,-}} u_-+1_{\Omega_{t,+}}u_+, 
\end{equation}
with $u_\pm \in C^\infty_0(\R_t \times \mathcal{M})$ and such that \eqref{trans cond general 1} and \eqref{trans cond general 2} hold. These conditions are called \textit{transmission conditions} and for $u \in \mathcal{W}$ one has $Pu \in L^2$. The above calculations show as well that $P$ is formally self adjoint on this domain and therefore by classical methods (energy estimates or semi-group theory) one has that the evolution problem \eqref{system 1} is well posed.

\begin{remark}
	The first transmission condition expresses the continuity across the interface and the second one the continuity for the normal flux. Notice that the second condition excludes many smooth functions from the space $\mathcal{W}$. On the other hand elements of $\mathcal{W}$ are Lipschitz continuous and in particular one has  $\mathcal{W} \subset H^1$.
\end{remark}

For technical reasons it will be useful to work with non-homogeneous transmission conditions as well. More precisely, we shall denote by $\tran$ the space of functions of the form \eqref{form of the functions} satisfying additionally the following \textit{non homogeneous transmission conditions}:
\begin{align}
\label{non hom trans cond intrinsq}
{u_-}_{|\Sigma}&={u_+}_{|\Sigma}+\theta \\
\cp \partial_{\nu} {u_{+}}_{|\Sigma}&=\cm\partial_{\nu} {u_{-}}_{|\Sigma}+\Theta,
\end{align}
where $\theta$ and $\Theta$ are smooth functions of the interface $\Sigma$.

\subsection{Local setting in a neighborhood of the interface}

\label{local setting}
Since we show a local Carleman estimate we can state it directly in adapted local coordinates. In a sufficiently small neighborhood $V$ of a point $x_0$ of $S$ one can use normal geodesic coordinates with respect to the spatial variables $x$ (see \cite{Hoermander:V3}, Appendix C.5, \cite{rousseau2022elliptic} Section 9.4 ). In such a coordinate system the interface $S$ is given by $S=\{x_n=0\}$ and the \textit{principal part} of the operator $P$ denoted by $P_2$ takes the form (on both sides of the interface)
$$
P_2=\partial_t^2-c(x)(\partial^2_{x_n}-r(x,\partial_{x^\prime}/i)),
$$
with $r(x, \xip)$ a $x$family of second order polynomials in $\xip$ that satisfy 
$$
r(x, \xip)\in \R,\quad C_1|\xip|^2 \leq r(x,\xip)\leq C_2 |\xip|^2,
$$
for $x\in V$, $\xip \in \R^{n-1}$ and $0<C_1<C_2 < \infty$. 
The transmission conditions become after this change of variables:
\begin{align}
\label{non hom trans cond 1}
{u_-}_{|x_n=0}&={u_+}_{|x_n=0}+\theta \\
\label{non hom trans cond 2}
\cp \partial_{x_n} {u_{+}}_{|x_n=0}&=\cm\partial_{x_n} {u_{-}}_{|x_n=0}+\Theta,
\end{align}

In this setting, the two sides of the interface become $\Omega_{t,+}=\{x_n>0\}$ and $\Omega_{t,-}=\{x_n<0\}$. We shall use the notation 
$$
H_{\pm}=1_{\Omega_{t,\pm}}.
$$
and we have for the coefficient $c:$
	$$
	c(x)=H_-c_-(x)+H_+c_+(x).
	$$
The space $\tran$ consists in this local context of functions $u$ of the form 
$$ 
u=H_-u_- +H_+u_+,\quad u_\pm \in C^{\infty}_0(\R_t \times V),
$$
satisfying also \eqref{non hom trans cond 1} and \eqref{non hom trans cond 2}. We define as well
$$
P^\pm:=\partial_t^2-\text{div}(c_\pm \nabla \cdot).
$$

Following \cite{Tataru:95, Hor:97, RZ:98} we are seeking to prove a Carleman estimate containing the microlocal weight
\begin{equation}
	\label{def of microlocal weight}
	Q_{\delta, \tau}^{\phi}=e^{-\frac{\delta |D_t|^2}{2 \tau}} e^{\tau \phi }.
\end{equation}
We shall take $\phi$ in the following form:
\begin{equation}
\label{def of weight phi}
\phi=\phi(x_n)= \left(\alpha_{-} x_n+\frac{\beta x_n^2}{2}\right)H_- +\left(\alpha_{+}x_n  +\frac{\beta x_n^2}{2}\right)H_+ ,\quad \alpha_{\pm}, \beta >0.
\end{equation}

	The parameters $\alpha_\pm$ and $\beta$ will be chosen in the sequel. The parameter $\beta$ will be taken large and is related to the sub-ellipticity property (see for instance \cite{Hoermander:63} or \cite{Hoermander:V4}) which is necessary for a Carleman estimate to hold. The choice of $\alpha_\pm$ comes from the construction in the microlocally elliptic case (see Lemma~\ref{Lemma elliptic region}). It is a geometric condition on the interface that requires the jump of $\partial_{x_n}\phi$ (i.e. $\alpha_{+}-\alpha_{-}>0$) to be sufficiently large (see~\ref{geometric assumption}).

	\bigskip
	
	The following is our main Carleman estimate and its proof will occupy a large part of this article:
	
	\begin{thm}
		\label{the carleman ineq thm}
In the geometric situation presented just above let $(t_0,x_0) \in \R_t \times V$. Then there exists an appropriate weight $\phi$ and some positive constants $C$, $\tau_0$, $\delta$, $d$, $r_0$ such that
		\label{theorem}
		\begin{multline*}
			C \norm{H_-Q_{\delta, \tau}^{\phi} P^- u_-}{L^2}{2}+ C\norm{H_+Q_{\delta, \tau}^{\phi} P^+u_+}{L^2}{2} 
			\\
			+Ce^{-d \tau} \left(
			\tau^3 \norm{e^{\tau \phi} u }{L^2}{2}+\tau \norm{H_{+}\nabla e^{\tau \phi} u_{+}}{L^2}{2}+\tau \norm{H_{-}\nabla e^{\tau \phi} u_{-}}{L^2}{2}
			\right)+C T_{\theta, \Theta}\\
			 \geq \tau^3 \norm{Q_{\delta, \tau}^{\phi} u}{L^2}{2}+ \tau \norm{H_+ \nabla Q_{\delta, \tau}^{\phi}  u_+}{L^2}{2}+\tau \norm{H_- \nabla Q_{\delta, \tau}^{\phi}  u_-}{L^2}{2},
		\end{multline*}
		for $u \in \tran$ such that $\supp u \subset B((t_0,x_0),r_0)$ and $\tau \geq \tau_0$, where 
		\begin{equation}
		    \label{def of T}
		T_{\theta, \Theta}=\tau^3 \normsurf{Q_{\delta, \tau}^{\phi}\theta}{\ls}{2}+\tau \normsurf{Q_{\delta, \tau}^{\phi}\nabla \theta}{\ls}{2}+\tau \normsurf{Q_{\delta, \tau}^{\phi}\Theta}{\ls}{2}.
		\end{equation}
	\end{thm}

Several remarks are in order:

\begin{remark}
	Observe that the choice of $\phi$ is well adapted to the geometric situation: for small $|x_n|$ we have that $\Sigma=\{x_n=0\}=\{\phi=0\}$ and the two sides of the interface are given by $\{\phi>0\}$ and $\{\phi<0\}$. We also have (for $|x_n|$ sufficiently small) that $\phi^\prime >0$, and consequently $\phi$ is larger in $\Omega^+$: $\Omega^+$ is the observation region and the Carleman estimate of Theorem~\ref{the carleman ineq thm} will propagate uniqueness from $\Omega^+$ to $\Omega^-$. In particular, the weight function $\phi$ is \textit{suitable for observation} in the sense of \cite[Property (1-12)]{le2013carleman}.
\end{remark}

\begin{remark} 
Notice that in the case of homogeneous transmission conditions $\theta=\Theta = 0$ one has $T_{\theta,\Theta}=0$. The calculations carried out in Section~\ref{general trans cond } imply that $Pu \in L^2$ and therefore 
$$
\norm{H_-P^-u_-}{L^2}{}+\norm{H_+P^+u_+}{L^2}{}=\norm{Pu}{L^2}{}.
$$
Using moreover the fact that $u \in H^1$ we can write the estimate of Theorem~\ref{the carleman ineq thm} in a more concise way as:
\begin{equation*}
    C\left(\norm{Q_{\delta, \tau}^{\phi} P u}{L^2}{2} +e^{-d \tau} \norm{e^{\tau \phi}u }{H^1_\tau}{2}\right)\geq \tau \norm{Q_{\delta, \tau}^{\phi}u}{H^1_\tau}{2},
\end{equation*}
where the $\norm{\cdot}{H^1_\tau}{}$ norm is defined as
$$
\norm{w}{H^1_\tau}{}:=\tau \norm{w}{L^2}{}+\norm{\nabla w}{L^2}{}.
$$

\end{remark}

\begin{remark}
	As in \cite{le2013carleman} we explicitly construct an appropriate weight function $\phi$. It will initially depend only in the variable $x_n$ with $\Sigma=\{x_n=0\}$. In Section~\ref{convexification} we shall use a perturbation argument to allow dependence upon the other variables as well. 
\end{remark}

 \bigskip

	As usual, to prove Theorem~\ref{the carleman ineq thm} we work with the conjugated operator $P_{\phi,\delta}$ defined in our case by the relation:
	\begin{equation}
	\label{def of conjugated operators}
	e^{-\frac{\delta |D_t|^2}{2 \tau}}e^{\tau \phi}P^{\pm}=P^{\pm}_{\phi,\delta}e^{-\frac{\delta |D_t|^2}{2 \tau}} e^{\tau \phi }.		
	\end{equation}

	For a general weight $\phi$ the operators $P^{\pm}_{\phi, \delta}$ do not exist, however in the case where $\phi$ is quadratic in $t$ one can show that we have the following expression for $P^{\pm}_{\phi,\delta}$ (see for instance \cite[Chapter 2]{Hor:97}):
	\begin{align}
	\label{formule for conjugated operators}
		P^{\pm}_{\phi,\delta}&= c_\pm(x)(D_n+i \tau \phi^\prime-\delta \kappa^{\prime \prime}_{t,t}D_t)^2+c_\pm(x)\sum_{1 \leq j,k \leq n-1}b_{jk}(x)(D_j+i \tau \partial_j \phi - \delta \phi^{\prime \prime}_{t, x_j})(D_k+i \tau \partial_k \phi - \delta \phi^{\prime \prime}_{t, x_k}) \nonumber\\&\hspace{4mm}-(D_t+i \tau \phi^\prime _t -\delta \phi^{\prime \prime}_{t,t}D_t)^2
	\end{align}

	This will be used at a later stage where we will \textit{convexify} our initial weight $\phi$ as we know that this is in general a necessary procedure for our Carleman estimate to be used if one wishes to obtain qualitative or quantitative uniqueness results. We shall however initially consider a weight $\phi$ depending solely on the variable $x_n$, and a perturbation argument will be used to allow some convexification. In the case where $\phi$ is independent of $t$ the conjugated operator $e^{\tau \phi}Pe^{-\tau \phi}$ commutes with the Fourier multiplier $e^{-\frac{\delta D^2_t}{2 \tau}}$ and it takes the following particularly simple form:
	\begin{align*}
		P^{\pm}_{\phi} =c_\pm(x) (D_n+i \tau \phi^\prime )^2+ c_{\pm}(x)\bjk D_j D_k-D_t^2
	\end{align*}
	As for the smooth case shown by Tataru we will prove a sub elliptic estimate concerning the conjugated operator, which will act on functions of the form $w=\poids u$ for $u \in \tran$. We therefore have to understand the action of the conjugated operator on the transmission conditions. We use the following expressions:
	\begin{align*}
		w_{\pm} (t,x) &= \lefp \frac{\tau}{2 \pi \delta}\rightp^{\frac{1}{2}} \int_\mathbb{R} e^{-\frac{\tau}{2 \delta} (t-s)^2} e^{\tau \phi } u_{\pm}(s,x) ds \\
		\partial_n w_{\pm}(t,x)&= \lefp \frac{\tau}{2 \pi \delta}\rightp^{\frac{1}{2}} \int_\mathbb{R} e^{-\frac{\tau}{2 \delta} (t-s)^2} e^{\tau \phi } \lefp \tau \phi^\prime _{\pm}u_{\pm} + \partial_n u_{\pm}  \rightp(s,x) ds.
	\end{align*}
	Let us define
	\begin{equation}
		\label{def of Vt}
			V_t:=\R_t \times V.
	\end{equation} 
One has that $u \in \tran$ is equivalent to $w \in \tranphi$ with  the space $\tranphi$ being defined as the space containing functions $w$ such that
	$$
	w=H_-w_-+H_+w_+, \quad w_{\pm}\in C^{\infty}_0 (V_t),
	$$
	satisfying additionally the following modified \textit{transmission conditions}:
	\begin{align}
		\label{transmission cond1}
		{w_-}_{|\Sigma}&={w_+}_{|\Sigma}+\theta_\phi \\
		\label{transmission cond2}
		\cp \lefp D_n {w_{+}}+ i \tau \alpha_+ w_+ \rightp_{|\Sigma}&=\cm \lefp D_n {w_{-}}+ i \tau \alpha_- w_- \rightp_{|\Sigma}+\Theta_\phi,
	\end{align}
where $\theta_\phi= \poids \theta$ and $\Theta_\phi=\poids \Theta$. The following proposition is the main step in the proof of Theorem~\ref{the carleman ineq thm}:
	\begin{prop}
		\label{theorem bis}
			Let $(t_0,x_0) \in \Sigma$.	There exist a suitable weight $\phi$ and $C, \tau_0,  r_0>0$ such that
		\begin{multline*}
			C \bigg( \norm{P^{-}_\phi v_- }{\lrn}{2} +\norm{P^{+}_\phi v_+ }{\lrn}{2}+\tau \norm{H_+  D_t v_+}{\lrn}{2}+\tau \norm{H_-  D_t v_-}{\lrn}{2}  \\
			+\tau \left\vert (D_t v_+ )_{|\Sigma} \right\vert^2_{L^2(\Sigma)}+\tau \left\vert (D_t v_- )_{|\Sigma} \right\vert^2_{L^2(\Sigma)}+T_{\theta, \Theta} \bigg)\\\geq
			\tau^3 \norm{ v}{\lrn}{2}+ \tau \norm{H_+  \nabla v_+}{\lrn}{2}+\tau \norm{H_-  \nabla v_-}{\lrn}{2} \\
			+\tau^3\normsurf{v_+}{\ls}{2}+\tau^3\normsurf{v_-}{\ls}{2}+ \tau \left\vert (\nabla v_+ )_{|\Sigma} \right\vert_{L^2(\Sigma)}^2  +\tau \left\vert (\nabla v_- )_{|\Sigma}\right\vert_{L^2(\Sigma)}^2 ,
		\end{multline*}
		for $v \in \tranphi$ such that $\supp v \subset B((t_0,x_0),r_0)$ and $\tau \geq \tau_0$.

	\end{prop}
	
   Proposition~\ref{theorem bis} provides a sub elliptic estimate for the conjugated operator which contains an admissible error (compared to a standard Carleman estimate) in its left hand side, which we will call $E_t$ for convenience in the sequel:
	\begin{equation}
	\label{def of Etv}
	  	\errt :=\tau \norm{H_+  D_t v_+}{\lrn}{2}+\tau \norm{H_-  D_t v_-}{\lrn}{2}  \\
	+\tau \left\vert (D_t v_+ )_{|\Sigma} \right\vert^2_{L^2(\Sigma)}+\tau \left\vert (D_t v_- )_{|\Sigma} \right\vert^2_{L^2(\Sigma)}.  
	\end{equation}

	\subsection{Proof of Theorem \ref{the carleman ineq thm} from Proposition \ref{theorem bis}}

	One should notice that Theorem \ref{the carleman ineq thm} is not a straightforward consequence of Proposition \ref{theorem bis}. Indeed when one considers $w=\poids u$ the function $w$ is not necessarily compactly supported even though this is the case for $u$ and consequently Proposition~\ref{theorem bis} cannot be applied directly. In particular when we pass from Proposition~\ref{theorem bis} to Theorem~\ref{theorem} the Gaussian  weight localizes close to $\{D_t=0\}$ and that is why $E_t$ is an admissible remainder term. Nevertheless the passage from Proposition~\ref{theorem bis} to Theorem~\ref{theorem} is quite classical (\cite{Tataru:95,Hor:97}). In our context one has additionally to deal with the terms coming from the interface. Let us present a proof here:
	
	\begin{proof} [\textit{Proof that Proposition~\ref{theorem bis} implies Theorem~\ref{theorem}}]
We consider an element $u \in \tran$ satisfying additionally $$\supp{u} \subset B(0,r/4),$$ with $r:=\frac{r_0}{2}$ and $r_0$ fixed by Proposition~\ref{theorem bis}. We define as before $w=Q^\phi_{\delta,\tau}u=e^{-\frac{\delta |D_t|^2}{2 \tau}} e^{\tau \phi }u$ which is \textit{not} compactly supported on the time variable $t$. We take then $\zeta \in C^\infty_0((-r,r);[0,1] ) $ with $\zeta=1$ on $(-r/2,r/2)$. This implies that the function $v:=\zeta(t) w(t,x)\in \tranphi$ satisfies additionally $\supp{ v} \subset B(0,r_0)$ which means that we can apply Proposition~\ref{theorem bis} to it. We neglect the last two surface terms and write the estimate of Proposition~\ref{theorem bis} in a more compact and slightly weaker form to obtain:
\begin{equation*}
    \sum_{\pm} \left( \norm{ H_\pm P^{\pm}_\phi  v }{L^2}{2}+ \tau\norm{H_{\pm}D_t  v_{\pm}}{L^2}{2}+\tau\normsurf{D_t  v_{\pm}}{\ls}{2}+T_{ \zeta \theta, \zeta \Theta} \right) \gtrsim \tau \norm{ v}{H^1_\tau}{2}+\tau^3\normsurf{ v_\pm}{\ls}{2}.
\end{equation*}
Using the fact that $\zeta(t)\leq 1$, the definition of $T_{\theta,\Theta}$ in \eqref{def of T} and the property \eqref{property croissance} of $\poids $ we see that $T_{\zeta \theta, \zeta \Theta} \leq C T_{ \theta,\Theta}$. This yields:
\begin{equation}
    \label{application of thm bis}
     \sum_{\pm} \left( \norm{ H_\pm P^{\pm}_\phi  v }{L^2}{2}+ \tau\norm{H_{\pm}D_t  v_{\pm}}{L^2}{2}+\tau\normsurf{D_t  v_{\pm}}{\ls}{2}+T_{ \theta,  \Theta} \right) \gtrsim \tau \norm{ v}{H^1_\tau}{2}+\tau^3\normsurf{ v_\pm}{\ls}{2}.
\end{equation}
Now we estimate for $w$:
\begin{align}
\label{application of thm bis 2}
    \tau \norm{v}{H^1_\tau}{2}+\tau^3 \normsurf{w_\pm}{\ls}{2}
    &\lesssim \tau \norm{\zeta w}{H^1_\tau}{2}+\tau \norm{ (1-\zeta)w}{H^1_\tau}{2}+\tau^3 \normsurf{\zeta w_\pm}{\ls}{2}+\tau^3\normsurf{(1-\zeta)w_\pm}{\ls}{2} \nonumber\\
    &\lesssim   \sum_{\pm}\left( \norm{ H_\pm P^{\pm}_\phi \zeta w }{L^2}{2}+ \tau\norm{H_{\pm}D_t \zeta w_{\pm}}{L^2}{2}+\tau\normsurf{D_t \zeta w_{\pm}}{\ls}{2}+T_{\theta, \Theta}  \right)\nonumber\\
    &\hspace{4mm}+\tau \norm{ (1-\zeta)w}{H^1_\tau}{2}+\tau^3\normsurf{(1-\zeta)w_\pm}{\ls}{2},
\end{align}
thanks to \eqref{application of thm bis}. For the last two terms we remark that
$$
(1-\zeta )w=(1-\zeta)e^{-\frac{\delta |D_t|^2}{2 \tau}} e^{\tau \phi }u=(1-\zeta)e^{-\frac{\delta |D_t|^2}{2 \tau}} (\check{\chi} e^{\tau \phi } u),
$$
with $\check{\chi}=\check{\chi}(t)\in C^\infty_0((-r/3,r/3))$ with $\check{\chi}=1$ in a neighborhood of $[-r/4,r/4]$ which implies that $\check{\chi} u=u$. Since $1-\zeta$ is supported away from $(-r/2,r/2)$ one can apply Lemma~\ref{lemma 2.4 from ll} which gives:
\begin{equation*}
   \tau \norm{ (1-\zeta)w}{H^1_\tau}{2}+\tau^3\normsurf{(1-\zeta)w_\pm}{\ls}{2} \leq C\tau e^{-c \frac{\tau}{\delta}}\norm{e^{\tau \phi}u}{H^1_\tau}{2}+
   C\tau^3 e^{-c \frac{\tau}{\delta}}\normsurf{e^{\tau \phi}u_\pm}{\ls}{2}.
\end{equation*}
Therefore it remains to estimate the following three terms, appearing in the right hand side of \eqref{application of thm bis 2}:	
\begin{itemize}
    \item We take a function $\tilde{\chi}=\tilde{\chi}(t)$ such that $\tchi=1$ on $\supp(\zeta^{\prime})$ and $\tchi=0$ on $(-r/3,r/3)$ and find for the first one:
     \begin{align}
     \label{estimate bullet 1}
    \sum_{\pm}\norm{ H_\pm P^{\pm}_\phi \zeta w }{L^2}{2} &\lesssim \sum_\pm \left( \norm{H_\pm \zeta   P^{\pm}_\phi   w }{L^2}{2}+\norm{ H_\pm [ P^{\pm}_\phi , \zeta] w }{L^2}{2} \right) \nonumber\\
    &\lesssim \sum_\pm \left(\norm{  H_\pm P^{\pm}_\phi  w }{L^2}{2}+\norm{H_\pm [ P^{\pm}_\phi, \zeta] w }{L^2}{2} \right) \nonumber\\
    &=\sum_\pm \left(\norm{  H_\pm P^{\pm}_\phi  w }{L^2}{2}+\norm{H_\pm [ P^{\pm}_\phi, \zeta] \tchi w }{L^2}{2} \right) \nonumber \\
    &\lesssim \sum_\pm \norm{  H_\pm P^{\pm}_\phi  w }{L^2}{2}+\norm{ \tchi w }{H^1_\tau}{2} \nonumber \\
    &\lesssim \sum_\pm\norm{  H_\pm P^{\pm}_\phi  w }{L^2}{2}+e^{-c \frac{\tau}{\delta}}\norm{e^{\tau \phi}u}{H^1_\tau}{2} 
\end{align}
    
where we used the properties of the support of $\tchi$ and $u$ combined with Lemma~\ref{lemma 2.4 from ll}. 

\item For the second term we have, using the support of $\zeta ^\prime$, as well (to alleviate notation we drop the $\pm$):
\begin{align*}
\tau \norm{D_t \zeta w }{L^2}{2} &\lesssim \tau \norm{ \zeta^{\prime}w }{L^2}{2}+\tau \norm{D_t w }{L^2}{2} \\ &\lesssim \tau e^{-c \frac{\tau}{\delta}}\norm{e^{\tau \phi}u}{L^2}{2}+\tau \norm{D_t w }{L^2}{2}.
\end{align*}
To estimate $\tau \norm{D_t w }{L^2}{2}$ we work on the Fourier domain (with respect to the time variable $t$) and distinguish between the frequencies smaller or bigger than $\sigma \tau$ for an arbitrary $\sigma>0$ (as in \cite[Section 5.2]{Laurent_2018}). One has:
\begin{align*}
    \norm{D_t w}{L^2}{}&\leq \norm{1_{|D_t|\leq \sigma \tau}D_t w}{L^2}{}+\norm{1_{|D_t|\geq \sigma \tau}D_t w}{L^2}{}\\
    &=\norm{1_{|D_t|\leq \sigma \tau}D_t w}{L^2}{}+\norm{1_{|D_t|\geq \sigma \tau} D_t e^{-\frac{\delta |D_t|^2}{2 \tau}} e^{\tau \phi }u}{}{}\\
    &\leq \norm{1_{|D_t|\leq \sigma \tau}D_t w}{L^2}{}+  \underset{\xi_t \geq \sigma \tau}{\textnormal{max}} \left( \xi_t e^{-\frac{\delta \xi^2_t}{2 \tau}} \right)\norm{e^{\tau \phi}u}{L^2}{}.
\end{align*}
We see that if $\tau \geq \frac{1}{\sigma^2 \delta}$ the function $s\mapsto se^{-\frac{\delta s^2}{2 \tau}}$ is decreasing on the interval $[\sigma \tau, +\infty)$. We obtain therefore for $\tau \geq \textnormal{max}(\tau_1,\frac{1}{\sigma^2 \delta})$ the estimate:
\begin{equation}
\label{estimate bullet 2}
\tau \norm{D_t w }{L^2}{2} \lesssim \sigma^2 \tau^3 \norm{w}{L^2}{2}+\sigma^2 \tau^3 e^{-\tau  \sigma^2 \delta} \norm{e^{\tau \phi}u}{L^2}{2}.
\end{equation}

Above and in the sequel the hidden constant will be independent of $\sigma$.

\item For the third term $\tau \normsurf{D_t \zeta w_\pm}{\ls}{2}$ in \eqref{application of thm bis 2} one can proceed exactly as above to find:
\begin{equation}
\label{estimate bullet 3}
    \tau \normsurf{D_t \zeta w}{\ls}{2} \lesssim \tau e^{-c \frac{\tau}{\delta}}\normsurf{e^{\tau \phi}u}{\ls}{2}+\sigma^2 \tau^3 \normsurf{w}{\ls}{2}+\sigma^2 \tau^3 e^{-\tau  \sigma^2 \delta} \normsurf{e^{\tau \phi}u}{\ls}{2}.
\end{equation}

\end{itemize}
We inject estimates \eqref{estimate bullet 1}, \eqref{estimate bullet 2} and \eqref{estimate bullet 3} in the right hand side of \eqref{application of thm bis 2} to obtain:
\begin{align*}
    \tau \norm{w}{H^1_\tau}{2}+\tau^3 \normsurf{w}{\ls}{2}
    &\lesssim \norm{ P_\phi  w }{L^2}{2}+\left(\tau e^{-c \frac{\tau}{\delta}}+\sigma^2 \tau^3 e^{-\tau  \sigma^2 \delta} \right)\norm{e^{\tau \phi}u}{H^1_\tau}{2}\\
   &\hspace{4mm}+ \left(\tau^3 e^{-c \frac{\tau}{\delta}}+\sigma^2 \tau^3 e^{-\tau  \sigma^2 \delta}\right)\normsurf{e^{\tau \phi}u}{\ls}{2} \\
   &\hspace{4mm}+\sigma^2 \tau^3 \norm{w}{L^2}{2}+\sigma^2 \tau^3 \normsurf{w}{\ls}{2}+T_{\theta, \Theta}.
\end{align*}
We now choose $\sigma$ sufficiently small to absorb the last two terms above in the left hand side of our estimate. Then there exists $d>0$ such that for $\tau \geq \max(\tau_0, \frac{1}{\sigma^2 \delta})$:
\begin{align*}
     \sum_{\pm}\left( \tau \norm{w_\pm}{H^1_\tau}{2}+\tau^3 \normsurf{w_\pm}{\ls}{2}\right)  &\lesssim \sum_\pm \left( \norm{H_\pm P^{\pm}_\phi  w_\pm }{L^2}{2}+e^{-d \tau}\left(\norm{e^{\tau \phi}u_\pm}{H^1_\tau}{2}+\normsurf{e^{\tau \phi}u_\pm}{\ls}{2} \right)+T_{\theta, \Theta}\right)  \\
     & \lesssim \sum_\pm \left( \norm{H_\pm P^{\pm}_\phi  w_\pm }{L^2}{2}+e^{-d \tau}\norm{e^{\tau \phi}u_\pm}{H^1_\tau}{2}+T_{\theta, \Theta}\right) \\
     & \lesssim \sum_\pm \left( \norm{H_\pm \poids P^{\pm}  u_\pm }{L^2}{2}+e^{-d \tau}\norm{e^{\tau \phi}u_\pm}{H^1_\tau}{2}+T_{\theta, \Theta}\right),
\end{align*}
where we have used the trace estimate $\normsurf{e^{\tau \phi}u}{\ls}{2} \lesssim \norm{e^{\tau \phi}u}{H^1}{2}$ as well as the definition of the conjugated operator $P_\phi$. This concludes the proof of Theorem~\ref{theorem} from Proposition~\ref{theorem bis}.
\end{proof}

\section{Proof of Proposition \ref{theorem bis} for a toy model}
\label{proof of toy model}

The goal of this section is to prove the subelliptic estimate of Proposition \ref{theorem bis} in the particular situation where the coefficient $c$ is piecewise constant. This case works as a sketch of proof since it is technically simpler but at the same time it allows to understand the core of the arguments that will be used for the proof of the general case in Section \ref{proof for the general case}.

\bigskip

\noindent	\textbf{Notations:} Before going further let us fix some useful notations that will be used in the sequel. We write $\R^{n+1}_+=\{x_n >0\} \times \R^n$ with the analogous definition for $\R^{n+1}_-$. In particular $\Rp$ will refer to $\{x_n>0\}$ and $\Rm$ to $\{x_n<0\}$.  We note $\big(a,b\big)$ the inner product in $\lrn$ and $x^\prime=(x_1,...,x_{n-1})$. The inner product on $\lrpn, \lrmn, \Sigma$ will de denoted by $(\cdot,\cdot)_+, (\cdot,\cdot)_-, (\cdot,\cdot)_\Sigma$ respectively. When we consider norms on $\Sigma$ the argument will automatically be considered to be restricted in $\Sigma$, even though we shall not always write $|\Sigma$ to simplify our notation. We will simply write $\hat{v}$ for the $\textit{partial}$ Fourier transform of $v$ in the variables $(t,x^\prime)$, whose dual variables are $(\xi_t, \xip)$. We recall that the space of functions $\tranphi$ and the small neighborhood $V$ have been defined in Section~\ref{local setting}. To alleviate notation we denote 
$$
\mathcal{W}_{\phi}:=\mathcal{W}_{\phi}^{0,0},
$$
for the case of homogeneous transmission conditions.

	We recall that we work in the setting introduced in Section \ref{local setting}.
 We suppose additionally only in this section that the coefficient $c$ is piecewise constant. We write then  $c=H_-c_-+H_+c_+$ with $c_\pm>0$ constants and we consider homogeneous transmission conditions $\theta=\Theta=0$. This allows to write $v_-(0)=v_+(0)=v(0)$ for $v \in \mathcal{W}_{\phi}$ and it implies as well that $\cnop v \in L^2(V_t)$.

	\bigskip
	
\subsection{Factorization and first estimates }
\label{section factorization microlocal and estimates}

	\textit{In the sequel when we use the notation $\lesssim$ or $\gtrsim$ the implicit constant will depend on the coefficients of $\cnop$ (here $c_{\pm}$) and on the coefficients $\alpha_{\pm}, \beta$ of our weight function $\phi$. The constants denoted by $C$ will depend on the same variables and they can be different from one line to another.}
	
	\bigskip

	\bigskip
	
	In the elliptic case of \cite{le2013carleman} the authors use a factorization argument which takes advantage of the fact that $-\Delta$ is positive and therefore one can define its square root. In our proof of Proposition~\ref{theorem bis} we use and extend this idea of factorization, however as we no longer have the positivity property this factorization is not always possible. Before giving the details let us observe that we can identify the operator $\cnop$ with its symbol in the tangential variables $(t,x^\prime)$. Indeed, using Plancherel's theorem and denoting by $\mathcal{F}_{t,x^\prime}$ the partial Fourier transform in the variables $(t,x^\prime)$ we have:
	\begin{align*}
	(2 \pi)^{n/2}	\norm{\cnop v}{\lrn}{}&=\norm{ \mathcal{F}_{t,x^\prime}\cnop v }{\lrn}{} \\
		&=\norm{H_+\left(c_+ (D_n+i \tau \phi^\prime )^2+ c_{+}|\xip|^2-\xi_t^2\right) \mathcal{F}_{t \rightarrow \xi_t ,x^\prime \rightarrow \xip}v_+}{\lrn}{}\\
		&\hspace{4mm}  + \norm{H_-\left(c_- (D_n+i \tau \phi^\prime )^2+ c_{-}|\xip|^2-\xi_t^2\right) \mathcal{F}_{t \rightarrow \xi_t ,x^\prime \rightarrow \xip}v_-}{\lrn}{}.
	\end{align*}
	Here we used the fact that the coefficients of $P$ (and thus of $\cnop$ since $\phi= \phi(x_n)$) are constant. In the general case, although this identification is no longer valid, symbolic calculus will allow to exploit the core of the arguments carried out below. We write:
	$$
	\cnop^{\pm}=c_{\pm}\lefp (D_n+ i \tau \phi^\prime )^2+|\xip|^2-\frac{1}{c_{\pm}}|\xi_t|^2 \rightp.
	$$
For $\epsilon > 0 $ small to be chosen, we distinguish the following regions of the tangential frequency space $\R^n \ni (\xip, \xi_t)$:
	
	\begin{enumerate}
		
		\item $\mathcal{E^{\pm}_{\epsilon}}: $ \boxed  { |\xip|^2 - \invc_{\pm} |\xi_t|^2 \geq \epsilon (|\xip|^2+ |\xi_t|^2) }. This is the elliptic region. Here we can factorize in precisely the same way as in the elliptic case \cite{le2013carleman} with the same estimates for the first order factors. For $(\xi_t,\xip)\in\mathcal{E^{-}_{\epsilon}}\cap \mathcal{E^{+}_{\epsilon}} $ we can proceed exactly as in the elliptic case.
		
		\item $\mathcal{GH^{\pm}_{\epsilon}}: $ \boxed{|\xip|^2 -\invc_{\pm} |\xi_t|^2 < \epsilon (|\xip|^2+ |\xi_t|^2) }. This is the union of the hyperbolic and glancing regions  (see for instance \cite[Chapter 23.2]{Hoermander:V3} or \cite{galkowski2020control}). The important fact is that here we have for sufficiently small $\epsilon$ that $|\xi_t |\gtrsim |\xip |$. Since $\norm{D_t v}{}{}=\norm{\xi_t v}{}{}$ and $|D_t v|=|\xi_t v|$ are admissible remainder terms (in view of the statement of Proposition \ref{theorem bis}) we obtain directly a useful estimate on all tangential derivatives. It thus only remains to obtain an estimate on $\norm{D_n v}{}{}$ and $|D_n v|$.
		
	\end{enumerate}

	We have thus considered two regions adapted to the geometry in each side of the interface (that is  $\{x_n<0\}$ and $\{x_n>0\}$). This gives us four regions which cover all of the tangent dual space. Note that making an assumption on the \textit{sign} of the jump could reduce the number of regions one has to deal with but we do not make such an assumption.

The following simple Lemma will be used in multiple occasions in the sequel:
	
	\begin{lem}
		\label{one of boundary derivatives}
		One has for $v \in \mathcal{W}_{\phi}$ and all $\tau>0$:
		$$
		\tau |D_n\hat{v}_{\pm}(0)|_{}^2 \lesssim \tau |D_n \hat{v}_{\mp}(0)|_{}^2+ \tau^3 |\hat{v}(0)|_{}^2
		$$

	\end{lem}
	This means that if we have one of the two terms $\tau \normsurf{D_n\hat{v}_+}{L^2(\Sigma)}{2}$ or $\tau \normsurf{D_n \hat{v}_-}{L^2(\Sigma)}{2}$ we automatically have the other one too, modulo the surface term $\tau^3 |\hat{v}|^2$.
	
	\begin{proof}
		This is a result of the transmission conditions \eqref{transmission cond1} and \eqref{transmission cond2} and of the fact that partial Fourier transform commutes with restriction on $\Sigma=\{x_n=0\}$. The first transmission condition allows to write $\hat{v}_-(0)=\hat{v}_+(0)=\hat{v}(0)$
		and \eqref{transmission cond2} reads
		$$
		D_n\hat{v}_-(0)=\frac{\cp}{\cm} (D_n\hat{v}_+ + i \tau \alpha_+ \hat{v}_+ (0))- i \tau \alpha_- \hat{v}_-(0),
		$$
and therefore 
$$
|c_-D_n\hat{v}_-(0)-c_+D_n\hat{v}_+(0)| \lesssim \tau |\hat{v}(0)|,
$$
from which the Lemma follows.
	\end{proof}

	Consequently, if for instance one has $\norm{\cnop v}{\lrn}{2}+\errt \gtrsim \tau \normsurf{D_n v_-}{L^2(\Sigma)}{2}+ \tau^3 |v(0)|_{L^2(\Sigma)}^2$  then thanks to Lemma~\ref{one of boundary derivatives} we have:
	$$
\tau |D_nv_{+}(0)|_{L^2(\Sigma)}^2 \lesssim \tau |D_n v_{-}(0)|_{L^2(\Sigma)}^2+ \tau^3 |v(0)|_{L^2(\Sigma)}^2 \lesssim \norm{\cnop v}{\lrn}{2}+\errt.
	$$

	As we aim at proving a Carleman estimate (which involves a large parameter $\tau$) we expect that regions depending on $\tau$ arise. The following lemma furnishes such a region which is rather favorable:
	
	\begin{lem}
		\label{regime favorable} For all $\tau, c_0>0$,  $v\in \mathcal{W}_\phi $ and $Y \subset \{|\xi_t| \geq c_0 ( \tau + |\xip|)\}$ one has
		\begin{multline*}
			\tau \norm{ \widehat{D_t v_+}}{L^2(\Rp \times Y)}{2}+\tau \norm{ \widehat{D_t v_-}}{L^2(\Rm \times Y)}{2} + \tau  \normsurf{\widehat{D_t v_+}(0)}{L^2(\Sigma\cap Y)}{2}+\tau  \normsurf{\widehat{D_t v_-}(0)}{L^2(\Sigma \cap Y)}{2}  \\ \gtrsim \tau^3 \norm{\widehat{v}}{L^2(\R \times Y)}{2}+\tau^3 |\widehat{v (0)} |_{L^2(\Sigma \cap Y)}^2+\tau\left\vert\widehat{ \nabla_{x^\prime} v }(0) \right\vert_{L^2(\Sigma \cap Y)}^2   \\
			+\tau \norm{\widehat{\nabla_{x^\prime}v_{-}}}{L^2(\Rm \times Y)}{2}+\tau \norm{\widehat{\nabla_{x^\prime}v_{+}}}{L^2(\Rp \times Y)}{2}.
		\end{multline*}
	\end{lem}

	\begin{proof}
	We simply write
		\begin{align*}
			\tau \norm{\widehat{D_t v_{\pm}}}{L^2(\R \times Y)}{2} &= \tau \norm{\xi_t \hat{v}_{\pm}}{L^2(\R \times Y)}{2} \gtrsim \tau \norm{(\tau + |\xip|)\hat{v}_{\pm})}{L^2(\R \times Y)}{2} \\
			&\gtrsim \tau^3 \norm{\hat{v}_{\pm}}{L^2(\R \times Y)}{2}+ \tau \norm{\widehat{\nabla_{x^\prime}v}}{L^2(\R \times Y)}{2}
		\end{align*}
	And:
		\begin{align*}
			\tau \normsurf{D_t v}{L^2(\Sigma \cap Y)}{2} &= \tau \normsurf{\xi_t \hat{v}}{L^2(\Sigma \cap Y)}{2} \gtrsim \tau \normsurf{(\tau + |\xip|)\hat{v})}{L^2(\Sigma \cap Y)}{2} \\
			&\gtrsim \tau^3 \normsurf{\hat{v}}{L^2(\Sigma \cap Y)}{2}+ \tau \normsurf{\widehat{\nabla_{x^\prime}v}}{L^2(\Sigma \cap Y)}{2}.
		\end{align*}

	\end{proof}

 The following lemma allows to obtain the volume derivatives modulo the reminder of the terms of the right hand-side of Proposition~\ref{theorem bis}.

	\begin{lem}
		\label{derivatives Dn}
		One has for $v \in \mathcal{W}_\phi$ and $\tau\geq \tau_0$:
		\begin{multline*}
			\norm{\cnop v}{L^2(\R^{n+1})}{2} +\errt +\bigg(\tau^3\norm{v}{\lrpn}{2}+\tau^2\normsurf{v(0)}{L^2(\Sigma)}{2}+  \normsurf{D_n v_{+}(0)}{L^2(\Sigma)}{2}+\normsurf{D_n v_{-}(0)}{L^2(\Sigma)}{2} \bigg)\\
			\gtrsim \tau \norm{\nabla v_{-}}{\lrmn}{2}+	\tau \norm{\nabla v_{+}}{\lrpn}{2}.
		\end{multline*}
\end{lem}
	
	\begin{remark}
  Lemma~\ref{derivatives Dn} allows to forget the volume norm of the derivatives in the proof of our proposition. From now on we will try to obtain all the other terms of Proposition~\ref{theorem bis}. Notice also that there is no restriction on the frequency support of $v$.
 
	\end{remark}
	
	\begin{proof}
	We start with the positive half-line. One has the following elementary inequality:
		
		$$
		\norm{\cnop^+ v_+}{L^2(\R^{n+1}_+)}{2}+\tau^2\norm{v_+}{L^2(\R^{n+1}_+)}{2}\geq 2 \tau \operatorname{Re}(P^+_{\phi}v_+,H_+v_+).
		$$
Now we can integrate by parts. Indeed, recall that:
		\begin{align*}
			P^+_{\phi} &=c^{+}(D_n^2+ \tau \beta +2 i \tau \phi^\prime D_n-\tau^2|\phi^\prime | ^2)
			+c^{+}|D^{'}|^2-D_t^2
		\end{align*}
		and that the definition of the weight function $\phi$  gives $\phi^\prime (x_n)=\alpha_{\pm}+\beta x_n $, therefore:
			\begin{align}
			\label{estim volume derivatives toy model }
			2\tau \operatorname{Re} (P^+_{\phi}v,H_+v)&= c_+\Big(2\tau \norm{D_n v_+}{\lrpn}{2}+2 \tau \int_{\R^n}\operatorname{Re}(\partial_{x_n}v_+(0) \Bar{v}_+(0))dx^\prime dt \nonumber\\& \hspace{4mm} +2\tau^2 \beta \norm{v_+}{\lrpn}{2} -2\tau^2 \alpha_+|v_+(0)|^2 -2\tau^2 \beta \norm{v_+}{\lrpn}{2}  \nonumber \\
			&\hspace{4mm}-2 \tau^3\norm{\phi^\prime  v_+}{\lrpn}{2}+2 \tau \norm{D^{'} v_+}{\lrpn}{2} \Big)-2 \tau \norm{D_t v_+}{\lrpn}{2}.
		\end{align}
		As one has 
		$$
		\left \vert 2 \tau \int_{\R^n}\operatorname{Re}(\partial_{x_n}v_+(0) \Bar{v}_+(0))dx^\prime dt \right \vert \lesssim \normsurf{D_n v_+(0)}{L^2(\Sigma)}{2}+ \tau^2 \normsurf{v_+(0)}{L^2(\Sigma)}{2},
		$$
    \eqref{estim volume derivatives toy model } can be rewritten as
		$$
		2\tau \operatorname{Re} (P^+_{\phi}v,H_+v))_{L^2(\mathbb{R}_{+}^{n+1})}=C \tau \norm{\nabla v{_+}}{\lrpn}{2}+R,
		$$
		with
		$$
		|R|\lesssim \tau^3\norm{v_+}{\lrpn}{2}+\tau \norm{D_t v_+}{\lrpn}{2}+\tau^2\normsurf{v_+(0)}{L^2(\Sigma)}{2}+  \normsurf{D_n v_+(0)}{L^2(\Sigma)}{2}.
		$$
		which implies the desired inequality for $v_+$.
		Since the proof above is insensitive to the sign of the boundary terms coming from the integration by parts in \eqref{estim volume derivatives toy model } we also obtain the desired inequality for the negative half-line.

	\end{proof}

	The following simple calculation will be at the core of the one dimensional estimates that will be used in the sequel, with $s=x_n$:
	\begin{lem}
		\label{calcul clef}
		One has for $\gamma>0$, $\lambda \in \mathbb{C}$ and $v \in C_c^1(\R)$:
		\begin{align}
		 \label{equality clef positif}
			\norm{(D_s+i \gamma s + \lambda)v}{\lrp}{2}= 
			& \norm{D_sv+\operatorname{Re} \lambda v }{\lrp}{2}+\norm{(\gamma s +\operatorname{Im} \lambda )v}{\lrp}{2}  \\ \nonumber
			&+\gamma \norm{v}{\lrp}{2} +\operatorname{Im} \lambda |v(0)|^2,
		\end{align}
		and
		\begin{align}
		\label{equality clef negatif}
			\norm{(D_s+i \gamma s + \lambda)v}{\lrm}{2}= 
			& \norm{D_sv+\operatorname{Re} \lambda v }{\lrm}{2}+\norm{(\gamma s +\operatorname{Im} \lambda )v}{\lrm}{2}  \\ \nonumber
			&+\gamma \norm{v}{\lrm}{2} -\operatorname{Im} \lambda |v(0)|^2.
		\end{align}
	\end{lem}

	In particular one can deduce the following downgraded estimates:
	\begin{align}
		\label{ineg 1 positif}
		& \norm{(D_s+i \gamma s + \lambda)v}{\lrp}{2}\geq \norm{(\gamma s +\operatorname{Im} \lambda )v}{\lrp}{2}+\operatorname{Im} \lambda |v(0)|^2,\\
		\vspace{8mm}
		\label{ineg 2 negatif}
		&\norm{(D_s+i \gamma s + \lambda)v}{\lrm}{2} \geq \gamma \norm{v}{\lrm}{2}+\norm{(\gamma s +\operatorname{Im} \lambda )v}{\lrm}{2}-\operatorname{Im} \lambda |v(0)|^2 .
	\end{align}

\begin{proof}[Proof of Lemma \ref{calcul clef}]
We develop:
\begin{align*}
    \norm{(D_s+\rel +i \gamma s + i \iml )v}{\lrp}{2}&=\norm{(D_s+\rel)v}{\lrp}{2}+\norm{(\gamma s +\iml )v}{\lrp}{2} \\
    &\hspace{4mm}+2\operatorname{Re} (H_+(D_s+\rel)v, i \gamma s v+i \iml v).
    \end{align*}
We have $\operatorname{Re} (H_+\rel v, i \gamma s v+i \iml v)=0$ and then integrate by parts:

$$
(H_+D_sv,i \iml v)=\iml|v(0)|^2-\overline{(H_+D_sv,i \iml v)},
$$
and

$$
(H_+D_s v, i \gamma s v)=\gamma \norm{v}{\lrp}{2}-\overline{(H_+D_s v, i \gamma s v)},
$$
which allows to explicitly obtain the real parts and get the stated equality.
For the proof of \eqref{equality clef negatif} observe that the boundary term comes out with a negative sign.
\end{proof}

We first deal with the case $\tau \gg |\xip|+|\xi_t|$. Here we can use the ellipticity of 
$$
(D_n+ i \tau \phip)^2
$$
as an 1D operator, and then a perturbation argument. Recall that $V_t$ has been defined in~\eqref{def of Vt}.

\begin{lem}
\label{tau grand toy}
There exists $\sigma_0, \tau_0>0$ and $\mathcal{V} \Subset V_t$ such that  for all $v \in \mathcal{W}_\phi$ with $\supp v \subset \mathcal{V}$ and 
\begin{equation}
\label{support condition tau grand toy}
Y\subset \{ \tau \geq  \frac{1}{\sigma}(|\xi_t|+|\xip|) \} 
\end{equation}
one has for all $\sigma \leq \sigma_0$ and $\tau \geq \tau_0:$
\begin{multline*}
			\norm{\widehat{\cnop v}}{L^2(\R \times Y)}{2}+\tau \norm{ \widehat{D_t v_+}}{L^2(\Rp \times Y)}{2}+\tau \norm{ \widehat{D_t v_-}}{L^2(\Rm \times Y)}{2} 
			\\+ \tau  \normsurf{\widehat{D_t v_+(0)}}{L^2(\Sigma\cap Y)}{2}+\tau  \normsurf{\widehat{D_t v_-(0)}}{L^2(\Sigma \cap Y)}{2}  \\  \gtrsim \tau^3 \norm{ \widehat {v}}{L^2(\R \times Y)}{2}
			 +\tau^3 \left\vert \widehat{v (0)} \right\vert_{L^2(\Sigma \cap Y)}^2 +\tau \left\vert\widehat{ (\nabla v_+ )}(0) \right\vert_{L^2(\Sigma \cap Y)}^2+\tau \left\vert \widehat{ (\nabla v_- )}(0) \right\vert_{L^2(\Sigma \cap Y)}^2.
		\end{multline*}
	\end{lem}

\begin{proof}
We identify, with a slight abuse of notation, the operator with its symbol. One has:
	
	$$
	\invcp P_\phi ^+=(D_n+ i \tau \phi^\prime )^2+|\xip|^2-\invcp \xi_t^2:=A_+ +R_+
	$$
 	where $A_+:=(D_n+ i \tau \phi^\prime )^2$ is elliptic as a 1D operator and $R_+:=|D^{'}|^2-\invcp D_t^2$. One has with $w=( D_n+ i \tau \phi^\prime )\widehat{v}_+=(D_n+ i \tau \alpha_+ + i \tau \beta x_n)\widehat{v}_+$, using \eqref{ineg 1 positif}:
	\begin{align*}
		\norm{A_+\widehat{v}_+}{L^2(\R_+ \times Y)}{2}&=\norm{(D_n+ i \tau \alpha_+ + i \tau \beta x_n)w}{L^2(\R_+ \times Y)}{2} \\
		&\geq \norm{(\tau \beta x_n + \tau \alpha_+ )w}{L^2(\R_+ \times Y)}{2} + \alpha_+ \tau \normsurf{w}{L^2(\Sigma \cap Y)}{2}\gtrsim \tau^2 \norm{w}{L^2(\R_+ \times Y)}{2}+\normsurf{w}{L^2(\Sigma \cap Y)}{2}\\ &\gtrsim \tau^2\left(\norm{D_n \widehat{v}_+}{L^2(\R_+ \times Y)}{2}+ \norm{ (\tau \alpha_+ +  \tau \beta x_n)\widehat{v}_+}{L^2(\R_+ \times Y)}{2}+ \tau \alpha_+ \normsurf{\widehat{v}_+}{L^2(\Sigma \cap Y)}{2}\right)\\
		& \hspace{4mm}+ \alpha_+ \tau \normsurf{(D_n+i \tau \alpha)\widehat{v}_+}{L^2(\Sigma \cap Y)}{2}.
	\end{align*}
In particular we find that $\norm{A_+\widehat{v}_+}{L^2(\R_+ \times Y)}{2} \gtrsim \tau ^4 \norm{\widehat{v}_+}{L^2(\R_+ \times Y)}{2} + \tau^3 \normsurf{\widehat{v}_+}{L^2(\Sigma \cap Y)}{2}+\tau \normsurf{(D_n+i \tau \alpha)\widehat{v}_+}{L^2(\Sigma \cap Y)}{2}$, and then using $\tau ^3\normsurf{D_n \widehat{v}_+}{L^2(\Sigma \cap Y)}{2} \lesssim \tau^3 \normsurf{\widehat{v}_+}{L^2(\Sigma \cap Y)}{2}+ \tau \normsurf{(D_n+i \tau \alpha)\widehat{v}_+}{L^2(\Sigma \cap Y)}{2}  $ we get 
\begin{equation*}
   \norm{A\widehat{v}_+}{L^2(\R_+ \times Y)}{2} \gtrsim  \tau ^4 \norm{\widehat{v}_+}{L^2(\R_+ \times Y)}{2} + \tau^3 \normsurf{\widehat{v}_+}{L^2(\Sigma \cap Y)}{2}+ \tau \normsurf{D_n \widehat{v}_+}{L^2(\Sigma \cap Y)}{2}.
\end{equation*}
Lemma \ref{one of boundary derivatives} then implies:
\begin{equation}
    \label{estim for A+}
     \norm{A_+\widehat{v}_+}{L^2(\R_+ \times Y)}{2} \gtrsim  \tau ^4 \norm{\widehat{v}_+}{L^2(\R_+ \times Y)}{2} + \tau^3 \normsurf{\widehat{v}_+}{L^2(\Sigma \cap Y)}{2}+ \tau \normsurf{D_n \widehat{v}_+}{L^2(\Sigma \cap Y)}{2}+\tau \normsurf{D_n \widehat{v}_-}{L^2(\Sigma \cap Y)}{2}. 
\end{equation}
Recalling that $A_+=P_\phi - R_+$ we now explain how  $\norm{R_+ \widehat{v}_+}{L^2(\R_+ \times Y)}{}$ can be absorbed as an error term. We estimate as follows:
\begin{align*}
\norm{  \cnp \hat{v}_+}{L^2(\R_+ \times Y)}{2} &\geq C\norm{A_+ \widehat{v}_+}{L^2(\R_+ \times Y)}{2}- \norm{R_+ \widehat{v}_+}{L^2(\R_+ \times Y)}{2} \\
&\geq  C\left(\tau ^4 \norm{\widehat{v}_+}{L^2(\Rp \times Y)}{2} + \tau^3 \normsurf{\widehat{v}_+}{L^2(\Sigma \cap Y)}{2}+\tau \normsurf{D_n\widehat{v}_+}{L^2(\Sigma \cap Y)}{2}\right)- \norm{R_+ \widehat{v}_+}{L^2(\R_+ \times Y)}{2} \\
&=\left((C\tau^4-R^*R)\widehat{v}_+,\widehat{v}_+\right)_+ +C\tau^3 \normsurf{\widehat{v}_+}{L^2(\Sigma \cap Y)}{2}+C\tau \normsurf{D_n\widehat{v}_+}{L^2(\Sigma \cap Y)}{2}.
\end{align*}
Using \eqref{support condition tau grand toy} we have
$$
C\tau^4-R^*R=C\tau^4-(|\xip|^2-c_+\xi^2_t)^2\geq \left(C-\sigma^4(1+\invcp)^2\right)\tau^4.
$$
Then, taking $\displaystyle \sigma \leq \left( \frac{C}{2(1+\invcp)^2}\right)^{1/4}$ we deduce
\begin{equation*}
  \norm{ \cnp \hat{  v}_+}{L^2(\R_+ \times Y)}{2} \gtrsim \tau ^4 \norm{\widehat{v}_+}{L^2(\Rp \times Y)}{2} + \tau^3 \normsurf{\widehat{v}_+}{L^2(\Sigma \cap Y)}{2}+ \tau \normsurf{D_n\widehat{v}_+}{L^2(\Sigma \cap Y)}{2}.
\end{equation*}
Using the fact that $\tau \gtrsim |\xip|$ together with the transmission condition \eqref{transmission cond1} we obtain the tangential terms :
\begin{align*}
    \tau^3 \normsurf{\widehat{v}_+}{L^2(\Sigma \cap Y)}{2}\gtrsim \tau \normsurf{(|\xip|^2+|\xi_t|^2)^{1/2} \hat{v}}{L^2(\Sigma \cap Y)}{2}=\tau \normsurf{ \widehat{\nabla v}}{L^2(\Sigma \cap Y)}{2}.
\end{align*}
This yields
\begin{equation*}
  \norm{ \cnp \hat{  v}_+}{L^2(\R_+ \times Y)}{2} \gtrsim \tau ^4 \norm{\widehat{v}_+}{L^2(\Rp \times Y)}{2} + \tau^3 \normsurf{\hat{v}}{L^2(\Sigma \cap Y)}{2}+\tau \normsurf{\widehat{\nabla_{x^\prime}v} }{L^2(\Sigma \cap Y)}{2}+\tau \normsurf{D_n\widehat{v}_+}{L^2(\Sigma \cap Y)}{2}.
\end{equation*}
and finally thanks to Lemma~\ref{one of boundary derivatives}:
\begin{equation}
\label{estimate for A+ region tau grand toy}
  \norm{ \cnp \hat{  v}_+}{L^2(\R_+ \times Y)}{2}+\norm{\widehat{v}_+}{L^2(\R_+ \times Y)}{2} \gtrsim \tau ^4 \norm{\widehat{v}_+}{L^2(\Rp \times Y)}{2} + \tau^3 \normsurf{\hat{v}}{L^2(\Sigma \cap Y)}{2}+\tau \normsurf{\widehat{\nabla v_-}}{L^2(\Sigma \cap Y)}{2}+\normsurf{\nabla \widehat{v}_+}{L^2(\Sigma \cap Y)}{2}.
\end{equation}
The only missing term is the volume norm on the negative half-line. Here one needs to use $\cnm$. We write 
   $$
	\invcm P_\phi ^-=(D_n+ i \tau \phi^\prime )^2+|\xip|^2 -\invcm \xi_t^2=A_- +R_-, \quad A_-:=(D_n+ i \tau \phi^\prime )^2
	$$
In the region under consideration $\norm{R_- \hat{v}}{L^2(\R_- \times Y)}{}$ is a perturbation of $\norm{A_- \hat{v}}{L^2(\R_- \times Y)}{}$. After repeating the same steps as for the positive half-line one finds:
\begin{equation}
\label{estimate for A- region tau grand toy}
  \norm{\cnm  \hat{ v}_-}{L^2(\R_- \times Y)}{2}+ \tau^3 \normsurf{\hat{v}}{L^2(\Sigma \cap Y)}{2}+\tau \normsurf{\widehat{\nabla v_-}}{L^2(\Sigma \cap Y)}{2} \gtrsim \tau ^4 \norm{\hat{v}_-}{L^2(\R_- \times Y)}{2} .
\end{equation}
To finish the proof of the Lemma one can simply multiply \eqref{estimate for A+ region tau grand toy} by a sufficiently large constant and add it to \eqref{estimate for A- region tau grand toy}.
\end{proof}

We now give a Lemma for the region $\tau \lesssim |\xip|+|\xi_t|.$

\begin{lem}
\label{region tau petit non elliptic toy}
For all $c_1, c_2 >0 $ there exists $\mathcal{V} \Subset V_t $ and $\tau_0>0$ such that for all  $v \in \mathcal{W}_\phi$ with $\supp v \subset \mathcal{V}$ and $$Y\subset \{ \tau \leq c_1 ( |\xi_t|+|\xip|) \} \cap \{|\xi_t| \geq c_2 |\xip| \} $$ one has:
\begin{multline*}
			\norm{\widehat{\cnop v}}{L^2(\R \times Y)}{2}+\tau \norm{ \widehat{D_t v_+}}{L^2(\Rp \times Y)}{2}+\tau \norm{ \widehat{D_t v_-}}{L^2(\Rm \times Y)}{2} 
			\\+ \tau  \normsurf{\widehat{D_t v_+(0)}}{L^2(\Sigma\cap Y)}{2}+\tau  \normsurf{\widehat{D_t v_-(0)}}{L^2(\Sigma \cap Y)}{2}  \\  \gtrsim \tau^3 \norm{ \widehat {v}}{L^2(\R \times Y)}{2}
			 +\tau^3 \left\vert \widehat{v (0)} \right\vert_{L^2(\Sigma \cap Y)}^2 +\tau \left\vert\widehat{ (\nabla v_+ )}(0) \right\vert_{L^2(\Sigma \cap Y)}^2+\tau \left\vert \widehat{ (\nabla v_- )}(0) \right\vert_{L^2(\Sigma \cap Y)}^2.
		\end{multline*}
		for $\tau\geq \tau_0$.
\end{lem}

\begin{proof}
Observe that in such a region one has in particular
$$
|\xi_t| \gtrsim |\xip|+ \tau,
$$
and consequently we are in the regime of Lemma~\ref{regime favorable}. This implies:
\begin{multline}
		\label{first estimate G+ }
			\norm{\widehat{\cnop^+ v_+}}{L^2(\Rp \times Y)}{2}+\tau \norm{ \widehat{D_t v_+}}{L^2(\Rp \times Y)}{2} + \tau  \normsurf{\widehat{D_t v_+(0)}}{L^2(\Sigma\cap Y)}{2} \\\gtrsim \tau^3 \norm{ \widehat {v_+}}{L^2(\Rp \times Y)}{2}
			 +\tau \norm{\widehat{\nabla_{x^\prime}v_+}}{L^2(\Rp \times Y)}{2}     +\tau^3 \left\vert \widehat{v (0)} \right\vert_{L^2(\Sigma \cap Y)}^2 
			 \\+\tau \left\vert\widehat{ (\nabla v_+ )}(0) \right\vert_{L^2(\Sigma \cap Y)}^2+\tau \left\vert\widehat{ (\nabla v_-)}(0) \right\vert_{L^2(\Sigma \cap Y)}^2.
		\end{multline}
	The only remaining term is then $\tau \normsurf{D_n \hat{v}_+}{L^2(\Sigma \cap Y)}{2}$. To obtain it we use the commutator technique (see \cite[Section 3A]{LR:95}). we write
	$$
	\invcp P^{+}_\phi=Q_2+ i \tau Q_1,
	$$
	with $\qr=D^2_n-\tau^2 |\phi^\prime |^2+|\xip|^2-\invcp |\xi_t|^2  $ and $\qi=\phi^\prime D_n +D_n \phi^\prime  $. We integrate by parts taking into account the boundary terms to find:
	\begin{align}
	\label{equality with commutator}
	    \norm{\invcp P^{+}_{\phi} \hat{v}_+}{L^2(\Rp \times  Y)}{2}=\norm{\qr \hat{v}_+}{L^2(\Rp \times  Y)}{2}+ \tau^2\norm{\qi \hat{v}_+}{L^2(\Rp \times  Y)}{2}+ i \tau \left([\qr,\qi]\hat{v}_+,\hat{v}_+ \right )_++\tau \mathcal{B}(\hat{ v}).
	\end{align}
	With $\mathcal{B}(\hat{ v})$ the boundary term which can be written as (see for instance \cite{LR:95} or \cite[Proposition 3.24]{rousseau2022elliptic}):
	\begin{align*}
	    \mathcal{B}(\hat{ v})=2 \left(\phi^\prime D_n \hat{ v} ,D_n \hat{ v}\right)_{\Sigma \cap Y}+\left(M_1 \hat{ v} ,D_n \hat{ v} \right)_{\Sigma \cap Y}
	    +\left(M_1^{'}D_n \hat{ v} ,\hat{ v}\right)_{\Sigma \cap Y}+\left(M_2 \hat{ v} ,\hat{ v}\right)_{\Sigma \cap Y},
	\end{align*}
	where $M_j$ is a Fourier multiplier by a polynomial of degree $j$ in $(\xip, \xi_t, \tau)$. Using this as well as the Young inequality yields, for arbitrary $\delta>0$:
	\begin{align*}
	|\mathcal{B}(\hat{ v})-2  \left(\phi^\prime D_n \hat{ v} ,D_n \hat{ v}\right)_{\Sigma \cap Y} | \lesssim \tau^2 \normsurf{\hat{ v}}{L^2(\Sigma \cap Y)}{2}
+(1+\delta^{-1})\normsurf{\nabla_{x^\prime}\hat{ v}}{L^2(\Sigma \cap Y)}{2}+ \delta\normsurf{D_n \hat{v}_+}{L^2(\Sigma \cap Y)}{2}.    
	\end{align*}
We combine this last inequality with \eqref{equality with commutator}, recalling that $\phi^\prime >0$ in the support of $\hat{ v}$ and by choosing $\delta$ sufficiently small to find, for $\tau$ sufficiently large:
\begin{multline}
\label{estimate with commutator for G+}
    \norm{ P^{+}_{\phi} \hat{v}_+}{L^2(\Rp \times  Y)}{2}+\tau^3 \normsurf{\hat{v}_+}{L^2(\Sigma \cap Y)}{2}
+\tau \normsurf{\nabla_{x^\prime}\hat{v}_+}{L^2(\Sigma \cap Y)}{2} 
\\ \gtrsim
    \norm{\qr \hat{v}_+}{L^2(\Rp \times  Y)}{2}+ \tau^2\norm{\qi \hat{v}_+}{L^2(\Rp \times  Y)}{2}+ i \tau \left([\qr,\qi]\hat{v}_+,\hat{v}_+ \right )_++\tau \normsurf{D_n \hat{v}_+}{L^2(\Sigma \cap Y)}{2} .
\end{multline}
This almost gives us the desired term: we need to take care of the commutator. In our case, this can be done in a very simple way. Indeed, one can write:
\begin{equation*}
    i[\qr,\qi]=B_0 \qr + B_1 \qi +B_2,
\end{equation*}
with $B_j$ a Fourier multiplier by a polynomial of degree $j$ in $(\xip, \xi_t, \tau)$. This implies that:

\begin{equation*}
   \left| i \tau \left(B_0 \qr \hat{v}_+,\hat{v}_+\right)\right| \lesssim \tau^{-\frac{1}{2}} \norm{\qr \hat{v}_+}{L^2(\Rp \times  Y)}{2}+ \tau^{\frac{5}{2}} \norm{\hat{v}_+}{L^2(\Rp \times  Y)}{2},
\end{equation*}
as well as 
\begin{equation*}
 \left| i \tau \left(B_1 \qi \hat{v}_+,\hat{v}_+\right)\right| \lesssim \tau\norm{\qi \hat{v}_+}{L^2(\Rp \times  Y)}{2}+\tau^{3}\norm{\hat{v}_+}{L^2(\Rp \times  Y)}{2}+ \tau^{\frac{1}{2}}\norm{\widehat{\nabla_{x^\prime}v_+}}{L^2(\Rp \times  Y)}{2}+\tau \norm{\widehat{D_t v}}{L^2(\Rp \times  Y)}{2}
    \end{equation*}
    and
    \begin{equation*}
      \left| i \tau \left(B_2 \hat{v}_+,\hat{v}_+\right)\right| \lesssim \tau^3 \norm{\hat{v}_+}{L^2(\Rp \times  Y)}{2}+\tau \norm{\widehat{\nabla_{x^\prime}v_+}}{L^2(\Rp \times  Y)}{2}+\tau \norm{\widehat{D_t v}}{L^2(\Rp \times  Y)}{2}.
    \end{equation*}
	We can now inject these three estimates in the commutator term in \eqref{estimate with commutator for G+} taking $\tau \geq \tau_0 $, $\tau_0$ large to absorb the terms $\norm{Q_1}{}{2}, \norm{Q_1}{}{2} $ to finally find:
		\begin{multline}
		\label{estimate in G+ for Dn}
			\norm{\widehat{\cnop^+ v_+}}{L^2(\Rp \times Y)}{2}+\tau \norm{ \widehat{D_t v_+}}{L^2(\Rp \times Y)}{2} + \tau  \normsurf{\widehat{D_t v_+(0)}}{L^2(\Sigma\cap Y)}{2} + \tau^3 \norm{ \widehat {v_+}}{L^2(\Rp \times Y)}{2}\\
			 +\tau \norm{\widehat{\nabla_{x^\prime}v_+}}{L^2(\Rp \times Y)}{2}     +\tau^3 \left\vert \widehat{v (0)} \right\vert_{L^2(\Sigma \cap Y)}^2 
			 \gtrsim \tau \normsurf{D_n \widehat{v_+}(0)}{L^2( \Sigma \cap Y ) }{2}.
		\end{multline}
	We can now multiply \eqref{first estimate G+ } by a large constant and add it to \eqref{estimate in G+ for Dn} to conclude the proof of Lemma~\ref{region tau petit non elliptic toy}.
\end{proof}

	\subsection{End of the proof for the toy model}
	
	We finish here the proof of Proposition~\ref{theorem bis} for the constant coefficient case. One has to deal with all the possible cases and put together the estimates Section \ref{section factorization microlocal and estimates}. We have the following \textit{partition} of $\R^{n} \ni (\xip, \xi_t)$: 
	\begin{align*}
	\R^{n}&=\{\tau \geq \frac{1}{\sigma} (|\xip|+|\xi_t|)\}\sqcup
	\left(\{\tau <\frac{1}{\sigma}(|\xip|+|\xi_t|)\}\cap	\mathcal{E}_{\epsilon}^{-}\cap\mathcal{E}_{\epsilon}^{+} \right)\sqcup \left( \{\tau <\frac{1}{\sigma}(|\xip|+|\xi_t|)\}\cap\mathcal{G}_\epsilon\mathcal{H}_\epsilon\right) \\
	&=Y^{\sigma}_1 \sqcup Y^{\sigma, \epsilon}_2 \sqcup Y^{\sigma, \epsilon}_3
\end{align*}
with
$$ 
Y^{\sigma}_1:=\{\tau \geq \frac{1}{\sigma}  (|\xip|+|\xi_t|)\}, \quad Y^{\sigma, \epsilon}_2:=	\{\tau <\frac{1}{\sigma}(|\xip|+|\xi_t|)\}\cap	\mathcal{E}_{\epsilon}^{-}\cap\mathcal{E}_{\epsilon}^{+}, \quad Y^{\sigma, \epsilon}_3:= \{\tau <\frac{1}{\sigma}(|\xip|+|\xi_t|)\}\cap \left(\mathcal{G}_\epsilon \mathcal{H}_\epsilon^{-}\cup\mathcal{G}_\epsilon \mathcal{H}_\epsilon^{+}\right).
$$
We recall the notations/definitions of Section \ref{section factorization microlocal and estimates}. The crucial remark is that in all of the above regions with the exception of $\mathcal{E}^{-}\cap\mathcal{E}^{+}$ we have $|\xi_t| \gtrsim |\xip|$.
	
	In this particular toy model we are dealing with here, we work on the Fourier domain and we can simply restrict ourselves in each of those regions and prove the sought estimate. In the general case treated in Section \ref{proof for the general case} symbolic calculus will be used and as a result we will have to use \textit{overlapping} regions and an associated \textit{smooth} partition of unity. This being said we deal with the three regions above to conclude:
	
	\bigskip

	\boxed{Y^{\sigma}_1} Here we just apply the result of Lemma~\ref{tau grand toy} which gives:
	\begin{multline}
	\label{estimate tau grand toy}
			\norm{\widehat{\cnop v}}{L^2(\R \times  Y^{\sigma}_1)}{2}+\tau \norm{ \widehat{D_t v_+}}{L^2(\Rp \times  Y^{\sigma}_1)}{2}+\tau \norm{ \widehat{D_t v_-}}{L^2(\Rm \times  Y^{\sigma}_1)}{2} 
			\\+ \tau  \normsurf{\widehat{D_t v_+}(0)}{L^2(\Sigma\cap  Y^{\sigma}_1)}{2}+\tau  \normsurf{\widehat{D_t v_-}(0)}{L^2(\Sigma \cap  Y^{\sigma}_1)}{2}  \\  \gtrsim \tau^3 \norm{ \widehat {v_+}}{L^2(\Rp \times  Y^{\sigma}_1)}{2}
			 +\tau^3 \left\vert \widehat{v (0)} \right\vert_{L^2(\Sigma \cap  Y^{\sigma}_1)}^2 +\tau \left\vert\widehat{ (\nabla v_+} )(0) \right\vert_{L^2(\Sigma \cap  Y^{\sigma}_1)}^2+\tau \left\vert \widehat{ (\nabla v_- )}(0) \right\vert_{L^2(\Sigma \cap  Y^{\sigma}_1)}^2.
		\end{multline}
This region fixes the choice of $\sigma \leq \sigma_0$.

	\boxed{Y^{\sigma, \epsilon}_2}In this region our operator is elliptic and one can follow the proof of \cite{le2013carleman}. Indeed here one has an elliptic factorization on both sides of the interface. That is
	$$
		\invcp \cnop^{+}v_{+}=(D_n+i \tau \phi^\prime + im_+)(D_n+i \tau \phi^\prime - im_+)v_+,\quad \phi^\prime =\alpha_{+}+ \beta x_n,
		$$
	and
	$$
		\invcm \cnop^{-}v_{-}=(D_n+i \tau \phi^\prime + im_-)(D_n+i \tau \phi^\prime - im_-)v_-,\quad \phi^\prime =\alpha_{-}+ \beta x_n.
		$$
	with
$$ m_{\pm}= \sqrt{|\xip|^2-\invc_{\pm} \xi_{t}^2} \gtrsim |\xip| + |\xi_t|.$$ 
Consequently the arguments used in \cite{le2013carleman} can be used in this microlocal region as well. We refer to Lemma \ref{Lemma elliptic region} for a rigorous proof in the general case. 
	Notice that this region forces:
	$$
	\frac{\alpha_+}{\alpha_-} > \underset{x_n=0}{\sup} \frac{m_+}{m_-}.
	$$
	We obtain the following estimate:
	\begin{multline}
	\label{estimate in E_E+}
			\norm{\widehat{\cnop v}}{L^2(\R \times Y^{\sigma, \epsilon}_2)}{2}+\tau \norm{ \widehat{D_t v_+}}{L^2(\Rp \times Y^{\sigma, \epsilon}_2)}{2}+\tau \norm{ \widehat{D_t v_-}}{L^2(\Rm \times Y^{\sigma, \epsilon}_2)}{2} 
			\\+ \tau  \normsurf{\widehat{D_t v_+}(0)}{L^2(\Sigma\cap Y^{\sigma, \epsilon}_2)}{2}+\tau  \normsurf{\widehat{D_t v_-}(0)}{L^2(\Sigma \cap Y^{\sigma, \epsilon}_2)}{2}  \\  \gtrsim \tau^3 \norm{ \widehat {v_+}}{L^2(\Rp \times Y^{\sigma, \epsilon}_2)}{2}
			 +\tau^3 \left\vert \widehat{v (0)} \right\vert_{L^2(\Sigma \cap Y^{\sigma, \epsilon}_2)}^2 +\tau \left\vert\widehat{ (\nabla v_+} )(0) \right\vert_{L^2(\Sigma \cap Y^{\sigma, \epsilon}_2)}^2+\tau \left\vert \widehat{ (\nabla v_-} )(0) \right\vert_{L^2(\Sigma \cap Y^{\sigma, \epsilon}_2)}^2.
		\end{multline}
	\bigskip

\boxed{Y^{\sigma, \epsilon}_3} Here we are in the situation of Lemma~\ref{region tau petit non elliptic toy}. We obtain:
\begin{multline}
	\label{estimate in tau petit non ell}
			\norm{\widehat{\cnop v}}{L^2(\R \times  Y^{\sigma, \epsilon}_3)}{2}+\tau \norm{ \widehat{D_t v_+}}{L^2(\Rp \times  Y^{\sigma, \epsilon}_3)}{2}+\tau \norm{ \widehat{D_t v_-}}{L^2(\Rm \times  Y^{\sigma, \epsilon}_3)}{2} 
			\\+ \tau  \normsurf{\widehat{D_t v_+}(0)}{L^2(\Sigma\cap  Y^{\sigma, \epsilon}_3)}{2}+\tau  \normsurf{\widehat{D_t v_-}(0)}{L^2(\Sigma \cap  Y^{\sigma, \epsilon}_3)}{2}  \\  \gtrsim \tau^3 \norm{ \widehat {v_+}}{L^2(\Rp \times  Y^{\sigma, \epsilon}_3)}{2}
			 +\tau^3 \left\vert \widehat{v (0)} \right\vert_{L^2(\Sigma \cap  Y^{\sigma, \epsilon}_3)}^2 +\tau \left\vert\widehat{ (\nabla v_+ )}(0) \right\vert_{L^2(\Sigma \cap  Y^{\sigma, \epsilon}_3)}^2+\tau \left\vert \widehat{ (\nabla v_- )}(0) \right\vert_{L^2(\Sigma \cap  Y^{\sigma, \epsilon}_3)}^2.
		\end{multline}

\bigskip

\begin{proof}[End of the proof of proposition~\ref{theorem bis} for the toy model]
 We can now finish our proof in this particular setting by simply putting together the results of the three above regions. Indeed, adding \eqref{estimate tau grand toy}, \eqref{estimate in E_E+} and \eqref{estimate in tau petit non ell}  yield:
	\begin{multline*}
			\norm{\widehat{\cnop v}}{L^2(\R \times \R^n)}{2}+\tau \norm{ \widehat{D_t v_+}}{L^2(\Rp \times \R^n)}{2}+\tau \norm{ \widehat{D_t v_-}}{L^2(\Rm \times \R^n)}{2} 
			\\+ \tau  \normsurf{\widehat{D_t v_+}(0)}{L^2(\Sigma )}{2}+\tau  \normsurf{\widehat{D_t v_-}(0)}{L^2(\Sigma  )}{2}  \\  \gtrsim \tau^3 \norm{ \widehat {v}}{L^2(\R \times \R^n)}{2}
			 +\tau^3 \left\vert \widehat{v (0)} \right\vert_{L^2(\Sigma  )}^2 +\tau \left\vert\widehat{ (\nabla v_+ )}(0) \right\vert_{L^2(\Sigma  )}^2+\tau \left\vert \widehat{ (\nabla v_- )}(0) \right\vert_{L^2(\Sigma )}^2,
		\end{multline*}
	which using the Plancherel theorem translates to,
	\begin{multline}
	\label{but without volume derivatives}
			 \norm{\cnop v }{\lrn}{2} +\tau \norm{H_+  D_t v_+}{\lrn}{2}+\tau \norm{H_-  D_t v_-}{\lrn}{2}  \\
			+\tau \left\vert (D_t v_+ ) \right\vert^2_{L^2(\Sigma)}+\tau \left\vert (D_t v_- ) \right\vert^2_{L^2(\Sigma)} \\\gtrsim
			\tau^3 \norm{ v}{\lrn}{2}
			+\tau^3 \left\vert v  \right\vert_{L^2(\Sigma)}^2  +\tau \left\vert (\nabla v_+ ) \right\vert_{L^2(\Sigma)}^2  +\tau \left\vert (\nabla v_- )\right\vert_{L^2(\Sigma)}^2.
		\end{multline}
Notice that the only missing term from \eqref{but without volume derivatives} is the volume norm of the gradients. Proposition~\ref{theorem bis} is then a result of \eqref{but without volume derivatives} and Lemma~\ref{derivatives Dn}.
\end{proof}

\section{Proof of Proposition~\ref{theorem bis} for the general case}
\label{proof for the general case}

\subsection{Notation, microlocal regions and first estimates}
The proof for the general case uses essentially the ideas introduced in Section~\ref{proof of toy model}. The main difference is that one needs to consider this time sub-regions of the whole phase space and not only of the frequency space. To do so one needs to use some microlocal analysis tools. We refer to Appendix~\ref{a few facts on pseudo} for some basic properties that will be used in the sequel.

We first define, for $m \in \R$ the class of standard \textit{tangential} smooth symbols $\mathcal{S}^m$. These are the functions $a \in C^{\infty}( t,x,\xi_t, \xip)$ satisfying for all $(\alpha, \beta) \in \mathbb{N}^{n+1} \times \mathbb{N}^n$:
$$
\sup_{( t,x,\xi_t, \xip)} (1+|\xip|^2+ |\xi_t|^2)^{\frac{-m+ |\beta|}{2}}\left|(\partial^{\alpha}_{t,x} \partial^{\beta}_{\xi_t,\xip})a( t,x,\xi_t, \xip)\right| < \infty.
$$
We will also work in the class of smooth tangential symbols depending on a large parameter. This class will be denoted by $\mathcal{S}_{\tau}^m$ and contains the functions $a \in C^{\infty}( t,x,\xi_t, \xip, \tau)$ satisfying for all $(\alpha, \beta) \in \mathbb{N}^{n+1} \times \mathbb{N}^n$:
$$
\sup_{\substack{( t,x,\xi_t, \xip) \\ \tau \geq 1}} (\tau^2+|\xip|^2+ |\xi_t|^2)^{\frac{-m+ |\beta|}{2}}\left|(\partial^{\alpha}_{t,x} \partial^{\beta}_{\xi_t,\xip})a( t,x,\xi_t, \xip, \tau)\right| < \infty.
$$

To an element $a$ of $\mathcal{S}^m$ we associate an operator $\op(a) \in \psit^m$~\footnote{Notice that in our definition we use the Weyl quantization and not the standard one.}, which is an element of the class of \textit{tangential} pseudodifferential operators.  Notice that to alleviate our notation we do not use the tangential notation for the symbols, since all of the symbols we will consider will be tangential. The same remark applies to the notation $\Psi^m$ which refers to tangential operators even though it is not explicit in the notation. We have the analogous notation for $$\psitt^m=\op(\mathcal{S}^m_\tau).$$

Let us denote $\lambda^2:=1+|\xip|^2+ |\xi_t|^2$ and $\lambda^2_\tau=\tau^2+|\xip|^2+ |\xi_t|^2$. We introduce the following Sobolev norms, defined in the tangential variables:
$$
\normsurf{u(x_n,\cdot)}{H^s}{}=\normsurf{\op(\lambda^s) u(x_n, \cdot)}{L^2(\R^{n})}{}, \quad \normsurf{u(x_n,\cdot)}{H^s_{\tau}}{}=\normsurf{\op(\lambda^s_{\tau}) u(x_n, \cdot)}{L^2(\R^{n})}{}.
$$

\begin{remark}
 In the sequel we will have to consider symbols $a \in \mathcal{S}^{m}$ independent of $\tau$. However the natural symbol class for our Carleman estimate is $\mathcal{S}^m_\tau$. Lemma~\ref{lemma with different classes} provides with a sufficient condition for making use of pseudodifferential calculus mixing operators in $\Psi^m$ and $\Psi^m_\tau$.
\end{remark} 

\begin{remark}
\label{Carleman insensitive}
Of crucial importance is the fact that the Carleman estimate we are seeking to prove (in fact most Carleman estimates in general) is insensitive to perturbations with respect to elements of the class $$\psit^1+\tau \psit^0+ \psit^0 D_n,$$ up to taking even larger values for our parameter $\tau$. Let us briefly recall why. Suppose that Proposition~\ref{theorem bis} is proved for $P_{\phi}$ and consider $T \in \psit^1+\tau \psit^0+ \psit^0 D_n$, that is $T=S_1+\tau S_0+\tilde{S}_0 D_n$ for $S_j$ tangential operators of order $j$. Since $\norm{(P_{\phi}+T)v}{\lrn}{} \leq \norm{P_{\phi}v}{\lrn}{}+\norm{Tv}{\lrn}{}$ one simply needs to show that $\norm{Tv}{\lrn}{}$ can be absorbed in the right hand side of our estimate. By Sobolev regularity of the pseudo differential calculus one has:

$$\norm{S_1 v }{\lrn}{} \lesssim \norm{v}{L^2(\R; H^1)}{}, 
$$
which yields
$$
\tau^{\frac{3}{2}} \norm{v}{\lrn}{}+\tau \norm{\nabla_{x^\prime}v}{\lrn}{}- \norm{S_1 v }{\lrn}{}  \gtrsim \tau^{\frac{3}{2}} \norm{v}{\lrn}{}+\tau \norm{\nabla_{x^\prime}v}{\lrn}{},\forall \:  \tau \geq \tau_0
$$
 for $\tau$ sufficiently large. Similarly, using
$$
\norm{\tau S_0 v}{\lrn}{} \lesssim \tau \norm{v}{\lrn}{}, \quad \norm{ \tilde{S}_0 D_n v}{\lrn}{} \lesssim  \norm{D_n v}{\lrn}{},
$$
we see that the perturbation is absorbed in our estimate. We can therefore from now on replace without any loss of generality our operator $P_\phi$ by an element of $P_\phi+\psit^1+\tau \psit^0+ \psit^0 D_n$. Remark as well that if $L\in \mathcal{D}^1$ is a differential operator of order one then $L_\phi=e^{\tau \phi}Le^{-\tau \phi} \in \psit^1+\tau \psit^0+ \psit^0 D_n$.
\end{remark}

\bigskip

Recall that we are working in the local setting of Section~\ref{local setting} with the operator:
$$
P_2^{\pm}=-D^2_t+c_\pm(x)D^2_n+c_\pm(x)Q(x,D_{x^\prime}),
$$
with 
\begin{equation}
\label{def of Q}
 Q(x,\xip):=\bjk \xi_j \xi_k 
	\end{equation}  and $b_{jk}$ satisfying 
$$
b_1 |\xip|^2 \leq Q(x,\xip) \leq b_2 |\xip|^2, \quad b_1, b_2 >0.
$$
The conjugated operator is given then by
$$
P_\phi=H_-P_\phi^{-}+H_+P_\phi^{+},
$$
where
$$
c^{-1}_{\pm}(x)P^{\pm}_{\phi}=(D_n+ i \tau \phi^\prime )^2+Q(x,D_{x^\prime})   -c^{-1}_{\pm}(x)D^2_t.
$$

We now consider the analog of the microlocal regions used in the toy model, for $\epsilon > 0 $ small to be chosen:
	
\begin{equation}
\label{def of E}
 \mathcal{E^{\bullet}_{\epsilon}}:=\{ (t,x, \xip, \xi_t)\in \R \times \R^{n} \times \R^{n-1} \times \R \quad \textnormal{such that} \quad   Q(x,\xip)-c^{-1}_{\bullet}(x)\xi^2_t \geq \epsilon (|\xip|^2+ |\xi_t|^2)\}, 
\end{equation}
  
 \begin{equation}
     \label{def of GH}
     \mathcal{G}\mathcal{H^{\bullet}_{\epsilon}}:= \{ (t,x, \xip, \xi_t)\in \R \times \R^{n} \times \R^{n-1} \times \R \quad \textnormal{such that} \quad   Q(x,\xip) - c_{\bullet}(x)\xi^2_t \leq 2\epsilon (|\xip|^2+ |\xi_t|^2)\}.
 \end{equation}

Notice that for $\bullet=+$ or $\bullet=-$, $\mathcal{E^{\bullet}_{\epsilon}}$ and $\mathcal{G}\mathcal{H^{\bullet}_{\epsilon}}$ overlap and are conic in $(\xi_t,\xip)$ which will allow to construct an associated partition of unity. Another important remark is that whenever we are in the regions $\mathcal{G}\mathcal{H^{\bullet}_{\epsilon}}$  we have (for $\epsilon$ sufficiently small) thanks to the ellipticity of $b_{jk}$ that $|\xi_t| \gtrsim |\xip|$.

The Carleman estimate we want to show (see Proposition~\ref{theorem bis}) concerns functions $v$ supported in a compact set $K \subset V$. However the space of compactly supported functions is not stable by pseudo differential operators. Let us denote by $\pi_{t,x}$ the projection in the physical space of an element of $\R_t\times \R_x \times \R_{\xi_t}\times \R_{\xip}$. The natural space in which we will be working is the Schwartz space $\mathscr{S}$ and we will use cut-off functions  $\chi$ satisfying $\pi_{t,x}(\supp{\chi}) \subset K$. That is the projection of their support on the physical space will be contained on a compact set $K$. This will allow us to suppose that $(t,x)$ lies on a compact set. Indeed, if we consider an additional cut-off function $\tchi=\tchi(t,x)$ to the left of our operator, with $\supp{\tchi} \subset K$ and $\tchi=1$ on $\pi_{t,x}(\supp{\chi})$ then we have for $u \in \mathscr{S}(\R^{n+1})$:
\begin{align*}
    \norm{\cnop \op(\chi)u}{\lrn}{} \geq \norm{\tchi \cnop \op(\chi)u}{\lrn}{}- \norm{(1-\tchi)\cnop\op(\chi)u}{\lrn}{},
\end{align*}
since $(1-\tchi)P_\phi \op(\chi) \in \Psi^{-\infty}_\tau$ for $\chi \in \mathcal{S}^0_\tau$, the last term above yields an error term  which can be absorbed in our estimate. Up to replacing $\cnop$ by $\tchi \cnop$ we can indeed suppose that $(t,x)$ lies on a compact set of $\R^{n+1}$.

We shall work in the space
$$
\mathscr{S}_c=\{u \in \mathscr{S}(\R^{n+1}); \text{ there exists }  \eta>0,\: \supp{u} \subset \R_t \times \R^{n-1}_{x^\prime}  \times (-\eta,\eta)_{x_n}\}.
$$
Since all of the pseudo differential operators we consider are tangential, the support of a function in the $x_n$ direction is preserved and the above space is therefore stable by application of pseudodifferential operators in $\Psi^m$ or $\Psi^m_\tau$.

\bigskip

\textit{In the sequel, the implicit constants may depend on the coefficients $b_{jk}, c_{\pm}$ of $\cnop$ , on the coefficients of $\phi$  $(\alpha_{\pm}, \beta$) and they may also depend on $\epsilon$. However the value of $\epsilon$ will be a small fixed value depending on $\alpha_{\pm}, b_{jk}$ and $c_{\pm}$. More precisely, $\epsilon$ is fixed by the above remark guaranteeing that being in the regions $\mathcal{G}\mathcal{H}^{\pm}_{\epsilon}$  implies that $|\xi_t| \gtrsim |\xip|.$  Once the value of $\epsilon$ has been fixed in this way, the implicit constants depend on $b_{jk}, c_{\pm}, \alpha_{\pm}, \beta$. The choice of the coefficients $\alpha_{\pm}$ and $\beta$ is done in Lemma~\ref{Lemma elliptic region}. We take $\alpha_\pm$ such that the geometric condition \eqref{geometric assumption} is satisfied and $\beta$ large such that the sub-ellipticity condition \eqref{estim sous ell changins sign} is satisfied.}

\bigskip

Let us recover some of the basic estimates of Section~\ref{proof of toy model}. Recall the definition of the space $\tranphi$ as well as the setting given in Section~\ref{local setting}. In particular, elements of $\tranphi$ have small support contained in $V_t$. We start with the lemma giving the trace of the normal derivatives modulo the surface norm:
\begin{lem}
		\label{one of boundary derivatives cv}
	Let $\tilde{v}=H_-\tilde{v}_-+H_+\tilde{v}_+ \in \tranphi$ and suppose that $\chi \in \mathcal{S}^{0}$ or that $\chi \in \mathcal{S}^0_\tau$. Consider $v_{\pm}=\op(\chi) \tilde{v}_{\pm}$.	Then one has:
		$$
		\tau |D_nv_{\pm}|_{L^2(\Sigma)}^2 \lesssim \tau |D_n v_{\mp}|_{L^2(\Sigma)}^2+ \tau^3 |v_{+}|_{L^2(\Sigma)}^2+\tau^3 |v_{-}|_{L^2(\Sigma)}^2+T_{\theta,\Theta}+\normsurf{\tilde{v}_\mp}{\ls}{2}.
		$$

	\end{lem}
	
   	\begin{proof}
		This is a result of the transmission conditions and of the fact that $\op(\chi)$ is a \textit {tangential} operator. This implies that it commutes with restriction on $\Sigma$ (here restriction on $x_n=0$). We use then~\eqref{transmission cond1} and~\eqref{transmission cond2}. The first one allows to write $v_-(0)=v_+(0)+\op(\chi)\theta_\phi$
		and \eqref{transmission cond2} reads
		$$
		D_nv_-(0)=\frac{\cp(x)}{\cm(x)} (D_n v_+(0) + i \tau \alpha_+ v_+ (0))- i \tau \alpha_- v_-(0)+\frac{1}{\cm(x)}\op(\chi)\Theta_\phi+[D_n,\op(\chi)]\tilde{ v}_+.
		$$
	Since $\left|[D_n,\op(\chi)]\tilde{ v}_\pm\right | = |\op(\partial_{x_n}\chi) \tilde{ v}_\pm|$, using the fact that $0<c_{\min}<c(x)< c_{\max}$ we find that there exist constants $C_1$, $C_2$ and $C_3$ depending on $c_{\pm}$ and $\alpha_{\pm}$ such that 
		
		$$
		\tau \normsurf{D_n v_-}{L^2(\Sigma)}{2} \leq C_1 \tau \normsurf{D_n v_+}{L^2(\Sigma)}{2}+C_2 \tau^3  \normsurf{v_\pm}{L^2(\Sigma)}{2}+C_3 T_{\theta,\Theta}+\normsurf{\tilde{v}_\pm}{\ls}{2},
		$$
	where we have used the fact that since $\chi \in \mathcal{S}^0$ one has $\norm{\op(\chi)}{L^2\rightarrow L^2}{}\lesssim 1$. This gives one of the two desired inequalities and we can get the second one using again \eqref{transmission cond2}.
	\end{proof}
	
Recall that $E_t(v)$ has been defined in \eqref{def of Etv}. The following lemma allows to obtain the volume norms of the derivatives modulo the remainder of the terms. This corresponds to Lemma~\ref{derivatives Dn} in the toy model of Section~\ref{proof of toy model}.

\begin{lem}
		\label{derivatives Dn cv}
		There exists  $\tau_0>0 $ such that for $v \in \tranphi$ and $\tau \geq \tau_0$ we have:
	   \begin{multline*}
			\norm{H_-\cnop v_-}{L^2(\R^{n+1})}{2}+	\norm{H_+\cnop v_+}{L^2(\R^{n+1})}{2} +\errt +\bigg(\tau^3\norm{v}{\lrpn}{2}+\tau^2\normsurf{v_-(0)}{L^2(\Sigma)}{2}+\tau^2\normsurf{v_+(0)}{L^2(\Sigma)}{2}  
			\\ +\normsurf{D_n v_{+}(0)}{L^2(\Sigma)}{2}+\normsurf{D_n v_{-}(0)}{L^2(\Sigma)}{2} \bigg)
			\gtrsim \tau \norm{\nabla v_{-}}{\lrmn}{2}+	\tau \norm{\nabla v_{+}}{\lrpn}{2}.
		\end{multline*}
	\end{lem}

	\begin{proof}
		
		We start with the positive half-space $\R^{n+1}_+$. One has the following elementary inequality:
		
		$$
		\norm{\cnop^+ v_+}{L^2(\R^{n+1}_+)}{2}+\tau^2\norm{v_+}{L^2(\R^{n+1}_+)}{2}\geq 2 \tau \operatorname{Re}(P^+_{\phi}v_+,H_+v_+).
		$$
 Now we can integrate by parts. Indeed, recall the form of our operator:
		\begin{align*}
			\invcp(x) P^+_{\phi} &=(D_n^2+ \tau \beta +2 i \tau \phi^\prime D_n-\tau^2|\phi^\prime | ^2)
			+Q(x,D_{x^\prime}) - \invcp(x)D_t^2
		\end{align*}
		and that the definition of the weight function $\phi$  gives $\phi^\prime (x_n)=\alpha_{\pm}+\beta x_n $, therefore:
		\begin{align*}
			2\tau \operatorname{Re} (P^+_{\phi}v,H_+v)&\gtrsim 2\tau \norm{D_n v_+}{\lrpn}{2}+2 \tau \int_{\R^n}\operatorname{Re}(\partial_{x_n}v_+(0) \Bar{v}_+(0))dx^\prime dt\\& \hspace{4mm} +2\tau^2 \beta \norm{v_+}{\lrpn}{2} -2\tau^2 \alpha_+\normsurf{v_+}{\ls}{2} -2\tau^2 \beta \norm{v_+}{\lrpn}{2} \\
			&\hspace{4mm}-2 \tau^3\norm{\phi^\prime  v_+}{\lrpn}{2}+2 C\tau \norm{D^{'} v_+}{\lrpn}{2} -2C \tau \norm{D_t v_+}{\lrpn}{2}.
		\end{align*}
		Here we used the fact that $c$ is bounded as well as the ellipticity of $b_{jk}$. Indeed, up to adding an element of $\Psi^1$ (which does not have any influence on the estimate we are seeking to prove, see Remark~\ref{Carleman insensitive}) we can replace $Q(x,D_{x^\prime})$ by $-\textnormal{div}(B \cdot \nabla_{x^\prime})$ with $B=(b_{jk})_{1\leq j,k \leq n-1}$ which satisfies
		$$
		(-\textnormal{div}(B \cdot \nabla_{x^\prime})v,v)_{+} \gtrsim  \norm{\nabla_{x^\prime} v_+}{\lrpn}{2}.
		$$
		As one has 
		$$
		\left \vert 2 \tau \int_{\R^n}\operatorname{Re}(\partial_{x_n}v_+(0) \Bar{v}_+(0))dx^\prime dt \right \vert \lesssim \normsurf{D_nv_+}{L^2(\Sigma)}{2}+ \tau^2 \normsurf{v_+}{L^2(\Sigma)}{2},
		$$
		the above inequality can be written as
		$$
		2\tau \operatorname{Re} (P^+_{\phi}v,H_+v)_{L^2(\mathbb{R}_{+}^{n+1})} \gtrsim  \tau \norm{\nabla v{_+}}{\lrpn}{2}+R,
		$$
		with
		$$
		|R|\lesssim \tau^3\norm{v_+}{\lrpn}{2}+\tau \norm{D_t v_+}{\lrpn}{2}+\tau^2\normsurf{v_+}{L^2(\Sigma)}{2}+  \normsurf{D_nv_+}{L^2(\Sigma)}{2}.
		$$
		which gives the sought result. Since the proof above is insensitive with respect to the sign of the boundary terms coming from the integration by parts we also obtain the desired inequality on the negative half-space $\R^{n+1}_-$.
	 \end{proof}
When we are in a microlocal region where $|\xi_t|$ is large compared to $\tau$ and $|\xip|$ we automatically have a very good estimate:

\begin{lem}
\label{regime favor cv}
Let $c_0>0$ and  $\chi \in \mathcal{S}^{0}_\tau$ with  $\supp(\chi) \subset V \times \{|\xi_t| \geq c_0(\tau+ |\xip|)\}$ . Then there exists $\tau_0$ such that:
\begin{multline*}
    \tau \norm{  D_tv}{L^2(\R^{n+1}_{\pm})}{2} +\norm{u}{L^2(\R^{n+1}_{\pm})}{2}+ \normsurf{u}{\ls}{2} 
			+\tau \left\vert D_tv  \right\vert^2_{L^2(\Sigma)}\\\gtrsim
			\tau \norm{v}{L^2(\R_\pm ; H^1_\tau)}{2}
			+\tau^3 \normsurf{v}{\ls}{2}+\tau \norm{\nabla_{x^\prime}v}{L^2(\R^{n+1}_{\pm})}{2}  +\tau \normsurf{\nabla_{x^\prime}v}{\ls}{2} ,
\end{multline*}
for $\tau \geq \tau_0 $, $u \in \mathscr{S}_c(\R^{n+1}) $ and $v= \op(\chi)u$.
\end{lem}

\begin{remark}
Notice that thanks to Lemma~\ref{lemma with different classes} the support assumption on $\chi$ implies that if $\chi_0$ is an element of $\mathcal{S}^0$ independent of $\tau$ one has in fact that $\op(\chi_0)\op( \chi) \in \Psi^0_{\tau}$.
\end{remark}

\begin{proof}
Let $0\leq\tchi\leq 1$ satisfy the same properties as $\chi$ with moreover $\tchi=1$ on $\supp{\chi}$. Defining 
$$
\mathcal{S}^2_{\tau}\ni \tilde{a}:= \xi_t^2\tchi+(1-\tchi)\lambda^2_{\tau},
$$
we notice that $\tilde{a}\gtrsim \lambda_\tau^2$.
Let $a:= \xi_t^2 $ and remark that $\op(a) \op(\chi)= \op(\tchi a)\op(\chi)+R$ with $R=\op((1-\tchi)a)\op(\chi) \in \psitt^{-\infty}$, since $\tchi=1$ on $\supp{\chi}$. We then obtain:
\begin{align}
\label{auxil equality in lem regime fav }
&(\op(a)\op(\chi)u,\op(\chi)u)_{+} \\ \nonumber
&=(\op(\tilde{a})\op(\chi)u,\op(\chi) u)_{+}-(\op((1-\tchi)\lambda^2_{\tau})\op(\chi)u,\op(\chi)u)_{+}+(Ru,\op(\chi)u)_{+}
\end{align}
The last two terms yield an operator in $\psitt^{-\infty}$ which implies in particular
$$
|-(\op((1-\tchi)\lambda^2_{\tau})\op(\chi)u,\op(\chi)u)_{+}+(Ru,\op(\chi)u)_{+}|\lesssim  \tau^{-1}\norm{u}{\lrpn}{2}.
$$
We now can apply Gårding's inequality in the context of tangential pseudodifferential calculus with a large parameter (Lemma~\ref{Garding with a large}) to the first term to find (recall that $v=\op(\chi)u$):
$$
\operatorname{Re}(\op(\tilde{a})v,v)_+ \gtrsim \norm{v}{L^2(\Rp; H^1_{\tau})}{2}.
$$
The estimates above combined with the equality \eqref{auxil equality in lem regime fav } yield
\begin{align*}
\norm{  D_tv}{\lrpn}{2}=  (\op(a)v,v)_{+} \gtrsim \norm{v}{L^2(\Rp; H^1_{\tau})}{2} - \tau^{-1}\norm{u}{\lrpn}{2},
\end{align*}
We multiply the above estimate by $\tau$ and write the $H^1_\tau$ norm as $$
\norm{\cdot}{H^1_\tau}{2}\sim \tau^2 \norm{\cdot}{L^2}{2}+\norm{\nabla \cdot}{L^2}{2},
$$
to deduce
$$
 \tau \norm{ D_tv}{\lrpn}{2}+\norm{u}{\lrpn}{2} \gtrsim \tau^3 \norm{v}{\lrpn}{2}+ \tau \norm{\nabla_{x^\prime}v}{\lrpn}{2}.
$$
In a similar fashion, using the same notation we apply Gårding's inequality on $\Sigma$ (which here is $\R^n$). Notice that since the operators considered above are tangential they commute with the restriction on $\Sigma$ and we find
$$
\operatorname{Re}(\op(\tilde{a})v,v)_{\Sigma}\gtrsim\normsurf{v}{H^1_\tau}{2}-\tau^{-1}\normsurf{u}{\ls}{2},
$$
which implies in particular
$$
\tau \normsurf{D_t v}{\ls}{2} +\normsurf{u}{\ls}{2} \gtrsim \tau^3\normsurf{v}{\ls}{2}+ \tau \normsurf{\nabla_{x^\prime}v}{\ls}{2}.
$$
We obtain the same estimates when integrating in the negative half-line.
\end{proof}

\bigskip
We shall now show the desired estimate when micro-localized in a region where $$\tau \gg |\xi_t|+|\xip|.$$ 
From an heuristic point of view, in a such a region $(D_n+ i \tau \phip)^2$ is the most important term of $\cnop$. Since $D_n+ i \tau \phip$ is elliptic as an 1D operator we expect that this will give a good estimate. Let us define (as in \cite{le2013carleman}) $\psi \in C^{\infty}(\R)$ nonnegative with $\psi=0$ in $[0,1]$ and $\psi=1$ in $[2, +\infty)$ and then 
\begin{equation}
\label{defintion of psi sigma}
\psi_\sigma(\tau,\xi_t, \xip):=\psi\left(\frac{\sigma \tau}{(1+|\xi_t|^2+ |\xip|^2)^{\frac{1}{2}}}\right) \in \mathcal{S}^0_{\tau},    
\end{equation}
with $\sigma$ small to be chosen. The choice of $\psi$ implies that $\tau \geq (|\xi_{t}|^2+|\xip|^2)^{1/2}/\sigma$ on the support of $\psi_{\sigma}$. We have the following lemma:

\begin{lem}
\label{prop for tau grand}
There exists $\sigma_0, \tau_0$ and $\mathcal{V} \Subset V_t$ such that for $0<\sigma < \sigma_0$ we have:
\begin{align*}
     \norm{H_-\cnop v_-}{\lrn}{2}&+ \norm{H_+\cnop v_+}{\lrn}{2} +\norm{u}{\lrn}{2}+\normsurf{u_+}{\ls}{2}+\normsurf{u_-}{\ls}{2} +T_{\theta,\Theta}
     \\& \gtrsim \tau \norm{v}{L^2(\R ; H^1_\tau)}{2}+\tau^3 \normsurf{v_+}{\ls}{2}+\tau^3 \normsurf{v_-}{\ls}{2}
      +\tau\normsurf{ \nabla v_+}{\ls}{2}+\tau\normsurf{ \nabla v_-}{\ls}{2},
\end{align*}
for all $\tau \geq \tau_0$, $u \in \tranphi$ with $\supp u \subset \mathcal{V}$ and $v=\op(\psis)u$.
\end{lem}

\begin{remark}
In fact here we obtain a slightly better estimate (without the loss of a half-derivative in the volume norm). But we prefer to state it like this since that is how we use it in view of the final estimate, when all regions are put together.
\end{remark}

\begin{proof}
Recall that we write $u=H_-u_-+H_+u_+$ and $v_\pm=\op(\psis)u_\pm$. We start by working in the positive half-space $\R^{n+1}_+$. As one could expect this is where we obtain most terms of the sought estimate, since it is the observation region.

We write:
$$
\invcp(x) P_\phi ^+=(D_n+ i \tau \phi^\prime )^2+ Q(x,D_{x^\prime}) -\invcp D_t^2:=A_+ +R_+,
$$
where we have defined $A_+:=(D_n+ i \tau \phi^\prime )^2$ and $R_+:=Q(x,D_{x^\prime}) -\invcp D_t^2$. One has with $w=( D_n+ i \tau \phi^\prime )v_+=(D_n+ i \tau \alpha_+ + i \tau \beta x_n)v_+$, using \eqref{ineg 1 positif} twice:
	\begin{align*}
		\norm{A_+v_+}{\lrpn}{2}&=\norm{(D_n+ i \tau \alpha_+ + i \tau \beta x_n)w}{\lrpn}{2} \\
		&\geq \norm{(\tau \beta x_n + \tau \alpha_+ )w}{\lrpn}{2} + \alpha_+ \tau \normsurf{w}{\ls}{2}\\ &\gtrsim \tau^2\left(\norm{D_n v_+}{\lrpn}{2}+ \norm{ (\tau \alpha_+ +  \tau \beta x_n)v_+}{\lrpn}{2}+ \tau \alpha_+ \normsurf{v_+}{\ls}{2} \right)\\
		& \hspace{4mm}+ \alpha_+ \tau\normsurf{(D_n+i \tau \alpha)v_+}{\ls}{2}. 
	\end{align*}

In particular using the fact that $\alpha_+>0$ and taking $|x_n|$ sufficiently small we find that for $u$ supported in a small neighborhood $\mathcal{V} \Subset V_t$ one has  $$\norm{A_+v_+}{\lrpn}{2} \gtrsim \tau ^4 \norm{v_+}{\lrpn}{2} + \tau^3 \normsurf{v_+}{\ls}{2}+\tau\normsurf{(D_n+i \tau \alpha)v_+}{\ls}{2}.$$ 
We deduce then
\begin{equation*}
   \norm{A_+v_+}{\lrpn}{2} \gtrsim  \tau ^4 \norm{v_+}{\lrpn}{2} + \tau^3 \normsurf{v_+}{\ls}{2}+ \tau \normsurf{D_n v_+}{\ls}{2}.
\end{equation*}
Recalling that $A_+=c^{-1}P^{+}_\phi-R_+$ we obtain
\begin{align}
\label{aux ineq 1}
\norm{\cnp  v_+}{\lrpn}{2} &\geq C\norm{A_+ v_+}{\lrpn}{2} - \norm{R_+ v_+}{\lrpn}{2} \nonumber \\
&\geq  C\left(\tau ^4 \norm{v_+}{\lrp}{2} + \tau^3 \normsurf{v_+}{\ls}{2}+\tau \normsurf{D_nv_+}{\ls}{2}\right)- \norm{R_+ v_+}{\lrpn}{2}\nonumber \\
&=\left((C\tau^4-R_+^*R_+)v_+,v_+\right)_+ +C\tau^3 \normsurf{v_+}{\ls}{2}+C\tau \normsurf{D_nv_+}{\ls}{2},
\end{align}
with $C$ positive constant depending on the coefficients of $\cnop$ and of $\phi$. Observe now that the principal symbol $r_+^2=\left(Q(x,\xip)- \invcp(x) |\xi_t|^2\right)^2=\left(\bjk \xi_j \xi_k- \invcp(x) |\xi_t|^2\right)^2$ of $R^*_{+}R_{+} \in \Psi^4$ satisfies  $$r^2_+ \leq \max(b_2,  c^{-1}_{\min}) (|\xi_t|^2+|\xip|^2)^2.$$
 Using the fact that  $$\tau \geq \frac{1}{\sigma}(1+|\xi_t|^2+|\xip|^2)^\frac{1}{2},$$ 
 on the support of $\psis$ one obtains the existence of $\sigma_0$ sufficiently small depending on the coefficients of $\cnop$ and of $\phi$ such that for all $0<\sigma \leq \sigma_0$:
\begin{equation}
\label{ellipticity}
C \tau^4-r^2_+ \geq \lambda_\tau^4,
\end{equation}
on the support of $\psis$. We consider now $\tilde{\psi}\in C^{\infty}(\R^{+})$ with $\tilde{\psi}=1$ in $[1/3,\infty)$ and $\tilde{\psi}=0$ in $[0,1/4]$ and then define $\tilde{\psi}_{\sigma}$ similarly to $\psis$:
$$
\tilde{\psi}_{_\sigma}(\tau,\xi_t):=\tilde{\psi}\left(\frac{\sigma \tau}{(1+|\xi_t|^2+ |\xip|^2)^{\frac{1}{2}}}\right) \in \mathcal{S}^0_{\tau}.
$$
We write:
\begin{align}
\label{aux ineq 2}
    \left((C\tau^4-R^*_{+}R_{+})v_+,v_+\right)&=\left((C\tau^4-R^*_{+}R_{+})\op(\tilde{\psi}_{\sigma})v_+,v_+\right)_+ +\left((C\tau^4-R^*_{+}R_{+})(\op(1-\tilde{\psi}_{\sigma})v_+,v_+\right)_+.
\end{align}
Observe that $\tilde{\psi}_{\sigma}=1$ on the support of $\psis$ and that $\tau \lesssim |\xip|+ |\xi_t|$ on the support of $1- \tilde{\psi}_{\sigma}$ and this gives thanks to Lemma~\ref{lemma with different classes} that in fact  $\op(1-\tilde{\psi}_{\sigma})\op(\psis) \in \psitt^{-\infty}$. 

As a consequence, we have in particular:
$$
\left|\left((C\tau^4-R^*_{+}R_{+})(\op(1-\tilde{\psi}_{\sigma})v_+,v_+\right)_+\right| \lesssim \norm{u_+}{\lrpn}{2}.
$$
We consider now 
$$
\tilde{a}(t,x,\xi_t,\xip,\tau)=(C \tau^4-r_+^2)\tilde{\psi}_{\sigma}+(1-\tilde{\psi}_{\sigma})\lambda^4_\tau \gtrsim\lambda^4_\tau,
$$
by construction of $\tilde{\psi}_{\sigma}$ and~\eqref{ellipticity}. We obtain therefore the following relation:
\begin{align*}
    \left((C\tau^4-R^*_{+}R_{+})\op(\tilde{\psi}_{\sigma})v_+,v_+\right)_+=\left(\op(\tilde{a})v_+,v_+\right)_+-\left(\op((1-\tilde{\psi}_{\sigma})\lambda^4_\tau) v_+,v_+\right)_+ +\left( Sv_+, v_+\right)_+,
\end{align*}
where $S\in\Psi^3_\tau$ is a subprincipal term coming from the pseudodifferential calculus. Indeed, remark that in fact $C\tau^4-R_+^*R_+ \in \mathcal{D}^4_\tau \subset\Psi^4_\tau.$  In particular one can control this term by
$$
\normsurf{\left( Sv_+,  v_+\right)_+}{}{} \lesssim  \norm{v_+}{L^2(\Rp; H^{3/2}_\tau)}{2}
$$
As before, using that $\op(1-\tilde{\psi}_{\sigma})\op(\psis) \in \psitt^{-\infty}$  one has 
$$
\left|\left(\op((1-\tilde{\psi}_{\sigma})\lambda^4_\tau )v_+,v_+\right)_+\right|\lesssim \norm{u_+}{\lrpn}{2}.$$
We use then the fact that $\tilde{a}\gtrsim\lambda^4_\tau$ which thanks to Gårding's inequality with a large parameter (Lemma~\ref{Garding with a large}) yields:
\begin{equation}
\label{aux ineq 3}
(\op(\tilde{a})v_+,v_+)_+ \gtrsim \norm{v_+}{L^2(\Rp; H^2_\tau)}{2}.	
\end{equation}
Putting \eqref{aux ineq 1}, \eqref{aux ineq 2}, \eqref{aux ineq 3} together we find, taking $\tau$ large enough 
\begin{equation*}
  \norm{\cnp  v_+}{\lrpn}{2}+\norm{u_+}{\lrpn}{2} \gtrsim \tau ^4 \norm{v_+}{\lrp}{2} + \tau^3 \normsurf{v_+}{\ls}{2}+ \tau \normsurf{D_nv_+}{\ls}{2}.
\end{equation*}
Since $v_+$ localizes in a region where $\tau \gtrsim |\xip|$ one can simply control the trace of the tangential derivatives (thanks also to the transmission condition \eqref{transmission cond1}):
\begin{align*}
    \tau^3 \normsurf{v_+}{\ls}{2}+T_{\theta,\Theta}\gtrsim \tau^3 \normsurf{v_\pm}{\ls}{2} =\tau \normsurf{ \tau \hat{v}_\pm}{\ls}{2}
    \gtrsim \tau \normsurf{\lambda \hat{v}_\pm}{\ls}{2}=\tau \normsurf{\nabla_{x^\prime}v_\pm}{\ls}{2}.
\end{align*}
This yields:
\begin{multline*}
  \norm{\cnp  v_+}{\lrpn}{2}+\norm{u_+}{\lrpn}{2}+T_{\theta,\Theta}\\ \gtrsim \tau ^4 \norm{v_+}{\lrp}{2} + \tau^3 \normsurf{v_+}{\ls}{2}+\tau \normsurf{\nabla_{x^\prime} v_-}{\ls}{2}+\tau \normsurf{\nabla_{x^\prime} v_+}{\ls}{2}+\tau \normsurf{D_nv_+}{\ls}{2}.
\end{multline*}
and finally thanks to Lemma~\ref{one of boundary derivatives cv}:
\begin{multline}
\label{estimate for A+ region tau grand}
  \norm{\cnp  v_+}{\lrpn}{2}+\norm{u_+}{\lrpn}{2}+\normsurf{u_-}{\ls}{2}+ T_{\theta,\Theta} \\ \gtrsim \tau ^4 \norm{v_+}{\lrpn}{2} + \tau^3 \normsurf{v_-}{\ls}{2}+\tau^3 \normsurf{v_+}{\ls}{2}+\tau \normsurf{\nabla v_-}{\ls}{2}+\normsurf{\nabla v_+}{\ls}{2}.
\end{multline}
The only missing term is the volume norm on the negative half-line. Here one needs to use $\cnm$. We write 
   $$
	\invcm(x) P_\phi ^-=(D_n+ i \tau \phi^\prime )^2+Q(x,D_{x^\prime})-\invcm D_t^2=A_- +R_-.
	$$
In the region under consideration $\norm{R_- v}{\lrmn}{}$ is a perturbation of $\norm{A_- v}{\lrmn}{}$. The difference is that here we integrate on the negative half-line and the boundary terms come with the opposite sign when one calculates $\norm{A_- v}{\lrmn}{2}$. More precisely, after repeating the same steps as above one finds
\begin{multline}
\label{estimate for A- region tau grand}
  \norm{\cnp  v_-}{\lrmn}{2}+\norm{u_-}{\lrmn}{2}+ \tau^3 \normsurf{v_-}{\ls}{2}\\+\tau^3 \normsurf{v_+}{\ls}{2}+\tau \normsurf{\nabla v_-}{\ls}{2} +T_{\theta,\Theta}\gtrsim \tau ^4 \norm{v_-}{\lrmn}{2} .
\end{multline}
To finish the proof of the Lemma one can simply multiply \eqref{estimate for A+ region tau grand} by a sufficiently large constant and add it to \eqref{estimate for A- region tau grand}.
\end{proof}

\subsection{Microlocal estimates in the non-elliptic regions }
\label{microlocal estimates in the non elliptic region}
In \eqref{def of E} and \eqref{def of GH} we have defined two microlocal regions on each side of the interface. In this section we prove microlocal estimates inside the non elliptic regions. As expected most surface terms are estimated by the positive half-space where the observation takes place.

Lemma~\ref{prop for tau grand} proves the appropriate estimate in the sub-region $\tau \gg |\xi_t|+|\xip|$. We can consequently localize in its complementary region in the sequel. The following lemma deals with the microlocal sub-regions where the operator $\cnop$ is \textbf{not} elliptic microlocally. In this case the error terms in $\xi_t$ become very useful.

\begin{lem}[\textbf{Non-elliptic positive half-space}]
\label{region G+ cv}
Let $K \subset \R^{n+1}$ be a compact set, $c_0>0$, $Y \subset \{|\xi_t| \geq c_0 |\xip| \}$. Consider $\chi \in \mathcal{S}^{0}$ with $\supp(\chi) \subset Y$ and $\pi_{t,x}(\supp{\chi}) \subset K $. Then for all $\sigma>0$ there exists $\tau_0>0$ such that one has:
\begin{multline*}
  \norm{P^{+}_{\phi}v}{\lrpn}{2}  +\tau \norm{ D_t v}{\lrpn}{2}+\tau \left\vert D_t v  \right\vert^2_{L^2(\Sigma)}+\tau\norm{u}{\lrpn}{2} + \tau \left\vert  u  \right\vert^2_{L^2(\Sigma)}
			 \\ \gtrsim
			\tau \norm{ v}{L^2(\R_+ ; H^1_\tau)}{2}
			+\tau^3 \normsurf{v}{\ls}{2}  +\tau \left\vert \nabla v  \right\vert_{L^2(\Sigma)}^2 ,
\end{multline*}
for all $\tau\geq \tau_0$, $u \in \mathscr{S}_c(\R^{n+1}) $ and $v= \op(\chi)\op(1-\psis)u$.
\end{lem}

\begin{proof}
Here we are microlocally in a region where $\xi_t$ is large. We consider as before an auxiliary function $\check{\psi}\in C^{\infty}(\R^{+})$ with $\check{\psi}=1$ in $[4,\infty)$ and $\check{\psi}=0$ in $[0,3]$ and then define $\check{\psi}_{\sigma}$ similarly to $\psis$:
$$
\check{\psi}_{_\sigma}(\tau,\xi_t):=\check{\psi}\left(\frac{\sigma \tau}{(1+|\xi_t|^2+ |\xip|^2)^{\frac{1}{2}}}\right) \in \mathcal{S}^0_{\tau}.
$$

To simplify notation we consider additionally
$\theta_\sigma:=1-\psis$ and $\check{\theta}_{\sigma}:=1-\check{\psi}_{\sigma}$. Remark then $\theta_\sigma$ and $\tths$ localize in a region where $|\xi_t|+|\xip| \gtrsim \tau$ with moreover $\tths=1$ on $\supp{\ths}$. We introduce also $\tchi$ satisfying the same properties as $\chi$ with $\tchi=1$ on $\supp{\chi}$. 

Observe now that on the one hand one has $|\xi_t|+|\xip| \gtrsim \tau$ on the support of $\tths$, and on the other hand $|\xi_t| \gtrsim |\xip|$ on the support of $\tchi$. Consequently $\tths \tchi$ localizes in the regime of Lemma~\ref{regime favor cv} and belongs to $\mathcal{S}^0_\tau$ thanks to Lemma~\ref{lemma with different classes}. This yields the estimate:
\begin{multline}
\label{estimate almost all terms region E+ with op}
\tau \norm{  D_t\op(\tths \tchi)v}{L^2(\R^{n+1}_{+})}{2} +\norm{v}{L^2(\R^{n+1}_{+})}{2}+ \normsurf{v}{\ls}{2} 
			+\tau \left\vert D_t\op(\tths \tchi)v  \right\vert^2_{L^2(\Sigma)}\\\gtrsim
			\tau^3 \norm{\op(\tths \tchi)v}{L^2(\R^{n+1}_{+})}{2}
			+\tau^3 \left\vert \op(\tths \tchi)v \right\vert_{L^2(\Sigma)}^2\\+\tau\norm{\nabla_{x^\prime}\op(\tths \tchi)v}{\lrpn}{2}+\tau \normsurf{\nabla_{x^\prime}\op(\tths \tchi)v}{\ls}{2}.
\end{multline}
One then has:
\begin{align*}
   \tau^{\frac{3}{2}} \norm{\op(\tths \tchi) v}{\lrpn}{}&=\tau^{\frac{3}{2}} \norm{v-\op(1-\tths \tchi)v}{\lrpn}{} \\
    &\geq\tau^{\frac{3}{2}}\norm{v}{\lrpn}{}- \tau^{\frac{3}{2}}\norm{ \op(1-\tths \tchi)v}{\lrpn}{} \\
    &= \tau^{\frac{3}{2}}\norm{v}{\lrpn}{}- \tau^{\frac{3}{2}} \norm{R u}{\lrpn}{},
\end{align*}
with $R=\op(1-\tths \tchi)\op(\chi) \op(\ths)$. Since 
$$\supp{(1-\tths \tchi)} \cap \supp{\ths} \cap \supp{\chi}=\emptyset,$$
one has immediately that $R \in \psit^{-\infty}$. Since $\ths \in \mathcal{S}^{0}_{\tau}$ has support in a region where $\tau \lesssim |\xi_t|+|\xip|$ Lemma~\ref{lemma with different classes} implies $R \in \psitt^{-\infty}$  and therefore 
$$\tau^{\frac{3}{2}} \norm{R u}{\lrpn}{} \lesssim \norm{u}{\lrpn}{}.$$ 
We control in a similar fashion all the terms in \eqref{estimate almost all terms region E+ with op} in which  $\op(\tths \tchi)$ appears. Taking then $\tau\geq\tau_0$,  $\tau_0$ large yields
\begin{multline}
\label{estimate almost all terms region G+ cv}
\tau \norm{  D_tv}{L^2(\R^{n+1}_{+})}{2}	+\tau \left\vert D_tv  \right\vert^2_{L^2(\Sigma)} +\norm{u}{L^2(\R^{n+1}_{+})}{2}+ \normsurf{u}{\ls}{2} 
		\gtrsim
			\tau^3 \norm{v}{L^2(\R^{n+1}_{+})}{2}
			+\tau^3 \normsurf{v}{\ls}{2}\\+
			\tau\norm{\nabla_{x^\prime}v}{\lrpn}{2}+\tau \normsurf{\nabla_{x^\prime}v}{\ls}{2}.
\end{multline}
The only term we need now is the trace of the normal derivative. Here we have no factorization and we will use the commutator technique  (see  \cite[Section 3A]{LR:95}, \cite[Chapter 3.4]{rousseau2022elliptic}). In our case we already control almost all of the terms thanks to \eqref{estimate almost all terms region G+ cv}.
We write
	$$
	\invcp(x) P^{+}_\phi=Q_2+ i \tau Q_1,
	$$
	with 
	\begin{equation}
	\label{def of Q2}
		\qr=D^2_n-\tau^2 |\phi^\prime |^2+Q(x,D_{x^\prime}) -\invcp(x) D_t^2  	
	\end{equation}
	and 
	\begin{equation}
	\label{def of Q1}
	\qi=\phi^\prime D_n +D_n \phi^\prime.	
	\end{equation}
We integrate by parts taking into account the boundary terms to find:
	\begin{equation}
	\label{equality with commutator cv }
	    \norm{\invcp P^{+}_{\phi} v}{\lrpn}{2}=\norm{\qr v}{\lrpn}{2}+ \tau^2\norm{\qi v}{\lrpn}{2}+ i \tau \left([\qr,\qi]v,v \right )_+ +\tau \mathcal{B}(v).
	\end{equation}
	After some calculations (which are the same as in \cite[Chapter 3.4]{rousseau2022elliptic}) we see that the  boundary term $\mathcal{B}(v)$ can be written as:
		\begin{align*}
	    \mathcal{B}(v)=2 \left(\phi^\prime D_n v ,D_n v\right)_{\Sigma}+\left(M_1 v ,D_n v \right)_{\Sigma}
	    +\left(M_1^{'}D_n v ,v\right)_{\Sigma}+\left(M_2 v ,v\right)_{\Sigma},
	\end{align*}
	where $M_j \in \mathcal{D}^j_{\top, \tau} $, that is a tangential differential operator of order $j$ depending on $\tau$ (see Section~\ref{differential operators} for a definition). Using this as well as the Young inequality yields, for arbitrary $\delta_2>0$:
	\begin{align*}
	|\mathcal{B}(v)-2  \left(\phi^\prime D_n v ,D_n v\right)_{\Sigma} | \lesssim  (1+\delta_2^{-1})\left( \tau^2 \normsurf{v}{\ls}{2}
+\normsurf{\nabla_{x^\prime}v}{\ls}{2}\right)+ \delta_2\normsurf{D_n v}{\ls}{2}.    
	\end{align*}
We want to combine this last inequality with \eqref{equality with commutator cv }. To do this we recall that $\phi^\prime >0$ close to $\Sigma$ (that is for $|x_n|$ small enough) and we choose $\delta_2$ sufficiently small. One finds then, for $\tau$ sufficiently large:
\begin{multline}
\label{estimate with commutator for G+ cv}
    \norm{ P^{+}_{\phi} v}{\lrpn}{2}+\tau^3 \normsurf{v}{\ls}{2}
+\tau \normsurf{\nabla_{x^\prime}v}{\ls}{2} 
\\ \gtrsim
    \norm{\qr v}{\lrpn}{2}+ \tau^2\norm{\qi v}{\lrpn}{2}+ i \tau \left([\qr,\qi]v,v \right )_++\tau \normsurf{D_n v}{\ls}{2} .
\end{multline}
We need to take care of the commutator to deduce an estimate on $\tau \normsurf{D_n v}{\ls}{2}$. We simply write:
\begin{equation*}
    i[\qr,\qi]=B_0 D^2_n + C_1 D_n +C_2,
\end{equation*}
with $C_j$ tangential operators of order $j$ depending on $\tau$. Since $\phi^\prime>0$ close to $\Sigma$ we can use \eqref{def of Q2} and \eqref{def of Q1} to express $D^2_n$ and $D_n$ in terms of $\qr$, $\qi$. This yields the following expression:
\begin{equation*}
    i[\qr,\qi]=B_0 \qr + B_1 \qi +B_2,
\end{equation*}
with $B_j$ tangential operators of order $j$ depending on $\tau$. This implies that:

\begin{equation*}
   \left| i \tau \left(B_0 \qr v,v\right)\right| \lesssim \tau^{-\frac{1}{2}} \norm{\qr v}{\lrpn}{2}+ \tau^{\frac{5}{2}} \norm{v}{\lrpn}{2},
\end{equation*}
as well as 
\begin{equation*}
 \left| i \tau \left(B_1 \qi v,v\right)\right| \lesssim \tau\norm{\qi v}{\lrpn}{2}+\tau^3\norm{v}{\lrpn}{2}+ \tau\norm{\nabla_{x^\prime}v}{\lrpn}{2}+\tau \norm{D_t v}{\lrpn}{2}
    \end{equation*}
    and
    \begin{equation*}
      \left| i \tau \left(B_2 v,v\right)\right| \lesssim \tau^3 \norm{v}{\lrpn}{2}+\tau \norm{\nabla_{x^\prime}v}{\lrpn}{2}+\tau \norm{D_t v}{\lrpn}{2}.
    \end{equation*}
We inject now these last three estimates in the commutator term in \eqref{estimate with commutator for G+ cv} and take $\tau \geq \tau_0$, $\tau_0$ large to absorb the terms with $\norm{\qr}{}{}$ and $\norm{\qi}{}{}$ to find:
\begin{multline}
\label{trace of normal deri for G+ cv}
\norm{ P^{+}_{\phi} v}{\lrpn}{2}+\tau^3 \norm{v}{\lrpn}{2}+\tau \norm{\nabla_{x^\prime}v}{\lrpn}{2}
+\tau^3 \normsurf{v}{\ls}{2}
+\tau \normsurf{\nabla_{x^\prime}v}{\ls}{2} 
\gtrsim \tau \normsurf{D_n v}{\ls}{2}
\end{multline}
We multiply finally \eqref{estimate almost all terms region G+ cv} by a large constant and add it to  \eqref{trace of normal deri for G+ cv} to obtain:
\begin{multline*}
      \norm{\cnp v}{\lrpn}{2} +\tau \norm{  D_tv}{L^2(\R^{n+1}_{+})}{2}	+\tau \left\vert D_tv \right\vert^2_{L^2(\Sigma)}
      +\norm{u}{L^2(\R^{n+1}_{+})}{2}+ \normsurf{u}{\ls}{2} \\
     \gtrsim 	\tau^3 \norm{v}{L^2(\R^{n+1}_{+})}{2}
			+\tau^3 \normsurf{v}{\ls}{2}  +\tau \left\vert D_nv \right\vert_{L^2(\Sigma)}^2+\tau \norm{\nabla_{x^\prime}v}{\lrpn}{2}+\tau \normsurf{\nabla_{x^\prime}v}{\ls}{2},
\end{multline*}
which gives immediately the sought estimate.
\end{proof}

\begin{lem}[\textbf{Non elliptic negative half-space}]
	\label{region G- cv}
Let $K \subset \R^{n+1}$ be a compact set, $c_0>0$, $Y \subset \{|\xi_t| \geq c_0 |\xip| \}$. Consider $\chi \in \mathcal{S}^{0}$ with $\supp(\chi) \subset Y$ and $\pi_{t,x}(\supp{\chi}) \subset K $. Then for all $\sigma>0$, there exists $\tau_0>0$ such that:
\begin{equation*}
\tau \norm{  D_tv}{L^2(\R^{n+1}_{-})}{2}	+\tau \left\vert D_tv  \right\vert^2_{L^2(\Sigma)} +\norm{u}{L^2(\R^{n+1}_{-})}{2}+ \normsurf{u}{\ls}{2} 
\gtrsim
\tau\norm{v}{L^2(\R_- ; H^1_\tau)}{2},
\end{equation*}
for all $\tau\geq \tau_0$, $u \in \mathscr{S}_c(\R^{n+1}) $ and $v= \op(\chi)\op(1-\psis)u$.
\end{lem}

\begin{proof}
We introduce, as in the proof of Lemma~\ref{region G+ cv} the function $\tths$ and $\tchi$ such that $\tths \tchi$ has a support contained in the region of Lemma~\ref{regime favor cv}, that is $\{\tau \gtrsim |\xi_t|+|\xip|\}$. We obtain as above the sought estimate.
\end{proof}

\subsection{Microlocal estimate in the elliptic region}
\label{section for elliptic region}

We now state the desired estimate in a microlocal region where the operator $\cnop$ is \textbf{elliptic}. In this case we adapt step by step the proof of \cite{le2013carleman}. Indeed, here we can factorize our operator as in \cite[Eq. 2-17]{le2013carleman} and obtain the same estimates for the first order factors. Notice that in this regime the error terms in $D_t$ in the left hand side are useless.

We recall that $\psis$ has been defined in \eqref{defintion of psi sigma}. The aim of this subsection is to prove the following lemma:

\begin{lem}[\textbf{Elliptic region}]
\label{Lemma elliptic region}
Let $K \subset \R^{n+1}$ be a compact set, $\chi \in \mathcal{S}^{0}$ with $\supp(\chi) \subset \mathcal{E}^-_\epsilon \cap \mathcal{E}^+_\epsilon  $ and $\pi_{t,x}(\supp{\chi}) \subset K $. Then for all  $\sigma>0$ there exist $\mathcal{V} \Subset V_t$ and $\tau_0>0$ such that:
\begin{align*}
    \norm{H_-\cnop v_+}{\lrn}{2}+ &\norm{H_+\cnop v_+}{\lrn}{2}+\norm{u}{\lrn}{2}+ \normsurf{u_+}{\ls}{2}+\normsurf{u_-}{\ls}{2}+T_{\theta,\Theta} \\
    &\gtrsim \tau\norm{v}{L^2(\R ; H^1_\tau)}{2}+\tau^3 \normsurf{v_+}{\ls}{2}+\tau^3 \normsurf{v_-}{\ls}{2}
    \\& \hspace{4mm}+\tau\normsurf{ \nabla v_+}{\ls}{2}+\tau\normsurf{ \nabla v_-}{\ls}{2}+\norm{\nabla v_+}{\lrpn}{2}+\norm{\nabla v_-}{\lrmn}{2},
\end{align*}
for all $\tau \geq \tau_0$ and $v_{\pm}=\op(1-\psis)\op(\chi)u_{\pm}$ with $u\in \tranphi$ satisfying $\supp u \subset \mathcal{V}$ .

\end{lem}

\begin{proof}[\unskip\nopunct]
To simplify we remove the $\pm$ notation from $v\pm$. Let $\tchi$ satisfy the same properties as $\chi$, satisfying additionally $\tchi=1$ on $\supp{\chi}$ and define $$s(x,\xi_t,\xip):= Q(x,\xip)-\invcp (x)\xi^2_t,$$ which means that $\invcp(x)\cnp=(D_n+ i \tau \phip)^2+\op(s)+\tilde{R}^+$, where $\tilde{R}^+ \in \Psi^1$. We write:
\begin{align}
\label{expression for factorization}
    \invcp(x)\cnp v
    =\left((D_n+ i \tau \phip)^2+\op(s \tchi)\right)v+\op((1-\tchi)s)\op(1-\psis)\op(\chi)u+\tilde{R}^+u.
\end{align}
The support condition on $\tchi$ guarantees that $\op((1-\tchi)s)\op(1-\psis)\op(\chi) \in \psitt^{-\infty}$. We define now the symbol 
\begin{equation}
\label{definition of m+}
m_{+}^2:=s \tchi+(1-\tchi)\lambda^2. 
\end{equation}
To justify the slightly abusive notation of $m^2_+$ we notice that by definition of the region $\mathcal{E}_{\epsilon}^+$ (see \eqref{def of E}) one has $s \gtrsim \epsilon \lambda^2,$ on the support of $\tchi$ and consequently 
$
m_{+}^2 \gtrsim \epsilon \lambda^2.
$ This means that $m_{+}^2 \in\mathcal{S}^2$ is elliptic positive and therefore it is indeed a square of another symbol in $\mathcal{S}^1$. We now write:
\begin{align*}
   \op(s \tchi)v=\op(m_{+}^2)v- \op((1-\tchi)\lambda^2)\op(1-\psis)\op(\chi)u,
\end{align*}
with $\op((1-\tchi)\lambda^2)\op(1-\psis)\op(\chi) \in \Psi^{-\infty}_{\tau} $. Coming back to ~\eqref{expression for factorization} we have obtained
$$
 \invcp(x)\cnp v=\left((D_n+ i \tau \phip)^2+\op(m_{+}^2)\right)v+\tilde{R}^+u+R^+u,
$$
with $R^+ \in \Psi^{-\infty}_\tau$. In particular $\norm{R^+u}{\lrpn}{2} \lesssim \norm{u}{\lrpn}{2},$ which can be absorbed in the left hand side of the sought estimate. Using the positive ellipticity of $m^2_+$ we define its square root which we denote by $m_+ \in\mathcal{S}^1$. Using symbolic calculus we obtain
\begin{align}
\label{microlocal factorization +}
 \invcp(x)\cnp v &=\left((D_n+ i \tau \phip)^2+M_+^2\right)v +R^+_1u \nonumber \\
         & = (D_n+ i \tau \phip -iM_{+})(D_n+ i \tau \phip+iM_{+})v +R^+_2u \nonumber \\
         & = (D_n+ i \tau \phip+iM_{+})(D_n+ i \tau \phip -iM_{+})v+R^+_3u ,
\end{align}
where $M_{+}=\op(m_{+}) \in \psit^1$ and the operators $R^+_j \in \psit^1+\tau \psit^0+ \psit^0 D_n$ which is an admissible perturbation (see Remark \ref{Carleman insensitive}). Similarly, using that $\supp(\chi) \subset \mathcal{E}^-_\epsilon \cap \mathcal{E}^+_\epsilon$ we may write:
\begin{align}
\label{microlocal factorization -}
\invcm(x)\cnm v &=\left((D_n+ i \tau \phip)^2+M_-^2\right)v +R^-_1u \nonumber \\
& = (D_n+ i \tau \phip -iM_{-})(D_n+ i \tau \phip+iM_{-})v +R^-_2u \nonumber \\
& = (D_n+ i \tau \phip+iM_{-})(D_n+ i \tau \phip -iM_{-})v+R^-_3u ,
\end{align}
where $M_-=\op(m_-)$ and $m_- \in \mathcal{S}^1$ is (similar to $m_+$) \textit{positive elliptic} and \textit{homogeneous} of degree one. Recall as well that $\op(1-\psis)$ localizes in a region where $\tau \lesssim |\xi_t|+|\xip|$. We can therefore suppose without loss of generality, up to introducing an admissible error in our estimate, that $m_\pm \in S^1_\tau$.

So far we have obtained microlocally a factorization as in \cite{le2013carleman} with the same weight $\phi$ and with operators $M_{\pm}$ having real, positive elliptic, homogeneous symbols of degree one. This is sufficient for the proof of the elliptic case to hold in our context too. Let us present the arguments of ~\cite{le2013carleman} in the present context.

\bigskip

We have two first-order factors on each side of the interface. Using the same notation as in \cite{le2013carleman} we write:
\begin{align*}
    &D_n+i(\tau \phip +M_{\pm}):=D_n+iE_\pm, \quad e_\pm=\tau \phip+m_\pm, \: E_\pm=\op(e_{\pm}) \\
    &D_n+i(\tau \phip -M_{\pm}):=D_n+iF_\pm, \quad f_\pm=\tau \phip-m_\pm, \: F_\pm=\op(f_{\pm}) .
\end{align*}
The crucial remark here is that one has always
$$
\mathcal{S}^1_\tau \ni e_\pm \gtrsim \lambda_\tau,
$$
thanks to the positive ellipticity of $m_\pm$, whereas the sign of $f_\pm$ may change. The quality of the estimate obtained depends on the ellipticity/positivity of the factors above, therefore one needs to distinguish cases concerning their sign. 

We collect now a series of lemmas giving some first order estimates. These correspond in the elliptic case to the estimates obtained in \cite[Sections from 3.2 to 3.7]{le2013carleman} and the proofs are similar.

For the elliptic factors $e_\pm$ one has:

\begin{lem}[\textbf{Positive factor on the positive half-space}]
Let $e_+ \in \mathcal{S}^1_\tau$ such that $e_+\gtrsim \lambda_\tau$ and $E_+=\op(e_+).$ Then for all $l\in \R$ there exists $\tau_0>0$ such that
   \begin{equation}
    \label{estimate for e+}
    \norm{(D_n+ i E_+)v}{L^2(\Rp; H^l_\tau)}{} \gtrsim \normsurf{v}{H^{l+1/2}_\tau}{}+\norm{v}{L^2(\Rp; H^{l+1}_\tau)}{}+\norm{D_n v}{L^2(\Rp; H^l_\tau)}{},
    \end{equation}
   for $\tau \geq \tau_0$ and $v \in \mathscr{S}_c(\R^{n+1}).$
\end{lem}

\begin{proof}
We sketch the proof for $l=0$.
An integration by parts in the $x_n$ variable yields:
\begin{align*}
2 \operatorname{Re}\left((D_n+i E_+)v,i \op(\lambda_\tau)v \right)_+= \normsurf{v}{H^{1/2}_\tau}{2}+ 2\operatorname{Re}\left(E_+v,\op(\lambda_\tau) v\right)_+,
\end{align*}
and then we simply use the fact that $e_+ \gtrsim \lambda_\tau$ as well as Gårding's inequality with a large parameter which implies:
$$
\left(E_+v,\op(\lambda_\tau) v\right)_+=\left(\op(\lambda_\tau)E_+v, v\right)_+ \gtrsim \norm{v}{L^2(\Rp; H^1_\tau)}{2},
$$
and \eqref{estimate for e+} without the term $\norm{D_n v}{}{}$ follows for $l=0$ from Young's inequality. To introduce $\norm{D_n v}{}{}$ in our estimate we simply use the expression 
$$ \norm{(D_n+ i E_+)v}{\lrpn}{2}=\norm{D_n v}{\lrpn}{2}+2 \operatorname{Re}(D_n v, i E_+v)_++\norm{E_+ v}{\lrpn}{2},$$
and we bound from above the last two terms by $\norm{v}{L^2(\Rp; H^{1}_\tau)}{2}$.
For general $l$ we just consider $$2 \operatorname{Re}\left((D_n+i E_+)v,i \op(\lambda^{2l+1}_\tau)v \right)_+$$ 
instead of $2 \operatorname{Re}\left((D_n+i E_+)v,i \op(\lambda_\tau)v \right)_+$.
\end{proof}

Estimate~\ref{estimate for e+} is of great quality as one could expect since we have an elliptic factor in the observation region. In the sequel, to alleviate notation, we sketch the proofs for $l=0$.

 \begin{lem}[\textbf{Positive factor on the negative half-space} ]
Let $e_- \in \mathcal{S}^1_\tau$ such that $e_-\gtrsim \lambda_\tau$ and $E_-=\op(e_-).$  Then for all $l\in \R$ we have that there exists $\tau_0>0$ such that
 \begin{equation}
    \label{estimate for e-}
    \norm{(D_n+ i E_-)v}{L^2(\Rm; H^l_\tau)}{}+\normsurf{v}{H^{l+1/2}_\tau}{} \gtrsim \norm{v}{L^2(\Rm; H^{l+1}_\tau)}{}+\norm{D_n v}{L^2(\Rm; H^l_\tau)}{},
    \end{equation}
 for $\tau \geq \tau_0$ and $v \in \mathscr{S}_c(\R^{n+1}) $.
 \end{lem}
 
 \begin{proof}
 The proof is the same as for \eqref{estimate for e+} but the boundary term comes with different sign.
 \end{proof}

Since the factors $f_\pm$ do not have a constant sign we need to consider appropriate microlocal regions.

\begin{lem}[\textbf{Positive $f_+$ on the positive half-space}]
Let $f_+ \in \mathcal{S}^1_\tau$ and $F_+=\op(f_+).$ We consider $c_0>0$ and $\chi=\chi(t,x,\xi_t,\xip,\tau) \in \mathcal{S}^{0}_\tau$ such that $f_+ \geq c_0 \lambda_\tau$ on $\supp{\chi}$. Then for all $l\in \R$ there exists $\tau_0>0$ such that
\begin{align}
\label{positive f+ }
    \norm{(D_n+i F_+)\op(\chi)v}{L^2(\Rp; H^l_\tau)}{}&+\norm{v}{\lrpn}{} \\ \nonumber  &\gtrsim \normsurf{\op(\chi)v}{H^{l+1/2}_\tau}{}+\norm{\op(\chi)v}{L^2(\Rp; H^{l+1}_\tau)}{}\\ \nonumber &\hspace{4mm}+ \norm{D_n \op(\chi)v}{L^2(\Rp; H^l_\tau)}{},
\end{align}
for $\tau \geq \tau_0$ and $v \in \mathscr{S}_c(\R^{n+1})$.
\end{lem}

\begin{proof}
The proof only uses the positive ellipticity of $f_+$ on the support of $\chi$ and is exactly the same as the proof of Lemma 3.3 in \cite{le2013carleman}. The (harmless) error term on the left hand side comes from the microlocalization.
\end{proof}

\begin{lem}[\textbf{Positive $f_-$ on the negative half-space}]
Let $e_- \in \mathcal{S}^1_\tau$ such that $e_-\gtrsim \lambda_\tau$ and $E_-=\op(e_-).$ Let $f_- \in \mathcal{S}^1_\tau$  and $F_-=\op(f_-).$ We consider $c_0>0$ and $\chi=\chi(t,x,\xi_t,\xip,\tau) \in \mathcal{S}^{0}_\tau$ such that $f_- \geq c_0 \lambda_\tau$ on $\supp{\chi}$. Then for all $l\in \R$ there exists $\tau_0>0$ such that
\begin{align}
\label{positive f- }
    &\norm{(D_n+i F_-)(D_n+iE_-)\op(\chi)v}{L^2(\Rm; H^l_\tau)}{}+\norm{v}{\lrmn}{}+\norm{D_n v}{\lrmn}{}\nonumber\\
   &\hspace{8mm}+\normsurf{(D_n+iE_-)\op(\chi)v}{H^{l+1/2}_\tau}{}  \gtrsim \norm{(D_n+iE_-)\op(\chi)v}{L^2(\Rm; H^{l+1}_\tau)}{},
\end{align}
for $\tau \geq \tau_0$ and $v \in \mathscr{S}_c(\R^{n+1})$.
\end{lem}

\begin{proof}
 The proof is as \cite[Lemma 3.4]{le2013carleman}. We set $u=(D_n+iE_-)\op(\chi)v$ and write
\begin{align*}
    2 \operatorname{Re}((D_n+i F_-)u,\op(\lambda_\tau)u)_-=-\normsurf{u}{H^{1/2}_\tau}{}+2\operatorname{Re}(F_-u,\op(\lambda_\tau)u)_-.
\end{align*}
 We then use the fact that $f_-\geq c_0 \lambda_t$ on the support of $\chi$ combined with Gårding's inequality with a large parameter.
\end{proof}

Even though one would expect that the crucial surface terms should come from the observation region $\R^{n+1}$ there is a situation in which the “good estimate” arises from the negative side $\R^{n+1}_-$. This is the content of the next estimate.

 \begin{lem}[\textbf{Negative $f_-$ on the negative half-space}]
 Let $e_- \in \mathcal{S}^1_\tau$ such that $e_-\gtrsim \lambda_\tau$ and $E_-=\op(e_-).$ Let $f_- \in \mathcal{S}^1_\tau$  and $F_-=\op(f_-).$ We consider $c_0$ and $\chi=\chi(t,x,\xi_t,\xip,\tau) \in \mathcal{S}^{0}_\tau$ such that $f_- \leq - c_0 \lambda_\tau$ on $\supp{\chi}$. Then there exists $\tau_0>0$ such that
\begin{align}
\label{negative f- }
&\norm{(D_n+i F_-)(D_n+iE_-)\op(\chi)v}{\lrmn}{}+\norm{v}{\lrmn}{}\nonumber\\ & \hspace{8mm}+\norm{D_n v}{\lrmn}{}\gtrsim \normsurf{(D_n+iE_-)\op(\chi)v}{H^{1/2}_\tau}{} + \norm{(D_n+iE_-)\op(\chi)v}{L^2(\Rm; H^1_\tau)}{},
\end{align}
for $\tau \geq \tau_0$ and $v \in \mathscr{S}_c(\R^{n+1})$.
\end{lem}

\begin{proof}
 This is the analogue of Lemma 3.5 in \cite{le2013carleman}. The proof is in the previous Lemma, here however the fact that $f_-$ is negative elliptic on the support of $\chi$ will allow us to control the crucial surface term.
\end{proof}

We finally give the estimate when $f_\pm$ has no constant sign. Here we exhibit an estimate with loss of a half derivative. Notice that in the estimates above, no assumption on the coefficients of the weight function $\phi$ was made. The sub-ellipticity (pseudoconvexity) assumption below \eqref{estim sous ell changins sign} is used in the next region. As in the classic case (with no interface) (see for example \cite[Chapter 28]{Hoermander:V4}) this translates into taking $\beta$ large enough. We sketch the proof. We recall that the weight $\phi$ is defined in~\eqref{def of weight phi}.

\begin{lem}[\textbf{Changing signs for $f_\pm$}]

Let $m_\pm  \in \mathcal{S}^1_\tau$ be real positive elliptic homogeneous symbols. Define $f_\pm=\tau \phip - m_\pm$ and $F_\pm=\op(f_\pm)$. Then for all $l\in \R$ there exist $\mathcal{V} \Subset V_t$ and $\tau_0>0$ such that
\begin{equation}
    \label{changin sign f}
\norm{(D_n+i F_\pm)v_\pm}{L^2(\R_\pm; H^l_\tau)}{}+\normsurf{v_\pm}{H^{l+1/2}_\tau}{} \gtrsim \tau^{-\frac{1}{2}}\left(\norm{v_\pm}{L^2(\R^{n+1}_\pm; H^{l+1}_\tau)}{}+\norm{D_n v_{\pm}}{L^2(\R^{n+1}_\pm; H^{l}_\tau)}{} \right),
\end{equation}
for $\tau \geq \tau_0$ and $v \in \mathscr{S}_c(\R^{n+1})$ with $\supp v \subset 
\mathcal{V}$.

\end{lem}

\begin{proof}
We do the proof for $l=0$. Since the symbol of $F_+ \in \psitt^1$ is real we have $F_+^*=F_+$, according to the Weyl quantization. We compute then
\begin{align}
\label{changing sign aux 1}
\norm{(D_n+iF_+)v}{\lrpn}{2}&=\norm{D_n v}{\lrpn}{2}+\norm{F_+ v}{\lrpn}{2}+(D_n v ,i F_+ v)_+ +(iF_+v,D_nv)_+  \nonumber \\
&\geq\norm{F_+ v}{\lrpn}{2}+(D_n v ,i F_+ v)_+ +(iF_+v,D_nv)_+   \nonumber\\
&=\operatorname{Re}\left((F_+^2v,v)_+ +i([D_n,F_+]v,v)_+ +(F_+v,v)_\Sigma  \right)  \nonumber\\
&\geq \operatorname{Re}\left(\tau^{-1}\mu(  F_+^2v,v)_+ +i([D_n,F_+]v,v)_+ +(F_+v,v)_\Sigma  \right)  \nonumber\\
&=\tau^{-1} \operatorname{Re}(\mu F_+^2+i \tau  [D_n,F_+]v,v)_++ \operatorname{Re}(F_+v,v)_\Sigma,
\end{align}
for all $\tau \geq \mu$. Now since $F_+ \in \psitt^1$ one has 
\begin{equation}
\label{changing sign aux 2}
|\operatorname{Re}i(F_+v,v)_\Sigma | \lesssim \normsurf{v}{H^{1/2}_\tau}{2}	
\end{equation}
which goes to the left hand side of \eqref{changin sign f}. We need therefore to estimate the term
$$
\operatorname{Re}(\mu F_+^2+i\tau [D_n,F_+]v,v)_+.
$$
The principal symbol of $\mu F_+^2+i\tau( [D_n,F_+]v,v)_+$ in the class $\psitt^2$ is $$\mu f_+^2+\tau \{\xi_n,f_+\}=\mu f_+^2+\tau\partial_{x_n}f_+ \in \mathcal{S}^2_\tau.$$
Notice in particular that the commutator kills the non tangential part. This is the point where the sub-ellipticity condition appears. Indeed, for $\beta$ sufficiently large one has that for $\mu$ sufficiently large,
\begin{equation}
\label{estim sous ell changins sign}
\mu f_+^2+\tau \{\xi_n,f_+\} \gtrsim \lambda_\tau^2.    
\end{equation}

To prove \eqref{estim sous ell changins sign} we start by noticing that there exists $\epsilon_1>0$ such that the following implication holds:
$$
|f_+| \leq \epsilon_1 \lambda_\tau \Longrightarrow \tau \sim \lambda=(|\xip|^2+|\xi_t|^2)^{\frac{1}{2}},
$$
where the notation $\tau \sim \lambda$ means that there exists $C>0$ such that $1/C \tau \leq \lambda \leq C \tau$.
The above implication is a consequence of the facts $f_+=\tau \phip- m_+$, $\phip\gtrsim 1$ and $m_+ \lesssim \lambda$. We distinguish two cases:

\begin{enumerate}
    \item Suppose that $|f_+| \leq \epsilon_1 \lambda_\tau$. Then the above implies $\tau \sim \lambda$ and we can compute as follows 
    $$
    \{\xi_n,f_+\}=\tau \beta-\partial_{x_n}m_+ \gtrsim \beta \lambda-\partial_{x_n}m_+.
    $$
Using homogeneity and compactness we obtain that for $\beta$ large enough one has 
$$
 \beta \lambda-\partial_{x_n}m_+ \gtrsim \lambda,
$$
and since  $\tau\sim \lambda$ this implies 
$
\tau    \{\xi_n,f_+\} \gtrsim \lambda_\tau^2,$ which implies \eqref{estim sous ell changins sign}.

\item If $|f_+| \geq \epsilon_1 \lambda_\tau$ then again by homogeneity and compactness one has for $\mu$ sufficiently large and all $\tau \geq \mu$ the sought estimate \eqref{estim sous ell changins sign}.
\end{enumerate}

One can then apply Gårding's inequality with a large parameter which gives
$$
\operatorname{Re}(\mu F_+^2+i\tau [D_n,F_+]v,v)_+ \gtrsim \norm{v}{L^2(\Rp ; H^1_\tau)}{2}.
$$
Combined with \eqref{changing sign aux 1} and \eqref{changing sign aux 2}, this proves inequality \eqref{changin sign f} in the positive half-space. The proof for the negative half-space is exactly the same.
\end{proof}

We have presented the first order estimates that we will iterate. In order for the iteration to work, we need to impose a condition on the coefficients of the weight function $\phi$. The crucial assumption that we use is the following \textbf{geometric hypothesis} of \cite{le2013carleman}. We choose the coefficients $\alpha_-$, $\alpha_+$ such that:

\begin{equation*}
\frac{\alpha_+}{\alpha_-} > \sup_{\substack{t,x^\prime,\xi_t,\xip \\ |(\xi_t,\xip)|\geq 1}} \frac{m_+(t,0,x^\prime,\xi_t,\xip)}{m_-(t,0,x^\prime,\xi_t,\xip)},
\end{equation*}
where the symbols $m_\pm$ defined above (see \eqref{ellipticity}) are real positive elliptic and homogeneous of degree one. 
This geometric hypothesis can be stated as well by saying that:
\begin{equation}
\label{geometric assumption}
\exists \: \mu>1, \quad     
\frac{\alpha_+}{\alpha_-} = \mu^2 \sup_{\substack{t,x^\prime,\xi_t,\xip \\ |(\xi_t,\xip)|\geq 1}} \frac{m_+(t,0,x^\prime,\xi_t,\xip)}{m_-(t,0,x^\prime,\xi_t,\xip)}.
\end{equation}

We now recall how the geometric assumption on the weight \eqref{geometric assumption} allows to effectively combine all of the above estimates. This is expressed through the following lemma. 

\begin{lem}
\label{regions gamma}
Let $\mu_0>\mu>1$ and $\alpha_\pm$, be positive numbers such that \eqref{geometric assumption} holds. For $s>0$ define the following subsets of $\R^{n}_{x^\prime,t} \times \R^{n}_{\xi_t,\xip} \times \R^*_+$ by:
\begin{align*}
    \Gamma_s &=\{(t,x^\prime, \xi_t,\xip,\tau); \: |(\xip, \xi_t)|<2\textnormal{ or } \tau \alpha_+>s m_+(t,0,x^\prime,\xi_t,\xip),\\
    \tilde{\Gamma}_s &=\{(t,x^\prime, \xi_t,\xip,\tau); \: |(\xip, \xi_t)|>1\textnormal{ and } \tau \alpha_+<s m_+(t,0,x^\prime,\xi_t,\xip)\}.
    \end{align*}
Then there exists $\eta, \tau_0>0$ such that for $|x_n|\leq \eta $ and $\tau\geq \tau_0$ we have $\R^{n}_{x^\prime,t} \times \R^{n}_{\xi_t,\xip} \times \R^*_+=\Gamma_{\mu_0} \cup \tilde{\Gamma}_\mu$ and
\begin{align*}
    \Gamma_{\mu_0} &\subset \{(t,x^\prime, \xi_t,\xip,\tau) \in \R^n\times \R^n \times \R^*_+; \: f_+( t,x,\xi_t, \xip) \geq C \lambda_ \tau  \textnormal{ for $0\leq x_n\leq \eta$}\} \},\\
    \tilde{\Gamma}_\mu &\subset \{(t,x^\prime, \xi_t,\xip,\tau) \in \R^n\times \R^n \times \R^*_+; \: f_-( t,x,\xi_t, \xip) \leq -C \lambda_ \tau  \textnormal{ for $-\eta \leq x_n\leq 0$}\}.
\end{align*}
\end{lem}
The proof of this Lemma uses the ellipticity and homogeneity of $m_\pm$ and is exactly the same as in Section 4A of \cite{le2013carleman}.

The end of the proof is also similar to \cite{le2013carleman}. Let us sketch it here for the sake of completeness. The crucial property of Lemma~\ref{regions gamma} is that we have covered the tangential dual space by two regions such that $f_+$ is positive elliptic on the one region and $f_-$ is negative elliptic on the other. We consider a homogeneous partition of unity
$$
1=\chi_{\Gamma,0}(t,x^\prime, \xi_t,\xip ,\tau)+\chi_{\Gamma,1}(t,x^\prime, \xi_t,\xip ,\tau), \: \supp{\chi_{\Gamma,0}} \subset\Gamma_{\mu_0}, \: \supp{\chi_{\Gamma,1}} \subset \tilde{\Gamma}_\mu.
$$
Remark that the derivatives of $\chi_{\Gamma,j}$ are supported in a region where $\tau \lesssim |\xip|+\xi_t$ and consequently one has $\chi_{\Gamma,j} \in \mathcal{S}^0_\tau$.

We now prove the estimate of Lemma \ref{Lemma elliptic region} in the interior of each of the microlocal sub-regions given by Lemma \ref{regions gamma}. After this, the last step will be to put together these two microlocal estimates.

\bigskip 

\textbf{Estimate in $\Gamma_{\mu_0}$}

We consider $\Xi_0=\op(\chi_{\Gamma,0}) \in \Psi^0_\tau$. We start by applying \eqref{estimate for e+} to get:
\begin{align*}
  \norm{  (D_n+iE_+)(D_n+iF_+)\xmo v_+}{\lrpn}{}\gtrsim \normsurf{(D_n+iF_+)\xmo v_+}{H^{1/2}_\tau}{}+\norm{(D_n+iF_+)\xmo v_+}{L^2(\Rp; H^1_\tau)}{}.
\end{align*}
The localization of $\chi_{\Gamma,0}$ allows to use Lemma \eqref{positive f+ } which gives
$$
\norm{(D_n+iF_+)\xmo v_+}{L^2(\Rp; H^1_\tau)}{}+\norm{v_+}{L^2}{} \gtrsim \normsurf{\xmo v_+}{H^{3/2}_\tau}{}+\norm{\xmo v_+}{L^2(\Rp; H^2_\tau)}{}+\norm{D_n \xmo v_+}{L^2(\Rp; H^1_\tau)}{}.
$$
We obtain therefore:
\begin{align*}
&\norm{  (D_n+iE_+)(D_n+iF_+)\xmo v_+}{\lrpn}{}+\norm{v_+}{L^2}{} \\
&\hspace{4mm}\gtrsim \normsurf{(D_n+iF_+)\xmo v_+}{H^{1/2}_\tau}{}+\normsurf{\xmo v_+}{H^{3/2}_\tau}{}+\norm{\xmo v_+}{L^2(\Rp; H^2_\tau)}{}+\norm{D_n \xmo v_+}{L^2(\Rp; H^1_\tau)}{}.
\end{align*}
Using the fact that 
$$
\normsurf{(D_n+iF_+)\xmo v_+}{H^{1/2}_\tau}{}\geq \normsurf{D_n \xmo v_+}{H^{1/2}_\tau}{}-\normsurf{F_+ \xmo v_+}{H^{1/2}_\tau}{}\geq  \normsurf{D_n \xmo v_+}{H^{1/2}_\tau}{}-C \normsurf{\xmo v_+}{H^{3/2}_\tau}{},
$$
we finally get
\begin{align}
\label{positive half line region gamma one}
&\norm{  (D_n+iE_+)(D_n+iF_+)\xmo v_+}{\lrpn}{}+\norm{v_+}{L^2}{} \\ \nonumber
&\hspace{4mm}\gtrsim \normsurf{D_n \xmo v_+}{H^{1/2}_\tau}{}+\normsurf{\xmo v_+}{H^{3/2}_\tau}{}+\norm{\xmo v_+}{L^2(\Rp; H^2_\tau)}{}+\norm{D_n \xmo v_+}{L^2(\Rp; H^1_\tau)}{}.
\end{align}
For the negative half-space we start by using \eqref{changin sign f} for $l=1/2$ which yields:
\begin{align*}
  \norm{(D_n+iF_-)(D_n+iE_-)\xmo v_-}{\lrmn}{}+\normsurf{(D_n+iE_-)\xmo v_-}{L^2}{}\gtrsim \norm{(D_n+iE_-)\xmo v_-}{L^2(\Rm ; H^{1/2}_\tau)}{}.
\end{align*}
Estimate \eqref{estimate for e-} gives:
\begin{align*}
  \norm{(D_n+iE_-)\xmo v_-}{L^2(\Rm ; H^{1/2}_\tau)}{}  +\normsurf{\xmo v_- }{H^{1}_\tau}{}\gtrsim  \norm{\xmo v_-}{L^2(\Rm; H^{3/2}_\tau)}{}+\norm{\xmo D_n v_-}{L^2(\Rm ; H^{1/2}_\tau)}{},
\end{align*}
and consequently
\begin{align}
\label{negative half line region gamme one}
  \norm{(D_n+iF_-)(D_n+iE_-)\xmo v_-}{\lrmn}{}&+\normsurf{D_n \xmo v_-}{L^2}{}+ \normsurf{\xmo v_- }{H^{1}_\tau}{}\\ \nonumber
  &\gtrsim   \norm{\xmo v_-}{L^2(\Rm; H^{3/2}_\tau)}{}+\norm{\xmo D_n v_-}{L^2(\Rm ; H^{1/2}_\tau)}{}.
\end{align}
We can now multiply \eqref{positive half line region gamma one} by a large constant and add it to \eqref{negative half line region gamme one}, take $\tau$ large and use the transmission conditions to find:
\begin{align}
\label{estimate region gamma one}
\norm{\cnop^+ \xmo v_+}{\lrpn}{}&+\norm{\cnop^- \xmo v_-}{\lrmn}{}+T_{\theta,\Theta}^{1/2}+\tau^{1/2}\normsurf{v_+}{\ls}{}+\norm{v_+}{\lrpn}{}\\ \nonumber
&\gtrsim\tau^{1/2}\left(\normsurf{\xmo D_n v_-}{\ls}{}+\normsurf{\xmo D_n v_+}{\ls}{}+ \normsurf{\xmo v_- }{H^1_\tau}{} +\normsurf{\xmo v_+}{H^1_\tau}{} \right)\\ \nonumber
&\hspace{4mm}+\norm{\xmo v}{L^2(\R; H^{3/2}_\tau)}{}+\norm{\xmo D_n v_-}{L^2(\R; H^{1/2}_\tau)}{}+\norm{\xmo D_n v_+}{L^2(\R; H^{3/2}_\tau)}{}.
\end{align}

This gives the desired estimate in the microlocal region $\Gamma_{\mu_0}$.

\bigskip

\textbf{Estimate in $\tilde{\Gamma}_\mu$}

We consider $\xm=\op(\chi_{\Gamma,1}) \in \Psi^{0}_\tau$. We apply estimate \eqref{estimate for e+} which gives:
\begin{align}
\label{aux for gamma tilde 1}
  \norm{  (D_n+iE_+)(D_n+iF_+)\xm v_+}{\lrpn}{}\gtrsim \normsurf{(D_n+iF_+)\xm v_+}{H^{1/2}_\tau}{}+\norm{(D_n+iF_+)\xm v_+}{L^2(\Rp; H^1_\tau)}{}.
\end{align}
Thanks to the localization of $\chi_{\Gamma,1}$ we can use the estimate \eqref{negative f- } for the negative half-space:
\begin{align}
\label{aux for gamma tilde 2}
 \norm{ (D_n+iF_-)(D_n+iE_-)\xm v_-}{\lrpn}{} &+\norm{v_-}{\lrmn}{}+\norm{D_n v_-}{\lrmn}{} \nonumber\\
 &\hspace{-8mm}\gtrsim \normsurf{(D_n+iE_-)\xm v_-}{H^{1/2}_\tau}{} +\norm{(D_n+iE_-) \xm v_-}{L^2(\Rm; H^1_\tau)}{}.
\end{align}
 Estimates \eqref{aux for gamma tilde 1} and \eqref{aux for gamma tilde 2} imply in particular that we control:
 \begin{align*}
 c_+ \norm{\cnop \xm v_+}{\lrpn}{}&+c_-\norm{\cnop \xm v_+}{\lrmn}{}+\norm{v_-}{\lrmn}{}+\norm{D_n v_-}{\lrmn}{}\\&\gtrsim \tau^{1/2}\left( \normsurf{c_+(D_n+iF_+)\xm v_+}{\ls}{}+ \normsurf{c_-(D_n+iE_-)\xm v_-}{\ls}{}\right).
 \end{align*}
Using the transmission conditions \eqref{transmission cond1}, \eqref{transmission cond2} as well as the triangle inequality we find:
\begin{align*}
\tau^{1/2}\bigg(\normsurf{c_+(D_n+iF_+)\xm v_+}{\ls}{} &+ \normsurf{c_-(D_n+iE_-)\xm v_-}{\ls}{}\bigg) \\
&\geq \tau^{1/2} \normsurf{(c_-M_-+c_+M_+)\xm v_+}{\ls}{}-C \left(T_{\theta,\Theta}^{1/2}+\tau^{1/2}\normsurf{v_+}{\ls}{} \right).
\end{align*}
We now use the positive ellipticity of $M_\pm$ combined with Lemmata~\ref{Garding with a large} and~\ref{lemma with different classes} to obtain:
\begin{align*}
    \normsurf{(D_n+iE_-)\xm v_-}{H^{1/2}_\tau}{} &+\norm{(D_n+iE_-) \xm v_-}{L^2(\Rm; H^1_\tau)}{}+T_{\theta,\Theta}^{1/2}+\tau^{1/2}\normsurf{v_+}{\ls}{} \\
    &\hspace{8mm} \gtrsim \normsurf{\xm v_- }{H^1_\tau}{}+\normsurf{\xm v_+ }{H^1_\tau}{}.
\end{align*}
Hence we control
\begin{align*}
     &c_+ \norm{\cnop \xm v_+}{\lrpn}{}+c_-\norm{\cnop \xm v_+}{\lrmn}{}+\norm{v_-}{\lrmn}{}+\norm{D_n v_-}{\lrmn}{}\\
     &+T_{\theta,\Theta}^{1/2}+\tau^{1/2}\normsurf{v_+}{\ls}{} \gtrsim \tau^{1/2}\bigg(\normsurf{\xm v_- }{H^1_\tau}{}+\normsurf{\xm v_+ }{H^1_\tau}{}+\normsurf{\xm D_n v_-}{\ls}{}+\normsurf{\xm D_n v_+}{\ls}{} \\&\hspace{8mm}+\norm{(D_n+iF_+)\xm v_+}{L^2(\Rp; H^{1/2}_\tau)}{}
     +\norm{(D_n+iE_-)\xm v_+}{L^2(\Rp; H^{1/2}_\tau)}{} \bigg).
\end{align*}
From here one can proceed as for the region $\Gamma_{\mu_0}$ by using estimate \eqref{estimate for e-} for the term $$\norm{(D_n+iE_-)\xm v_+}{L^2(\Rp; H^{1/2}_\tau)}{}$$ and estimate \eqref{changin sign f} for $$\norm{(D_n+iF_+)\xm v_+}{L^2(\Rp; H^{1/2}_\tau)}{}.$$

We finally obtain the desired estimate microlocalized in $\tilde{\Gamma}_\mu$:
\begin{align}
\label{estimate region gamma two}
\norm{\cnop^+ \xm v_+}{\lrpn}{}&+\norm{\cnop^- \xm v_-}{\lrmn}{}+T_{\theta,\Theta}^{1/2}+\tau^{1/2}\normsurf{v_+}{\ls}{}+\norm{v_+}{\lrpn}{}\\ \nonumber
&+\norm{D_n v_-}{\lrmn}{}\gtrsim\tau^{1/2}\left(\normsurf{\xm D_n v_-}{\ls}{}+\normsurf{\xm D_n v_+}{\ls}{}+ \normsurf{\xm v_- }{H^1_\tau}{} +\normsurf{\xm v_+}{H^1_\tau}{} \right)\\ \nonumber
&\hspace{4mm}+\norm{\xm v}{L^2(\R; H^{3/2}_\tau)}{}+\norm{\xm D_n v_-}{L^2(\R; H^{1/2}_\tau)}{}+\norm{\xm D_n v_+}{L^2(\R; H^{3/2}_\tau)}{}.
\end{align}

\bigskip

\noindent \textit{End of the proof of Lemma~\ref{Lemma elliptic region}.} To finish the proof of Lemma~\ref{Lemma elliptic region} one has to add estimates \eqref{estimate region gamma one} and \eqref{estimate region gamma two}. This yields the estimate:
\begin{align*}
    \norm{\cnop \xmo v}{\lrn}{}&+\norm{\cnop \xm v}{\lrn}{}+T_{\theta,\Theta}^{1/2}+\tau^{1/2}\normsurf{v_+}{\ls}{}+\norm{v}{\lrn}{}+\norm{D_n v_-}{\lrmn}{} \\
    &\gtrsim \sum_{j=0,1}\tau^{1/2}\left(\normsurf{\Xi_{j} D_n v_-}{\ls}{}+\normsurf{\Xi_{j} D_n v_+}{\ls}{}+ \normsurf{\Xi_{j} v_- }{H^1_\tau}{} +\normsurf{\Xi_{j} v_+}{H^1_\tau}{} \right)\\ \nonumber
&\hspace{4mm}+\norm{\Xi_{j} v}{L^2(\R; H^{3/2}_\tau)}{}+\norm{\Xi_{j} D_n v_-}{L^2(\R; H^{1/2}_\tau)}{}+\norm{\Xi_{j} D_n v_+}{L^2(\R; H^{3/2}_\tau)}{}.
\end{align*}

The right hand side can be estimated from below by simply using the triangle inequality as well as the fact that $\xmo+\xm=\textnormal{Id}$. To bound from above the left hand side we argue with commutators noticing that $\cnop^{\pm} \in \Psi^2_\tau$, $\xmo,\xm \in \Psi^0_\tau$ and therefore 
$
[\cnop^{\pm},\Xi_{j}] \in \Psi^1_\tau.
$
This, combined with Sobolev regularity of the pseudo differential calculus allows to estimate as follows:
\begin{align*}
    \norm{\cnop^{\pm}\Xi_{j}v_\pm}{L^2(\R^{n+1}_\pm)}{}&\leq \norm{\Xi_{j}\cnop^{\pm}v_\pm}{L^2(\R^{n+1}_\pm)}{}+\norm{[\cnop^{\pm},\Xi_{j}]v_\pm}{L^2(\R^{n+1}_\pm)}{} \\
    &\lesssim \norm{\cnop^{\pm}v_\pm}{L^2(\R^{n+1}_\pm)}{}+\norm{v_\pm}{L^2(\R_\pm;H^1_\tau)}{}.
\end{align*}
We take $\tau\geq \tau_0$, $\tau_0$ large to absorb the error term $\norm{v_\pm}{L^2(\R_\pm;H^1_\tau)}{}$. This concludes the proof of Lemma~\ref{Lemma elliptic region}. 
\end{proof}

\subsection{Patching together microlocal estimates: End of the proof of Proposition~\ref{theorem bis}}

\begin{proof}[End of the proof of Proposition~\ref{theorem bis}]

We are now ready to finish the proof of the sub-elliptic estimate for $P$. We have considered above two regions on each side of the interface (given in our local coordinates by $\Sigma=\{x_n=0\}$). This yields the following covering of $\R \times \R^{n} \times \R \times \R^{n-1} \ni (t,x,\xi_t,\xip)$:
\begin{align*}
	\R \times \R^{n} \times \R \times \R^{n-1}= \left(\mathcal{E}_{\epsilon}^{-}\cap \mathcal{E}_{\epsilon}^{+}\right) \cup Y.
	\end{align*}
	where we have defined
	$$
	Y:= \left(\mathcal{E}_{\epsilon}^{-} \cap \mathcal{GH}_{\epsilon}^{+}\right) 
	\cup \left(\mathcal{E}_{\epsilon}^{+}\cap \mathcal{GH}_{\epsilon}^{-} \right)
	\cup \left( \mathcal{GH}_{\epsilon}^{-}\cap \mathcal{GH}_{\epsilon}^{+} \right).
	$$
Let us recall that the regions above have been defined in \eqref{def of E} and \eqref{def of GH}.	
The crucial remark here is that the definition of our microlocal regions implies that $Y \subset \mathcal{GH}_{\epsilon}^{-} \cup \mathcal{GH}_{\epsilon}^{+} $ so that
$$
|\xi_t| \gtrsim |\xip| \quad \textnormal{on} \: Y.
$$
Notice that the definition of the conic regions above imply in particular that $Y$ and $\mathcal{E}_{\epsilon}^{-}\cap \mathcal{E}_{\epsilon}^{+}$ overlap due to the factor $2\epsilon$ in the definition of $\mathcal{GH}_{\epsilon}^{\pm}$. That means that we can consider an associated homogeneous partition of unity. More precisely, given a compact set $K$ of $\R^{n+1}$ we can introduce $\chi_j \in \mathcal{S}^0$ homogeneous, $\supp{\chi_1} \cap \{|\xip|+|\xi_t|\geq 1\}\subset \mathcal{E}_{\epsilon}^{-}\cap \mathcal{E}_{\epsilon}^{+}$, $\supp{\chi_2} \cap \{|\xip|+|\xi_t|\geq 1\}\subset Y$~\footnote{In fact one has to construct the partition of unity on the cosphere bundle. That is the reason why we excluded a compact set which includes the zero section. However this does not pose any problem since a function localized on a compact subset of the phase space yields a residual operator and does not have any impact on our estimates.  }, $\pi_{t,x}(\supp \chi_j) \subset K$ and 
	$$
 \chi_1+\chi_2=1.
	$$
We pick now an element $u$ of $\tranphi$ which is compactly supported and write 
$$
\op(1-\psis)u=v_1+v_2, \quad v_j:=\op(\chi_j)\op(1-\psis)u, j=1,2.
$$	
One needs now to simply put together the estimates already obtained according to the microlocalization of $v_1$ and $v_2$: 

\begin{itemize}
    \item We apply Lemma~\ref{Lemma elliptic region}, with $v_1=\op(\chi_1) \op(1-\psis)u$ and we obtain the desired estimate microlocalized in $\mathcal{E_{\epsilon}}^{-} \cap \mathcal{E_{\epsilon}}^{+}$:
\begin{multline}
    \label{region Y1}
    \norm{\cnop v_1}{\lrn}{2}+\norm{u}{\lrn}{2}+\normsurf{u}{\ls}{2} +T_{\theta, \Theta}\\ \gtrsim \tau \norm{v_1}{L^2(\R ; H^1_\tau)}{2}+\tau^3 \normsurf{v_1}{\ls}{2}
    +\tau\normsurf{ \nabla v^+_1}{\ls}{2}+\tau\normsurf{ \nabla v^-_1}{\ls}{2},
\end{multline}
where we write $v_1=H_-v^-_1+H_+v^+_1$.

\item We consider $v_2= \op(\chi_2) \op(1-\psis)u=H_-v_2^-+H_+v_2^+$. Since $\chi_2$ localizes in particular in a region where $|\xi_t| \gtrsim |\xip|$ we can apply Lemma~\ref{region G- cv} to $v_2^-$  to find:
\begin{multline}
\label{patch aux 1}
  \norm{P^{-}_{\phi}v_2^-}{\lrmn}{2}  +\tau \norm{ D_t v_2^-}{\lrmn}{2}+\tau \left\vert D_t v_2^-  \right\vert^2_{L^2(\Sigma)}+\norm{u}{\lrmn}{2} + \left\vert  u  \right\vert^2_{L^2(\Sigma)}
			 \\  \gtrsim
			\tau \norm{ v_2^-}{L^2(\R_- ; H^1_\tau)}{2}.
	\end{multline}
Thanks to the localization of $\chi_2$ we can apply Lemma~\ref{region G+ cv} as well to $v_2^+$
\begin{multline*}
  \norm{P^{+}_{\phi}v_2^+}{\lrpn}{2}  +\tau \norm{ D_t v_2^+}{\lrpn}{2}+\tau \left\vert D_t v_2^+  \right\vert^2_{L^2(\Sigma)}+\norm{u}{\lrpn}{2} + \left\vert  u  \right\vert^2_{L^2(\Sigma)}
			 \\ \gtrsim
			\tau \norm{ v_2^+}{L^2(\R_+ ; H^1_\tau)}{2}
			+\tau^3 \left\vert v_2^+  \right\vert_{L^2(\Sigma)}^2  +\tau \left\vert \nabla_{x} v_2^+  \right\vert_{L^2(\Sigma)}^2.
\end{multline*}
We use then Lemma~\ref{one of boundary derivatives cv} and~\eqref{transmission cond2} to control the trace of the derivative of $v_2^-$ to deduce
\begin{multline}
\label{patch aux 2}
\norm{P^{+}_{\phi}v_2^+}{\lrpn}{2}  +\tau \norm{ D_t v_2^+}{\lrpn}{2}+\tau \left\vert D_t v_2^+  \right\vert^2_{L^2(\Sigma)}+\norm{u}{\lrpn}{2} + \left\vert  u  \right\vert^2_{L^2(\Sigma)}+T_{\theta,\Theta}
			 \\ \gtrsim
			\tau \norm{ v_2^+}{L^2(\R_+ ; H^1_\tau)}{2}
			+\tau \left\vert v_2^+  \right\vert_{L^2(\Sigma)}^2  +\tau \left\vert \nabla_{x} v_2^+  \right\vert_{L^2(\Sigma)}^2+\tau \left\vert \nabla_{x} v_2^-  \right\vert_{L^2(\Sigma)}^2.
\end{multline}
We can then multiply as usual the above estimate \eqref{patch aux 2} by a large constant and add it to \eqref{patch aux 1} to obtain:
\begin{multline}
    \label{region Y2}
     \norm{\cnop v_2}{\lrn}{2} +\tau \norm{ D_t v_2^{\pm}}{\lrmn}{2}+\tau \left\vert D_t v_2^{\pm}  \right\vert^2_{L^2(\Sigma)}+\norm{u}{\lrmn}{2} + \left\vert  u  \right\vert^2_{L^2(\Sigma)} +T_{\theta,\Theta}
     \\\gtrsim \tau \norm{v_2}{L^2(\R ; H^1_\tau)}{2}+\tau^3 \normsurf{v_2}{\ls}{2}
    +\tau\normsurf{ \nabla v^+_2}{\ls}{2}+\tau\normsurf{ \nabla v^-_2}{\ls}{2}.
\end{multline}

\end{itemize}

 Summarising, we have thus shown, for $j \in \{1,2\} $:
\begin{multline*}
\norm{\cnop \op(\chi_j)\op(1-\psis)u}{\lrn}{2} +\tau \norm{ D_t \op(\chi_j)\op(1-\psis)u_{\pm}}{L^2(\R^{n+1}_{\pm})}{2}\\ \hspace{4mm}+\tau \left\vert D_t \op(\chi_j)\op(1-\psis)u_{\pm}  \right\vert^2_{L^2(\Sigma)}
+\norm{u}{\lrmn}{2} + \left\vert  u  \right\vert^2_{L^2(\Sigma)} +T_{\theta,\Theta}
     \\\gtrsim \tau \norm{\op(\chi_j)\op(1-\psis)u}{L^2(\R ; H^1_\tau)}{2}+\tau^3 \normsurf{\op(\chi_j)\op(1-\psis)u}{\ls}{2}\\+ 
    \tau\normsurf{ \nabla \op(\chi_j)\op(1-\psis)u_\pm}{\ls}{2}.
\end{multline*}
We add the two estimates above. We control the right hand side from below by using the triangle inequality as well as the fact $$\op(1-\psis)u= \sum_j \op(\chi_j) \op(1-\psis)u.$$ For the left hand side we argue with commutators. For instance we control:
\begin{align*}
    \norm{\cnop \op(\chi_j) \op(1-\psis)u}{\lrn}{} &\leq  \norm{\op(\chi_j)\cnop  \op(1-\psis)u}{\lrn}{}\\
    &\hspace{4mm}+ \norm{[\cnop, \op(\chi_j)] \op(1-\psis)u}{\lrn}{}.
\end{align*}
Recall the notations $\tths=1-\check{\psi}_{\sigma}$ and $\ths=1-\psis$ with $\check {\psis}$ as defined in the proof of Lemma~\ref{region G+ cv}. We write
\begin{align*}
  [\cnop, \op(\chi_j)]=[\cnop, \op(\chi_j)\op(\tths)]+[\cnop, \op(\chi_j)\op(1-\tths)].
\end{align*}
Now on the one hand $\supp{(1-\tths)} \cap \supp{\ths}= \emptyset$ and on the other hand $\ths$ localizes in a region where $\tau \lesssim |\xip|+|\xi_t|$. This implies that
$$
[\cnop, \op(\chi_j)\op(1-\tths)]\op(\ths) \in \psitt^{-\infty},
$$
therefore
$$
\norm{[\cnop, \op(\chi_j)\op(1-\tths)]\op(\ths) u}{\lrn}{}\lesssim \norm{u}{\lrn}{}.
$$
We have $[\cnop, \op(\chi_j)\op(\tths)] \in \psitt^1.$
The above considerations as well as the fact that $\op(\chi_j) \in \psit^0$ finally yield:
\begin{align*}
    \norm{\cnop \op(\chi_j) \op(1-\psis)u}{\lrn}{} &\leq  \norm{\op(\chi_j)\cnop  \op(1-\psis)u}{\lrn}{}\\
    &\hspace{4mm}+ \norm{[\cnop, \op(\chi_j)] \op(1-\psis)u}{\lrn}{} \\
    &\lesssim \norm{\op(1-\psis)u}{L^2(\R; H^1_\tau)}{}+\norm{u}{\lrn}{}\\
    &\hspace{4mm}+\norm{\cnop \op(1-\psis)u}{\lrn}{}.
\end{align*}
We use the same argument for the other terms of the right hand side (noticing that $\xi_t \in \mathcal{S}^1_\tau$). We finally obtain, taking also $\tau$ large to absorb the error terms (as for instance $\norm{\op(1-\psis)u}{L^2(\R; H^1_\tau)}{}$)
\begin{multline}
\label{final estim tau petit}
    \norm{\cnop \op(1-\psis)u}{\lrn}{2} +\tau \norm{ D_t \op(1-\psis)u_{+}}{L^2(\R^{n+1}_{+})}{2}+\tau \norm{ D_t \op(1-\psis)u_{-}}{L^2(\R^{n+1}_{-})}{2}  
    \\+\tau \left\vert D_t \op(1-\psis)u_{+}  \right\vert^2_{L^2(\Sigma)}+\tau \left\vert D_t \op(1-\psis)u_{-}  \right\vert^2_{L^2(\Sigma)}
+\norm{u}{\lrn}{2} + \left\vert  u  \right\vert^2_{L^2(\Sigma)}+T_{\theta,\Theta} 
     \\\gtrsim \tau \norm{\op(1-\psis)u}{L^2(\R ; H^1_\tau)}{2}+\tau^3 \normsurf{\op(1-\psis)u}{\ls}{2} 
     \\+\tau\normsurf{ \nabla \op(1-\psis)u_+}{\ls}{2}+\tau\normsurf{ \nabla \op(1-\psis)u_-}{\ls}{2}.
\end{multline}

Lemma~\ref{prop for tau grand} furnishes an estimate in the complementary sub-region:
\begin{multline}
\label{final estimate tau grand}
     \norm{\cnop \op(\psis)u}{\lrn}{2} +\norm{u}{\lrn}{2} +T_{\theta,\Theta}+\normsurf{u_-}{\ls}{2}+\normsurf{u_+}{\ls}{2}
     \\\gtrsim \tau \norm{\op(\psis)u}{L^2(\R ; H^1_\tau)}{2}+\tau^3 \normsurf{\op(\psis)u}{\ls}{2}\\+\tau\normsurf{ \nabla \op(\psis)u_+}{\ls}{2}+\tau\normsurf{ \nabla \op(\psis)u_-}{\ls}{2}.
\end{multline}

To finish the proof we add \eqref{final estim tau petit} and \eqref{final estimate tau grand}. For the right hand side we simply use the triangle inequality. For example:
$$
\norm{\op(1-\psis)u}{\lrn}{}+\norm{\op(\psis)u}{\lrn}{}\geq \norm{u}{\lrn}{},
$$
since $\op(1-\psis)+\op(\psis)=\textnormal{Id}$. For the left hand side we argue with commutators as before. We notice that: $$\cnop \in \mathcal{D}^2_\tau, \: \op(\psis)\in \psitt^0, \: 1-\op(\psis) \in \psitt^0\: \textnormal{and} \: D_t \in \psitt^1.$$ Therefore in particular $$[\cnop, \op(\psis)] \in \psitt^1, \: [\cnop, 1-\op(\psis)] \in \psitt^1 \: \textnormal{and}\:
[D_t,1-\op(\psis)]\in \Psi^0_\tau \: .$$
We take as usual $\tau$ sufficiently big to absorb the errors terms and we have thus proven:
\begin{multline}
\label{final ineq but without volume derivatives}
    \norm{\cnop u}{\lrn}{2} +\tau \norm{ D_t u_{+}}{L^2(\R^{n+1}_{+})}{2}+\tau \norm{ D_t u_{-}}{L^2(\R^{n+1}_{-})}{2}  
    \\+\tau \left\vert D_t u_{+}  \right\vert^2_{L^2(\Sigma)}+\tau \left\vert D_t u_{-}  \right\vert^2_{L^2(\Sigma)}
+\norm{u}{L^2(\R;H^1_\tau)}{2} + \left\vert  u  \right\vert^2_{L^2(\Sigma)} +T_{\theta, \Theta}
     \\\gtrsim \tau \norm{u}{L^2(\R ; H^1_\tau)}{2}+\tau^3 \normsurf{u}{\ls}{2} +\tau\normsurf{ \nabla u_+}{\ls}{2}+\tau\normsurf{ \nabla u_-}{\ls}{2}.
\end{multline}
Observe that the only term that is missing in the above estimate compared to the statement of Proposition~\ref{theorem bis} is the volume norm of the derivatives. The latter are estimated thanks to Lemma~\ref{derivatives Dn cv}. Indeed, we multiply \eqref{final ineq but without volume derivatives} by a large constant and add it to the estimate of Lemma~\ref{derivatives Dn cv}. This finishes the proof of Proposition~\ref{theorem bis}.
\end{proof}

\bigskip

\subsection{Convexification: A perturbation argument}
\label{convexification}

\noindent For the sequel it is important to notice that the quality of the estimates obtained for the first order factors (such as \eqref{estimate for e+} etc) depend on the \textit{imaginary} part of the operator only. Indeed, consider $L,M \in \psit^1$ with \textit{real} symbols. One has:
\begin{align*}
\norm{(D_n+L+iM)v}{}{2}=\norm{(D_n+L)v}{}{2}+\norm{Mv}{}{2}+2\operatorname{Re}(D_n v, iM v)+2\operatorname{Re}(Lv,iMv).
\end{align*}
Now the fact that $L$ and $M$ have real principal symbols implies (since we are working with the Weyl quantization) that 
$$
2\operatorname{Re}(Lv,iMv)=i ([L,M]v,v).
$$
Consequently one has $[L,M] \in \Psi^1$ and 
$$
|([L,M]v,v)|  \lesssim \norm{v}{L^2(\R ; H^{1/2})}{2},
$$
which can be absorbed in first-order factor estimates such as \eqref{estimate for e+}.

\bigskip

The Carleman estimate we have obtained so far involves a weight function $\phi$ depending only on the variable $x_n$ which in our local coordinates describes the interface $\Sigma=\{x_n=0\}.$ However it is important for applications to allow dependence in the other variables too. We show here that indeed this is possible if one changes “slightly” the weight function $\phi$. Recall that we have
$
\phi=\alpha_{\pm}x_n+\frac{\beta x_n^2}{2},
$
and consider the new weight 
$$
\psi=\phi+\kappa(t,x_n,x^\prime), \quad \kappa \: \textnormal{real valued quadratic polynomial.}
$$
We have taken $\kappa$ quadratic polynomial for technical reasons related to the action of $e^{-\delta \frac{D^2_t}{2 \tau}}$ (see \cite{Tataru:95, Hor:97, RZ:98, LLnotes}). However this should be sufficient for the applications. We shall now verify that if $ \norm{\kappa^\prime}{L^\infty}{}$ is sufficiently small then the steps carried out in the preceding sections remain valid. 

\begin{prop}[\textbf {Geometric convexification}]
\label{prop of geometric convexification}
Consider the new weight $\psi=\phi+\kappa(t,x_n,x^\prime)$ with $\kappa$ a real valued quadratic polynomial. Then there exist $\eta, \delta_0>0$ depending on the coefficients of $\phi$ and $P$ such that if $ \norm{\kappa^\prime}{L^\infty}{}\leq \eta$ and $\delta \leq \delta_0$ then the estimates of Proposition~\ref{theorem bis} and consequently of Theorem~\ref{the carleman ineq thm} remain valid with the weight $\psi$.
\end{prop}

\begin{proof}[\unskip\nopunct]
We shall show that the crucial sub-elliptic estimate estimate Proposition~\ref{theorem bis} remains valid with the new weight. To do so, we shall revisit the key arguments and show that up to taking $\lkp$ small we can see the new conjugated operator as a perturbation of the conjugated operator with the weight depending only on $x_n$.

Recall the microlocal weight $Q^\psi_{\delta, \tau}$ and the conjugated operator $P^\pm_\psi$ have been defined in \eqref{def of microlocal weight} and \eqref{def of conjugated operators}.

Recall as well that we have the following formula for $P_\psi$: (see \cite[Chapter 3.3.1]{LLnotes} or~\cite{Hor:97})
\begin{align*}
c^{-1}_{\pm}(x)P^{\pm}_{\psi}&=(D_n+ i \tau \phi^\prime + i \tau  \kappa^\prime_{x_n} -\delta \kappa ^{\prime \prime}_{t,t}D_t)^2\\ &\hspace{4mm}+\bjk (D_j+i \tau \partial_j \kappa -\delta \kappa^{\prime \prime}_{t,x_j} D_t )(D_k+i \tau \kappa^\prime_{x_k} -\delta \kappa ^{\prime \prime}_{t,x_k} D_t)\\
&\hspace{4mm}-c^{-1}_\pm(x)((1 -\delta \kappa ^{\prime \prime}_{t,t})D_t+i \tau \partial_t \kappa)^2.
\end{align*}
In the sequel we shall denote by $\tilde{P}_{\psi}$ the principal part of the operator $P_{\psi}$.

Notice that since $\kappa$ is supposed to be quadratic $\kappa^{\prime \prime}_{t,x_j}$, $\kappa^{\prime \prime}_{t,t}$  are actually constants. We shall write $\psi^\prime_{x_n}= \partial_{x_n} \psi$. Let us start by showing that the result of Lemma~\ref{derivatives Dn cv} remains valid, up to taking small values for $\delta$ and $\norm{\kappa^\prime}{L^{\infty}}{}.$ We write:
\begin{align*}
c^{-1}_{+}(x)\tilde{P}_{\psi}&=(D_n+ i \tau \psi^\prime_{x_n})^2-2\delta \kappa ^{\prime \prime}_{t,t}D_t(D_n+\tau \psi^\prime_{x_n})\\ &\hspace{4mm}+Q(x,D_{x^\prime})-c^{-1}_{+}(x)(1-\delta \kappa ^{\prime \prime}_{t,t})D_t^2+\tilde{R}_+,
\end{align*}
with $\tilde{R}_+$ a tangential differential operator (see Section~\ref{differential operators}) satisfying
\begin{align}
\label{conv aux 1}
    (\tau \tilde{R}_+v,v)_+ \lesssim \norm{\kappa ^\prime}{L^\infty}{}\left(\tau \norm{\nabla_{x^\prime}v}{\lrpn}{2}+\tau^3 \norm{v}{\lrpn}{2}\right)+\tau \delta \norm{\nabla_{t,x^\prime}}{\lrpn}{2}.
\end{align}
We notice that $\psi^\prime_{x_n}(0)=\alpha_{\pm}+\kappa^\prime_{x_n} (0,t,x^\prime) $, therefore for $|x_n|$ sufficiently small and for $\alpha_{\pm}/\norm{\kappa^\prime}{L^\infty}{}$ sufficiently large the estimates involving $\phip$ remain valid for $\psi^\prime$ as well. One obtains then:
\begin{align*}
    \norm{\tilde{P}_{\psi}v}{\lrpn}{2}+ \tau^2 \norm{v}{\lrpn}{2} &\geq 2 \tau \operatorname{Re}(\tilde{P}_{\psi}v,v)_+\\
    &\gtrsim  \tau \operatorname{Re}(\cnp v,v)_+ + (\tau \tilde{R}_+v,v)_+- 2\tau \delta \kappa ^{\prime \prime}_{t,t} (D_nv+ i \tau \psi^\prime_{x_n} v,D_t v)_+\\
    &\gtrsim \tau \operatorname{Re}(\cnp v,v)_+ + (\tau \tilde{R}_+v,v)_+ \\
    &\hspace{4mm}- (\tau \delta \norm{D_n v}{\lrpn}{2}+\delta \tau^3 \norm{v}{\lrpn}{2}+\delta \tau \norm{D_t v}{\lrpn}{2}).
    \end{align*}
Using \eqref{conv aux 1} combined with the estimate obtained in Lemma~\ref{derivatives Dn cv} for $\tau \operatorname{Re}(\cnp v,v)_+$ we obtain, up to taking $\delta \leq \delta_0$ the same result as in Lemma ~\ref{derivatives Dn cv} but for the convexified weight.

\bigskip

We now investigate what happens with respect to the microlocal regions considered in Section~\ref{proof for the general case} and show that in fact we have the same estimates. We recall that there are three main regions. The first is the one where $\tau$ is large compared to $|\xip|+|\xi_t|$, the second is the non-elliptic region and the third is the elliptic one. 

\begin{itemize}
    \item We localize with $\op(\psis)$ in a region where $
    \tau \geq \frac{1}{\sigma}(|\xip|+|\xi_t|).
    $ (recall that $\psis$ has been defined in ~\eqref{defintion of psi sigma}).
This is the region covered in Lemma~\ref{prop for tau grand} and we check that the change of weight function from $\phi$ to $\psi$ only adds acceptable error terms. Indeed we write:
\begin{align*}
    \invcp(x)\tilde{P}_{\psi}&= (D_n-\delta \kappa ^{\prime \prime}_{t,t}D_t + i \tau \psi^\prime_{x_n})^2=\tilde{A}_+ +\tilde{R}_+,
\end{align*}
where $\tilde{A}_+=(D_n-\delta \kappa ^{\prime \prime}_{t,t}D_t + i \tau \psi^\prime_{x_n})^2$ and $\tilde{R}_+$ has a principal symbol $\tilde{r}_+$ satisfying
\begin{equation}
\label{conv aux 2}
\tilde{r}_+ \lesssim |\xip|^2+|\xi_t|^2+{\lkps} \tau^2.
\end{equation}

As in the proof of Lemma~\ref{prop for tau grand} we have a very good estimate for $\norm{\tilde{A}_+ v}{\lrpn}{2}$. We calculate for the first order factor:
\begin{align*}
    \norm{(D_n-\delta \kappa ^{\prime \prime}_{t,t}D_t + i \tau \psi^\prime_{x_n})^2v}{\lrpn}{2}&=\norm{(D_n-\delta \kappa ^{\prime \prime}_{t,t}D_t)^2}{\lrpn}{2}+\norm{\tau \psi^\prime_{x_n} v}{\lrpn}{2}\\
    &\hspace{4mm}+2\operatorname{Re}(D_nv,i\tau \psi^\prime_{x_n} v)_+-2\delta\operatorname{Re}(D_tv,i \tau \psi^\prime_{x_n} v)_+.
\end{align*}
Since an integration by parts in the $t$ variable yields  
$$
2\delta\operatorname{Re}(D_tv,i \tau \psi^\prime_{x_n} v)_+=0,
$$
we obtain the usual estimate:
$$
 \norm{(D_n-\delta \kappa ^{\prime \prime}_{t,t}D_t + i \tau \psi^\prime_{x_n})^2v}{\lrpn}{2} \gtrsim \tau^2\norm{v}{\lrpn}{2}+ \tau \normsurf{v}{\ls}{2}.
$$
We iterate twice to get, exactly as in the proof of Lemma~\ref{prop for tau grand}:
\begin{equation*}
   \norm{A_+v}{\lrpn}{2} \gtrsim  \tau ^4 \norm{v}{\lrpn}{2} + \tau^3 \normsurf{v}{\ls}{2}+ \tau \normsurf{D_n v}{\ls}{2}.
\end{equation*}
And we estimate in the same way:
\begin{align*}
\norm{\tilde{P}_{\tphi}  v}{\lrpn}{2} &\geq C\norm{\tilde{A}_+ v}{\lrpn}{2}- \norm{\tilde{R}_+ v}{\lrpn}{2} \\
&\geq  C\left(\tau ^4 \norm{v}{\lrp}{2} + \tau^3 \normsurf{v}{\ls}{2}+\tau \normsurf{D_nv}{\ls}{2}\right)- \norm{\tilde{R}_+ v}{\lrpn}{2} \\
&=\left((C\tau^4-\tilde{R}_+^2)v,v\right)_+ +C\tau^3 \normsurf{v}{\ls}{2}+C\tau \normsurf{D_nv}{\ls}{2},
\end{align*}
with $C$ positive constant depending on the the coefficients of $\cnop$ and of $\phi$. According to~\eqref{conv aux 2} we now have
$$
C\tau^4-\tilde{r}_+^2 \geq C \tau^4 - \tilde{C}\left(|\xip|^2+|\xi_t|^2-\lkps \tau^4\right).
$$
Choosing $\lkps\leq \eta^2 $ with $\eta$ sufficiently small depending on $C$ and $\tilde{ C}$ we have
$$
 C \tau^4 - \tilde{C}\left(|\xip|^2+|\xi_t|^2-\lkps \tau^4\right)\geq \frac{C}{2}\tau^4-\tilde{C}(|\xip|^2+|\xi_t|^2).
$$
Then one can choose $\sigma \leq \sigma_0$ small such that
$$
C\tau^4-\tilde{r}_+^2 \geq \lambda_\tau^4,
$$
on the support of $\psis$. This fixes the choice of $\sigma_0$. Then we obtain the same estimate as in Lemma \ref{prop for tau grand} with the weight $\psi$.

\item We localize now with $\op(1-\psis)$ in the sub-region where $\tau \lesssim \frac{1}{\sigma} (|\xip|+|\xi_t|)$ with $\sigma\leq \sigma_0 $. In this region on has for $\lkp\leq \eta $, with $\eta$ sufficiently small that
$$
1/C ( |\xip|^2+|\xi_t|^2) \leq|\xip|^2+|\xi_t|^2-\tau^2 \lkps \leq C( |\xip|^2+|\xi_t|^2),
$$
for some $C>0$.

We now investigate the non-elliptic and elliptic regions. 

\textbf{Non elliptic region}

We can treat the non elliptic region as in the proof of Lemmata~\ref{region G+ cv} and~\ref{region G- cv}. In this region the localization of $\op(1-\psis)$ implies that $\tau \lesssim |\xip|+|\xi_t| $ and being outside the elliptic region implies that $|\xi_t| \gtrsim |\xip|$ and thus $|\xi_t| \gtrsim \tau+ |\xip|$. Therefore by the same arguments as in the proofs of Lemmata~\ref{region G+ cv} and~\ref{region G- cv}  one needs to obtain only one of the trace of the normal derivatives (for the other one we use as again Lemma~\ref{one of boundary derivatives cv}). The commutator technique works here exactly as before. Indeed, we consider 
$$
\tilde{Q}_2=c(x)^{-1} \frac{\tilde{P}_{\psi}+\tilde{P}_{\psi}^*}{2}, \quad \tilde{Q}_1=c(x)^{-1}\frac{\tilde{P}_{\psi}-\tilde{P}_{\psi}^*}{2i},
$$
and we decompose 
$$
c(x)^{-1} \tilde{P}_{\psi}=\tilde{Q}_2+i \tau \tilde{Q}_1.
$$
We observe that 
$$\tilde{Q}_2=(D_n-\delta \kappa ^{\prime \prime}_{t,t}D_t)^2-\tau^2\left(|\psi^\prime_{x_n}|^2+Q(\nabla_{x^\prime}\kappa)\right) +T_2,$$
and
$$
\tilde{Q}_1=\psi^\prime_{x_n} D_n+ D_n \psi^\prime_{x_n}+T_1,
$$
where $T_j$ are \textit{tangential} operators of order $j$. We can then proceed exactly as in the proof of Lemma~\ref{region G+ cv}. What is crucial in the proof of this lemma is the sign of $\phip(0)$ (which is positive close to $\Sigma$). If $\phi$ is replaced by $\psi$ one can also have $\psi^\prime_{x_n}(0)>0$, if we choose $\alpha_+/ \lkp=\phip(0^+)/ \lkp$ is sufficiently large.

\textbf{Elliptic region}

This is the region $\mathcal{E}^-_\epsilon \cap \mathcal{E}^+_\epsilon $ with the definition of \eqref{def of E}. Here we are in the situation of Lemma~\ref{Lemma elliptic region} and we follow \cite[Section 4E]{le2013carleman}. We revisit the factorization argument. To do this we check that in this microlocal region one can define a square root for the operator
\begin{align} 
\label{definition of s } 
S&:=\bjk (D_j+i \tau \kappa^\prime_{x_j} -\delta \kappa^{\prime \prime}_{t,x_j} D_t )(D_k+i \tau \kappa^\prime_{x_k} -\delta \kappa ^{\prime \prime}_{t,x_k} D_t) \nonumber \\ &\hspace{4mm}-c^{-1}(x)((1 -\delta \kappa ^{\prime \prime}_{t,t})D_t+i \tau \partial_t \kappa)^2. 
\end{align} 
Since its principal symbol $s$ is no longer real, we study its real part. A sufficient condition for defining a square root is that its real part is positive elliptic. We thus compute: 
\begin{align*} \operatorname{Re}(s)&=  \bjk \left(\xi_j\xi_k- \delta \kappa ^{\prime \prime}_{t,x_j}\xi_j\xi_t- \tau^2 \kappa ^{\prime }_{x_j}\kappa ^{ \prime}_{x_k}- \delta\kappa ^{\prime \prime}_{t,x_k}\xi_t\xi_k+\delta^2  \kappa ^{\prime \prime}_{t,x_j} \kappa ^{\prime \prime}_{t,x_k}    \xi_t^2 \right)
\\ &\hspace{4mm}-c^{-1}(x)(1-\delta \kappa ^{\prime \prime}_{t,t} \xi_t^2+c^{-1}(x)\tau^2 (\kappa ^{\prime }_{t})^2 
\\ &=Q(x,\xip) -c^{-1}(x)\xi_t^2+r, 
\end{align*} 
with
\begin{equation}
\label{conv aux 3}
|r|\lesssim \delta \norm{ \kappa^{\prime \prime}}{L^\infty}{}(|\xip|^2+|\xi_t|^2) + \tau^2 \norm{\kappa^\prime}{L^\infty}{2}.
\end{equation}

When microlocalized in $\mathcal{E}^-_\epsilon \cap \mathcal{E}^+_\epsilon $ one has $$Q(x,\xip) -c^{-1}(x)\xi_t^2 \geq \epsilon (|\xip|^2+|\xi_t|^2). $$ 
Combining this with estimate ~\eqref{conv aux 3} we see that for $\delta\leq \delta_0$ we have

 $$ \operatorname{Re}(s) \geq C \epsilon(|\xip|^2+|\xi_t|^2)-\tau^2 \norm{\kappa^\prime}{L^\infty}{2}.$$ 
Recalling that $\tau \leq \frac{1}{\sigma}(|\xip|^2+|\xi_t|^2)$ in the support of $\op(1-\psis)$ we obtain that up to taking $\eta$ small enough we have for $\norm{\kappa^\prime}{L^\infty}{}\leq  \eta$ in the elliptic region:
$$
\operatorname{Re}(s) \gtrsim |\xip|^2+|\xi_t|^2.
$$
Using a cut-off $\tchi$ which localizes in the elliptic region we define then (as for the definition of $m_+$ in \ref{ellipticity}):
$$
\tilde{s}:=\tchi s +(1-\tchi)\lambda^2.
$$
We use then the principal value of the square root for complex numbers to define 
$$
\tilde{m}=\tilde{s}^{\frac{1}{2}} \in \mathcal{S}^1, \quad \operatorname{Re} \tilde{m} \gtrsim (|\xip|+|\xi_t|).
$$

Consequently we obtain the following almost-factorization:
$$
\invc(x)\tilde{P}_{\psi}v=\left(D_n-\delta \kappa ^{\prime \prime}_{t,t}D_t+ i \tau \psi^\prime_{x_n} -i \op(\tilde{m})\right)\left(D_n-\delta \kappa ^{\prime \prime}_{t,t}D_t+ i \tau \psi^\prime_{x_n} +i \op(\tilde{m})\right) v+\tilde{ R} v,
$$
where $v=\op(1-\psis)\op(\chi)u$ with $u \in \mathscr{S}_c(\R^{n+1})$, $\chi \in \mathcal{S}^{0}$ with $\supp(\chi) \subset \mathcal{E}^-_\epsilon \cap \mathcal{E}^+_\epsilon  $ and $\tilde{R}   \in \psit^1+\tau \psit^0+ \psit^0 D_n$.
We have already seen in the beginning of Section~\ref{convexification} that the imaginary part of the first-order factors determines the quality of the estimate we obtain. We write then:
\begin{align*}
\invc(x)\tilde{P}_{\psi}&=\left(D_n-\delta \kappa ^{\prime \prime}_{t,t}D_t+ \op(\operatorname{Im}\tilde{m})+ i( \tau \psi^\prime_{x_n} -  \op(\operatorname{Re}\tilde{m}))\right)\\
&\hspace{4mm} \cdot \left(D_n-\delta \kappa ^{\prime \prime}_{t,t}D_t-\op(\operatorname{Im}\tilde{m})+ i (\tau \psi^\prime_{x_n} + \op(\operatorname{Re}\tilde{m}))\right)v+\tilde{ R}v,
\end{align*}
and we focus on the imaginary part of the first order factors above. Since $ \op(\operatorname{Re}\tilde{ m})$ satisfies the same estimates (elliptic positive) as $m$ (as defined in Section \ref{section for elliptic region}) the proof remains valid, up to taking $\alpha_{\pm}/\lkp$ sufficiently large and under the same geometric hypothesis similarly to~\cite[Section 4.5]{le2013carleman}.
\end{itemize}

Taking everything into account we have obtained the same estimates as in Section~\ref{proof for the general case} with $\psi$ in place of $\phi$ and in the same microlocal regions. One can then patch these estimates together and obtain the desired result with the convexified weight, exactly as in Section~\ref{proof for the general case}. This finishes the perturbation argument and therefore the proof of Proposition~\ref{prop of geometric convexification}.
\end{proof}

\section{The quantitative estimates}
\label{the quant estimates}
With Theorem~\ref{theorem} at hand we are now ready to obtain the desired quantitative estimates following~\cite{Laurent_2018}. Firstly we obtain a local quantitative estimate (the analogue of Theorem 3.1 in \cite{Laurent_2018}). This estimate allows to propagate the information quantitatively from a small neighborhood of one point belonging to one side of the interface to some other neighborhood of the other side. For this estimate one needs to make sure that the methods used in \cite{Laurent_2018} can also be adapted to our context. 

We can then use this new local quantitative estimate to cross the interface and then continue the propagation process by directly using the results of \cite{Laurent_2018} which are valid as soon as the coefficients of our operator are smooth with respect to the space variable.

\subsection{Some definitions and statement of the local estimate}
\label{some definitions for the local quant estimate}

Before stating the Theorem we need to introduce some notation from \cite{Laurent_2018}. We only propagate \textit{low frequency} information with respect to time. Let $m(t)$ be a smooth radial function, compactly supported in $|t|<1$ such that $m(t)=1$ for $|t|<3/4$. We shall denote by $M^\mu$ the Fourier multiplier defined $M^\mu=m(\frac{D_t}{\mu})$, that is
$$
M^\mu u (t,x)=\mathcal{F}_t^{-1}\left(m\left(\frac{\xi_t}{\mu}\right) \mathcal{F}_t (u)(\xi_t,x)\right)(t).
$$
Therefore the upper index $\mu$ translates to an operator that localizes to times frequencies smaller than $\mu$. We shall also use a regularization operator. Given a function $f \in L^\infty(\R^{n+1})$ we set
$$
f_\lambda(t,x):=e^{-\frac{|D_t|^2}{\lambda}}f=\left(\frac{\lambda}{4 \pi}  \right)^{\frac{1}{2}}\int_{\R} f(s,x)e^{-\frac{\lambda}{4}|t-s|^2}ds.
$$
That is, the lower index $\lambda$ produces an analytic function with respect to the time variable. We will need also the combination of the two procedures above. Given $\lambda, \mu >0$, we write $M_\lambda^\mu$ for the Fourier multiplier defined by $M_\lambda^\mu=m_\lambda (\frac{D_t}{\mu})$ or more precisely:
$$
(M_\lambda^\mu u)(t,x)=\mathcal{F}_t^{-1}\left(m_\lambda\left(\frac{\xi_t}{\mu}\right)\mathcal{F}_t (\xi_t,x)\right)(t).
$$
That is, we first regularize and then localize. 

Let us consider as well a smooth function $\sigma \in C^{\infty}(\R)$ such that $\sigma=1$ in a neighborhood of $(-\infty,1]$, and $\sigma=0$ in a neighborhood of $[2,+\infty )$. Given a point $(t_0,x_0) \in \Sigma$ we write
\begin{equation}
    \label{definition of function sigma}
    \sigma_R(t,x):=\sigma\left(\frac{|(t,x)-(t_0,x_0)|}{R}\right).
\end{equation}

We can now state the local quantitative estimate. We recall that $\Sigma$ is defined as $\R_t \times S$ and that we are in the geometric situation presented in Section~\ref{setting and main}.

\begin{thm}
	\label{local quant estimate}
	Let $(t_0,x_0) \in \Sigma$ given locally by $\Sigma=\{\phi=0 \}$. Then there exists $R_0 >0$ such that for any $R \in (0,R_0)$ there exist $r, \rho, \tilde{\tau}_0>0$ such that for any $\theta \in C^\infty_0(\R_t \times \mathcal{M})$ with $\theta(x)=1 $ on a neighborhood of $\{\phi \geq 2 \rho \} \cap B((t_0,x_0),3R)$, for all $c_1, \kappa >0$ there exist $C, \kappa^\prime, \beta_0 >0$ such that for all $\beta \leq \beta_0$, we have 
	\begin{equation*}
	Ce^{\kappa \mu}\left(\norm{M^\mu_{c_1 \mu}\theta_{c_1 \mu}u}{H^1}{}+\norm{Pu}{L^2(B((t_0,x_0),4R))}{}\right)+Ce^{-\kappa^\prime \mu} \norm{u}{H^1}{}\geq    \norm{M^{\beta \mu}_{c_1 \mu} \sigma_{r,c_1\mu}u}{H^1}{} ,
	\end{equation*}
	for all $\mu \geq \tilde{\tau_0}/\beta$ and $u \in \mathcal{W}$ compactly supported.
\end{thm}

\begin{remark}
\label{propagation in all directions}
 In the statement of Theorem~\ref{local quant estimate} uniqueness is propagated quantitatively from $\Omega_+$ to $\Omega_-$. However, we have the same result in the other direction as well. Indeed, this comes from the fact that since there is no assumption on the jump of the coefficient $c$, the geometric situation as presented in Section~\ref{setting and main} is completely symmetrical with respect to $\Omega_+$ and $\Omega_-$ up to changing the sign of $c_+-c_-$. This will be important in the proof of the semi-global estimate (proof of Theorem~\ref{semi global }) where the local quantitative estimate will be applied successively in chosen points of the interface.
\end{remark}

\begin{figure}
	\centering
	
	\begin{tikzpicture}
	\draw[scale=0.5, domain=-9:9, smooth, variable=\x] plot (0.01*\x*\x,\x) ;
	\node at ( -0.75,4.5) {$\{\phi=0\}$};
		\node at (0,-4.5)  {$\Sigma$};
		
		\node at ( -4,3) {$\Omega_{t,-}$};	
		
		\node at ( 4,3) {$\Omega_{t,+}$};
	
	\draw[scale=0.5, domain=-9:9, smooth, variable=\x] plot (0.01*\x*\x+2,\x) 	node[above] {$\{\phi=2 \rho\}$};
\filldraw[black] (0,0) circle (2pt) node[anchor=south,scale=0.9]{$(t_0,x_0)$} ;
	
\draw  (0,0) circle (0.75cm) ;
	
\draw  (0,0) circle (3.2cm) ;

\draw  (0,0) circle (2.5cm) ;

\draw[pattern=north west lines,opacity=.5, pattern color=red] (0,0) circle (0.75cm) ;

\begin{scope}
\clip  (0,0) circle (2.5cm);

\clip   (1.1,3)--(1,0) -- (1.1,-3) -- (10,-3) -- (10,3) -- cycle;

\draw[pattern=north west lines,opacity=.5, pattern color=blue] (0,0) circle (3cm) ;
\end{scope}

	\draw(0,0)--(-0.53,-0.53)  node[pos=0.6,left] {$r$} ;
	
		\draw(0,0)--(-3.04,0.988)  node[left] {$4R$};
		
	  \draw(0,0)--(-0.95,-2.30)  node[below] {$3R$};

\end{tikzpicture}
\caption{Geometry of the local quantitative estimate. The function $\theta$ localizes in the blue region and $\sigma$ in the red one. This allows to propagate information from the blue to the red region.}
\label{geometry of local quant esimate}
\end{figure}
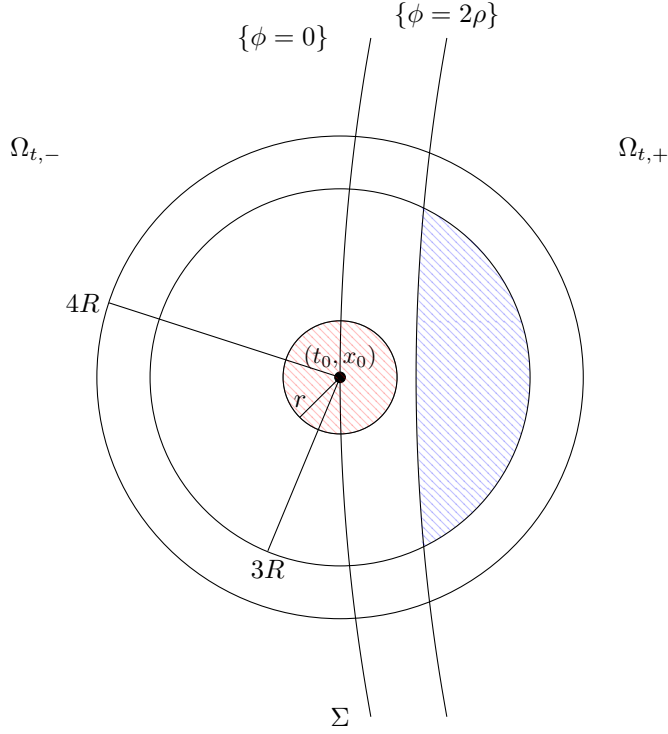

\subsection{Proof of Theorem~\ref{local quant estimate}}
\label{proof of local quant estimate}

\begin{proof}[\unskip\nopunct]

We work as usual in geodesic normal coordinates as explained in the local setting of Section~\ref{local setting}. This does not pose any problem since the estimate we are seeking to prove is invariant by change of coordinates in the $x$ variable. In our context, the first and most important step for the proof of Theorem~\ref{local quant estimate} will be to state a Carleman estimate with a \textit{geometrically} convexified weight. That is the purpose of Proposition~\ref{starting point}. This estimate will provide the analogue of Corollary 3.6 in \cite{Laurent_2018} and will be the starting point of the quantified version of Theorem~\ref{local quant estimate}.

\begin{prop}
\label{starting point}
Let $(t_0, x_0) \in \Sigma$ given locally by $\Sigma=\{\phi=0\}=\{x_n=0\}$. Then there exist $\Omega$ a neighborhood of $(t_0,x_0)$, a function $\psi: \Omega \rightarrow \R$ which is a quadratic polynomial in $t$ and $R_0>0$ such that $B((t_0,x_0),4R_0) \subset \Omega$ and for any $R \in (0,R_0]$, there exist $\epsilon, \delta,\rho, r, d,  \tau_0, C>0$ such that $\delta  \leq \frac{d}{8}$ and 

\begin{enumerate}
    \item The Carleman estimate 
    \begin{equation*}
    C\left( \norm{Q_{\epsilon, \tau}^{\psi} Pu}{L^2(\Omega_{t,-}\cup \: \Omega_{t,+})}{2} +e^{-d \tau} \norm{e^{\tau \psi}u }{H^1_\tau}{2} + T_{\theta, \Theta}\right) \geq \tau \norm{Q_{\epsilon, \tau}^{\psi}u}{H^1_\tau}{2},
\end{equation*}
holds for all $\tau \geq \tau_0$ and all $u \in \tran$ with $\supp{u} \subset B((t_0,x_0),4R)$;

\item One has
\begin{align}
        \left( B((t_0,x_0),5R/2) \backslash B((t_0,x_0),R/2) \cap \{-9 \delta \leq \psi \leq 2 \delta \} \right )&\Subset \{\phi> 2\rho \} \cap B((t_0,x_0),3R), \label{geom conv 1}\\
    \{\delta/4\leq \psi \leq 2 \delta  \} \cap B((t_0,x_0), 5R/2) \Subset \{\phi > &2\rho\} \cap B((t_0,x_0),3R), \label{geom conv 2}\\ \label{geom conv 3}
    \hspace{-10mm} B((t_0,x_0),2r)\Subset  \{-\delta/2 \leq \psi \leq \delta/2\} \: \cap & \: B((t_0,x_0),R).
\end{align}

\end{enumerate}
\end{prop}

\begin{remark}
The first item is the Carleman estimate we have already obtained and the second one says that we can have this estimate with a weight function whose level sets are appropriately curved with respect to the interface $\Sigma$. This is the geometric convexification part.

\end{remark}

\begin{proof}
We suppose to simplify that $(t_0,x_0)=0$. Theorem~\ref{theorem} gives us the desired estimate with a weight function $\phi$ defined in~\eqref{def of weight phi}. Proposition~\ref{prop of geometric convexification} gives the existence of $\tilde{\delta}$ sufficiently small such that the same estimate is valid with the weight $\psi$ defined as  
\begin{equation}
\label{def of weight psi}
\psi= \phi - \tilde{\delta}|(t,x)|^2.
\end{equation}

More precisely one has the existence of $R_0, \epsilon, d, \tau_0$ and $C$ such that
  \begin{equation*}
    C\left(\norm{Q_{\epsilon, \tau}^{\psi} Pu}{L^2(\Omega_{t,-}\cup \: \Omega_{t,+})}{2} +e^{-d \tau} \norm{e^{\tau \psi}u }{H^1_\tau}{2} +  T_{\theta, \Theta}\right)\geq \tau \norm{Q_{\epsilon, \tau}^{\psi}u}{H^1_\tau}{2},
\end{equation*}
 for all $\tau \geq \tau_0$ and all $u \in \tran$ with $\supp{u} \subset B(0,4R)$ and $R \leq R_0$. Consider now  $\delta>0$ such that 
 $$
 \delta \leq \frac{\tilde{\delta}R^2}{4 \cdot 10} \Leftrightarrow \frac{\tilde{\delta}R^2}{4} \geq 10 \delta.
 $$
 This implies that for $z=(t,x) \in  B(0,5R/2) \backslash B(0,R/2) \cap \{-9 \delta \leq \psi \leq 2 \delta \}  $ one has
 $$
 \psi \geq -9 \delta \Rightarrow \phi \geq \tilde{\delta}|z|^2-9\delta \geq \tilde{\delta}\frac{R^2}{4}-9\delta\geq \delta.
 $$
We choose then $\rho=\frac{\delta}{10}$ and \eqref{geom conv 1} is satisfied. For the second condition we consider again $z=(t,x) \in \{\delta/4 \leq \psi \leq 2 \delta \}$ and we have 
$$
\psi \geq \frac{\delta}{4} \Rightarrow \phi \geq \frac{\delta}{4} >2 \rho=\frac{\delta}{5},
$$
which shows that \eqref{geom conv 2} is satisfied as well. The last property is simply a continuity statement. Indeed, since $\psi(0)=\phi(0)=0$ and $\psi$ is continuous there exists $0<r<R/2$ sufficiently small such that 
$$
B(0,2r) \Subset  \{-\delta/2 \leq \psi \leq \delta/2\} \cap B(0,R).
$$
We choose $\delta \leq \min (\frac{\tilde{\delta}R^2}{4 \cdot 10},\frac{d}{8})$ and , with $\rho=\delta/10$ and $r$ as above, the proposition is proved.
\end{proof}

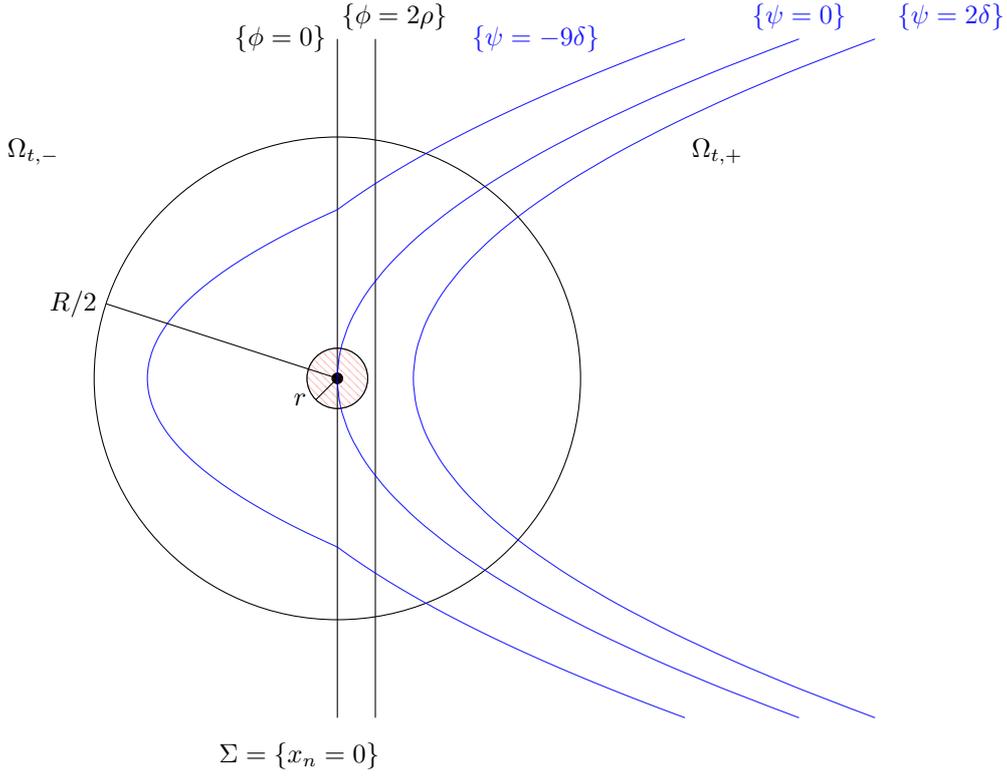
\begin{figure}[!ht]
	\centering
	
	\begin{tikzpicture}
	\draw[scale=0.5, domain=-9:9, smooth, variable=\x] plot (0,\x) ;
	\node at ( -0.75,4.5) {$\{\phi=0\}$};
		\node at (-0.5,-5)  {$\Sigma=\{x_n=0\}$};
		
		\node at ( -4,3) {$\Omega_{t,-}$};	
		
		\node at ( 5,3) {$\Omega_{t,+}$};
	
	\draw[scale=0.5, domain=-9:9, smooth, variable=\x] plot (1,\x) 	node[above, xshift=0.25cm] {$\{\phi=2 \rho\}$};

\filldraw[black] (0,0) circle (2pt)  ;
	
\draw  (0,0) circle (0.4cm) ;

\draw[pattern=north west lines,opacity=.5, pattern color=red] (0,0) circle (0.4cm) ;
	
\draw  (0,0) circle (3.2cm) ;

\draw[scale=0.5, domain=-9:9, smooth, variable=\x,blue, opacity=0.9] plot (0.15*\x*\x,\x)  node[above] {$\{\psi=0 \}$} ;

\draw[scale=0.5, domain=-9:9, smooth, variable=\x,blue, opacity=0.9] plot (0.15*\x*\x+2,\x)  node[above,xshift=1cm] {$\{\psi=2\delta \}$} ;

\draw[scale=0.5, domain=-4.47:4.47, smooth, variable=\x,blue, opacity=0.9] plot (0.25*\x*\x-5,\x ) ;

\draw[scale=0.5, domain=4.47:9, smooth, variable=\x,blue, opacity=0.9] plot (0.15*\x*\x-3,\x )  node[left, xshift=-1cm] {$\{\psi=-9 \delta \}$};

\draw[scale=0.5, domain=-9:-4.47, smooth, variable=\x,blue, opacity=0.9] plot (0.15*\x*\x-3,\x )  ;

\draw(0,0)--(-0.28,-0.28)  node[pos=1,left] {$r$} ;
	
\draw(0,0)--(-3.04,0.988)  node[left] {$R/2$};
		
\end{tikzpicture}
\caption{We have appropriately curved the level sets of $\phi$ with the help of the convexified weight  $\psi$, so that properties \eqref{geom conv 1} to \eqref{geom conv 3} of Proposition~\ref{starting point} are satisfied. Notice that the level sets of $\psi$ present a singularity when crossing the interface $\Sigma$.}
\label{geometric convexification}
\end{figure}

From this point on, one would like to follow the proof of Theorem 3.1 in \cite{Laurent_2018} from Step 2: Using the Carleman estimate (in the present setting Proposition~\ref{starting point}). The major difference is that in our context the coefficients of $P$ are no longer smooth, neither is the weight $\psi$. We show however that one can overcome this difficulty with a few modifications. Since this is a rather long and technical proof we only sketch the key arguments and explain, where necessary, what changes in our situation.

\begin{remark}
 Recall that the weight $\phi$ constructed in Section~\ref{the cerleman estimate} is Lipschitz continuous and in particular one has $\phi, \psi \in W^{1, \infty}(B(0,4R)).$
\end{remark}

\begin{remark}
With the notation introduced in Section~\ref{the cerleman estimate} one has 
	\begin{equation*}
	\norm{Q_{\epsilon, \tau}^{\psi} Pu}{L^2(\Omega_{t,-}\cup \: \Omega_{t,+})}{2}=\norm{H_-\poidsps P^-u_-}{L^2}{} +  \norm{H_+\poidsps P^+u_+}{L^2}{}.
	\end{equation*}
	
\end{remark}

We suppose to simplify that $(t_0,x_0)=0 \in \Sigma$. We consider then $\phi, \psi, R, d, \tau_0, C$ as given by Proposition~\ref{starting point}. We shall use the localization and regularization parameters $\lambda, \mu >0$ and we will suppose that $\lambda \sim \mu$, that is 
$$
1/\tilde{C} \mu \leq \lambda \leq \tilde{C} \mu,
$$
for some $\tilde{C}>0$.

\bigskip

We introduce now some cut-off functions that will allow us to localize and apply our Carleman estimate. We define $\chi(s) $ as a smooth function supported in $(-8,1)$ such that $\chi(s)=1$ for $s\in [-7, 1/2]$ and set 
\begin{equation}
    \label{def of chi delta}
    \chi_\delta (s):=\chi(s/\delta).
\end{equation}
We define as well $\tchi$ with $\tchi=1$ on $(-\infty, 3/2)$ and supported in $s \leq 2$, then $\tchi_\delta (s):=\tchi(s/\delta)$. 

\bigskip

Following \cite{Laurent_2018}, for $u \in \mathcal{W}$ compactly supported we wish to apply our Carleman estimate of Proposition~\ref{starting point} to
$$
\sigma_{2R}\sigma_{R, \lambda}\tchi_{\delta}(\psi)\chi_{\delta,\lambda}(\psi)u,
$$
where we recall that $\sigma_R$ has been defined in \eqref{definition of function sigma}. Note that, even though $u\in \mathcal{W}$ the function $\sigma_{2R}\sigma_{R, \lambda}\tchi_{\delta}(\psi)\chi_{\delta,\lambda}(\psi)u$ does not satisfy the homogeneous transmission conditions. This is why we need to consider non homogeneous transmission conditions.

One should notice that the fact that the operator $e^{-\frac{|D_t|^2}{\lambda}}$ is tangential with respect to the variable $x_n$ implies that
$$
f \in \tranphi \Longrightarrow f_\lambda \in \mathcal{W}^{\theta_\lambda,\Theta_\lambda}_\phi.
$$
Now the definition of  $\sigma_R$ in~\eqref{definition of function sigma} gives
$
(\partial_{x_n}\sigma_R){_{|\Sigma}}=0
$
and $(\partial_{x_n}\sigma_{R,\lambda})_{|\Sigma}=0$. This is true for $\tchi_{\delta}(\psi)$ also. To see this, we observe that by definition the derivative $\partial_{x_n} \tchi_{\delta}(\psi)$ is supported in $\{\psi \geq 3/2 \delta \}$ which according to the definition of $\psi$ in \eqref{def of weight psi} is away from the interface $\Sigma$. However, the term $\chi_{\delta}(\psi) $ may not be constant on $\Sigma$. More precisely, we have that:
$$
\sigma_{2R}\sigma_{R, \lambda}\tchi_{\delta}(\psi)\chi_{\delta,\lambda}(\psi)u \in \tran,
$$
with $\theta=0$ (since $\sigma_{2R}\sigma_{R, \lambda}\tchi_{\delta}(\psi)\chi_{\delta,\lambda}(\psi)u$ remains continuous) and 
$$
\Theta=\bigg((c_+-c_-)\sigma_{2R}\sigma_{R, \lambda}\tchi_{\delta}(\psi)u\partial_{x_n}(\chi_{\delta,\lambda})(\psi)\bigg)_{|\Sigma}.
$$
Notice that the definition of $\chi_\delta$ implies that 
$$
\supp{\partial_{x_n} \chi_\delta}_{|\Sigma} \subset \Sigma \cap \{-8\delta \leq \psi \leq -7 \delta\}.
$$
This support property combined with Lemma~\ref{lemma 2.13 from ll} allow to estimate the term $T_{\theta,\Theta}$ appearing in the left hand side of the estimate of Proposition~\ref{starting point} in the following way:
\begin{align}
    T_{\theta, \Theta}&=\tau \normsurf{e^{-\epsilon\frac{|D_t|^2}{2 \tau}}e^{\tau \psi} \Theta}{\ls}{2} \lesssim \tau \normsurf{e^{\tau \psi} \Theta}{\ls}{2}=\tau \normsurf{e^{\tau \psi} \bigg((c_+-c_-)\sigma_{2R}\sigma_{R, \lambda}\tchi_{\delta}(\psi)u\partial_{x_n}(\chi_{\delta,\lambda})(\psi)\bigg)_{|\Sigma}}{\ls}{2} \nonumber \\ \label{estim for transmission term}
    &\lesssim \tau \mu e^{-14 \delta \tau}e^{\frac{\tau^2}{\mu}} \normsurf{u}{\ls}{2} \lesssim \tau \mu e^{-14 \delta \tau}e^{\frac{\tau^2}{\mu}} \norm{u}{H^1}{2}.
\end{align}
The other term in the left hand side of Proposition~\ref{starting point} that we need to estimate is $$
\norm{\poidsps P \sigma_{2R}\sigma_{R, \lambda}\tchi_{\delta}(\psi)\chi_{\delta,\lambda}(\psi)u}{L^2(\Omega_{t,-}\cup \: \Omega_{t,+})}{}.
$$
We use again $\supp{\chi_\delta} \subset (-\infty, \delta)$ with Lemma~\ref{lemma 2.13 from ll} to obtain
\begin{align}
\label{elemantary estimate with commutator}
 \norm{\poidsps P \sigma_{2R}\sigma_{R, \lambda}\tchi_{\delta}(\psi)\chi_{\delta,\lambda}(\psi)u}{L^2 (\Omega_{t,-}\cup \: \Omega_{t,+})}{} &\leq    \norm{\poidsps  \sigma_{2R}\sigma_{R, \lambda}\tchi_{\delta}(\psi)\chi_{\delta,\lambda}(\psi)Pu}{L^2 (\Omega_{t,-}\cup \: \Omega_{t,+})}{} \nonumber\\ &\hspace{4mm}+ \norm{\poidsps  [\sigma_{2R}\sigma_{R, \lambda}\tchi_{\delta}(\psi)\chi_{\delta,\lambda}(\psi),P]u}{L^2 (\Omega_{t,-}\cup \: \Omega_{t,+})}{} 
\nonumber\\ &\lesssim \mu^{1/2}e^{C \frac{\tau^2}{\mu}}e^{\delta \tau} \norm{Pu}{L^2(B(0,4R)\cap (\Omega_{t,-}\cup \: \Omega_{t,+}))}{} \nonumber\\
 &\hspace{4mm}+ \norm{\poidsps  [\sigma_{2R}\sigma_{R, \lambda}\tchi_{\delta}(\psi)\chi_{\delta,\lambda}(\psi),P]u}{L^2 (\Omega_{t,-}\cup \: \Omega_{t,+})}{}.
\end{align}
We need therefore to estimate the commutator appearing in~\eqref{elemantary estimate with commutator}. This is the purpose of the following Lemma.

\begin{lem}
\label{commutator estimate }
There exists $R_0 >0$ such that for any $R \in (0,R_0)$ there exist $r, \rho>0$ such that for any $\theta \in C^{\infty}_0(\R^{n+1})$ such that $\theta(x)=1$ on a neighborhood of $\{\phi \geq 2 \rho\} \cap B(0,3R)$, there exist $C>0$, $c>0$ and $N>0$ such that 
\begin{align}
\label{comm estimate}
   C e^{2 \delta \tau}\norm{M^{2\mu}_{\lambda}\theta_\lambda u}{H^1}{}&+C\mu^{1/2}\tau\bigg(e^{-8\delta}+e^{-\frac{\epsilon \mu^2}{8 \tau}}+e^{-c \mu}e^{\delta \tau}\bigg)e^{C \frac{\tau^2}{\mu}}e^{\delta \tau}\norm{u}{H^1}{} \\
  &\hspace{10mm}\geq 
  \norm{\poidsps  [\sigma_{2R}\sigma_{R, \lambda}\tchi_{\delta}(\psi)\chi_{\delta,\lambda}(\psi),P]u}{L^2 (\Omega_{t,-}\cup \: \Omega_{t,+})}{},
\end{align}
for any $u \in \mathcal{W} $ compactly supported, $\mu \geq 1$, $\lambda \sim \mu$ and $\tau \geq 1$.

\end{lem}

This is the difference with respect to the situation of Lemma 3.7 in \cite{Laurent_2018}: the coefficients $p_\alpha$ are no longer smooth. They remain bounded however and 
the regularization-localization operators are tangential to the interface where the coefficients may jump. This allows to use the same techniques. Moreover in our situation we can exploit the fact that the coefficients $p_\alpha$ are independent of $t$ and therefore they commute with $\poidsps$.

\begin{proof}
The preceding remark allows us to write our operator as
$$
P=\sum_{|\alpha|\leq 2}p_{\alpha}(x)\partial^\alpha,
$$
with $p_{\alpha}=H_-p^{-}_{\alpha}+H_+p^{+}_{\alpha}$ and $p^{\pm}_{\alpha}=p^{\pm}_{\alpha}(x)$ smooth functions independent of $t$.

 This implies:
\begin{align}
\label{sum of commutators}
  \sum_{\pm}\norm{H_{\pm}\poidsps  [\sigma_{2R}\sigma_{R, \lambda}\tchi_{\delta}(\psi)\chi_{\delta,\lambda}(\psi),P_{\pm}]u_\pm}{L^2}{}&=\sum_{|\alpha|\leq 2}  \norm{H_{\pm}\poidsps  [\sigma_{2R}\sigma_{R, \lambda}\tchi_{\delta}(\psi)\chi_{\delta,\lambda}(\psi),p_{\alpha}(x)\partial^{\alpha}]u}{L^2}{}\nonumber\\
  &=\sum_{|\alpha|\leq 2}  \norm{\poidsps p_{\alpha}(x) [\sigma_{2R}\sigma_{R, \lambda}\tchi_{\delta}(\psi)\chi_{\delta,\lambda}(\psi),\partial^{\alpha}]u}{L^2}{} \nonumber\\
  &=\sum_{|\alpha|\leq 2}  \norm{ p_{\alpha}(x)\poidsps  [\sigma_{2R}\sigma_{R, \lambda}\tchi_{\delta}(\psi)\chi_{\delta,\lambda}(\psi),\partial^{\alpha}]u}{L^2}{} \nonumber\\
  &\leq C\sum_{|\alpha|\leq 2}  \norm{\poidsps  [\sigma_{2R}\sigma_{R, \lambda}\tchi_{\delta}(\psi)\chi_{\delta,\lambda}(\psi),\partial^{\alpha}]u}{L^2}{}.
\end{align}

By the Leibniz rule we can write 
\begin{multline*}
\partial^\alpha(\sigma_{2R}\sigma_{R, \lambda}\tchi_{\delta}(\psi)\chi_{\delta,\lambda}(\psi)u)\\
=\sum_{\alpha_1+\alpha_2+\alpha_3+\alpha_4+\alpha_5=\alpha}C_{(\alpha_i)}\partial^{\alpha_1}(\chi_{\delta,\lambda}(\psi))\partial^{\alpha_2}(\sigma_{2R})\partial^{\alpha_3}(\sigma_{R, \lambda})\partial^{\alpha_4}(\tchi_{\delta}(\psi))\partial^{\alpha_5}u.
\end{multline*}

We split then the commutators in ~\eqref{sum of commutators} in  a sum of differential operators of order one as follows:
\begin{equation*}
\sum_{|\alpha|\leq 2} [\partial^\alpha,\sigma_{2R}\sigma_{R, \lambda}\tchi_{\delta}(\psi)\chi_{\delta,\lambda}(\psi)]
=B_1+B_2+B_3+B_4,
\end{equation*}
where:

\begin{enumerate}
    \item $B_1$ contains the terms with $\alpha_1 \neq 0$ and $\alpha_2=\alpha_4=0$;
    
    \item $B_2$ contains some terms with $\alpha_2 \neq 0$;
    
    \item $B_3$ contains the terms with $\alpha_3 \neq0$ and $\alpha_1=\alpha_2=\alpha_4=0$;
    
    \item $B_4$ contains some terms with $\alpha_4 \neq 0$.
\end{enumerate}
Now we provide estimates for each of the terms above. Before that we further decompose $B_1$ in two terms by observing 
$$
(\chi_{\delta, \lambda})^{\prime}=\frac{1}{\delta}(\chi^\prime)_{\delta, \lambda}=\frac{1}{\delta}(\mathds{1}_{[\delta/2,\delta]}\chi^\prime)_{\delta, \lambda}+\frac{1}{\delta}(\mathds{1}_{[-8\delta,-7\delta]}\chi^\prime)_{\delta, \lambda}=\chi_{\delta, \lambda}^{+}+\chi_{\delta, \lambda}^{-},
$$
with $\chi_{\delta, \lambda}^{+}:=\frac{1}{\delta}(\mathds{1}_{[\delta/2,\delta]}\chi^\prime)_{\delta, \lambda}$ and $\chi_{\delta, \lambda}^{-}:=\frac{1}{\delta}(\mathds{1}_{[-8\delta,-7\delta]}\chi^\prime)_{\delta, \lambda}$, and we have used the properties of $\supp{\chi^\prime}$. 

\begin{enumerate}
\item This allows to decompose $B_1$ as a sum of generic terms of the form
\begin{equation}
\label{def of B+-}
B_{\pm}=b_{\pm}\partial^\gamma=f \sigma_{2R}\partial^{\beta}(\sigma_{R,\lambda})\chi_{\delta, \lambda}^{\pm}\tchi_\delta(\psi)\partial^\gamma,	
\end{equation}

where $|\beta|, |\gamma| \leq 1$, $f \in L^\infty(\R^{n+1})$, compactly supported and analytic in $t$. Notice that in the absence of the regularization parameter $\lambda$ the terms of $B_+$ would be supported in 
$$
\bigg(\{\delta/2 \leq \psi \leq \delta \} \cap B(0,2R)\bigg) \subset \bigg(\{\phi > 2 \rho \} \cap \{\psi \leq \delta \} \cap B(0,2R)\bigg)
$$
and those of $B_-$ in $\{-8\delta\leq \psi \leq-7\delta \} \cap B(0,2R).$

\item $B_2$ consists of terms where there is at least one derivative on $\sigma_{2 R}$ and contains terms of the form
$$
\tilde{b}\partial^{\beta}(\sigma_{R, \lambda})(\chi_{\delta,\lambda})^{(k)}(\psi)\tchi_\delta(\psi)\partial^\gamma,
$$
where $k, |\beta|, |\gamma| \leq 1$ with $\tilde{b}$ bounded and supported in $B(0,4R)\backslash B(0,2R)$.

\item $B_3$ consists of terms where there is at least one derivative on $\sigma_{R, \lambda}$ and none on $\chi_{\delta,\lambda}(\psi), \tchi_\delta(\psi)$ and $\sigma_R$. These are terms of the form 
$$
f \sigma_{2R}\partial^{\beta}(\sigma_{R, \lambda})\chi_{\delta, \lambda}(\psi)\tchi_\delta(\psi)\partial^\gamma,
$$
where $f$ is bounded and independent of $t$, $|\beta|=1$ and $|\gamma|\leq 1$. Notice that in the absence of the regularization parameter $\lambda$ these terms would be supported in 
$$
\bigg( \{-8 \delta \leq \psi \leq -7\delta \}\cap B(0,2R)\backslash B(0,R) \bigg) \subset \bigg(\{\phi > 2\rho\}\cap \{\psi \leq \delta\} \cap B(0,2R)\bigg).
$$
\item $B_4$ consists of terms where there is at least one derivative on $\tchi_\delta(\psi)$ and contains terms of the form
$$
\tilde{b}\partial^{\beta}(\sigma_{R,\lambda})(\chi_{\delta,\lambda})^{(k)}(\psi)\partial^\gamma,
$$
where $k, |\beta|, |\gamma|\leq 1 $ and the function $\tilde{b}$ is bounded and supported in $B(0,4R) \cap \{3  \delta/2 \leq \psi \leq 2 \delta\}$.
\end{enumerate}

To prove Lemma~\ref{commutator estimate } one needs to provide appropriate estimates for a generic term from each of the four groups defined above. What happens is that terms containing derivatives of a non regularized function (that is without the subscript $\lambda$) are easier to handle since they localize exactly. To deal with derivatives of regularized functions requires more work since they produce additional errors coming from the non exact localization properties.

\bigskip

\noindent \textbf{Estimating $B_-$} (defined in~\ref{def of B+-}). We use Lemma~\ref{lemma 2.13 from ll} applied to $\chi^{-}_\delta$ to find:
\begin{equation*}
    \norm{\poidsps B_- u}{L^2}{} \leq \norm{e^{\tau \psi}B_-}{L^2}{}\leq C_{\delta}\lambda^{1/2}e^{-7 \delta}e^{\frac{\tau^2}{\lambda}}\norm{u}{H^1}{}\leq C\mu^{1/2}e^{-7 \delta}e^{C \frac{\tau^2}{\mu}}\norm{u}{H^1}{}.
\end{equation*}

\bigskip

\noindent\textbf{Estimating $B_2$.} We use Lemma~\ref{lemma 2.3 from ll} applied to $\tilde{b}$ and $\partial^{\beta}(\sigma_{R, \lambda})$ which have supports away from each other. We apply then Lemma~\ref{lemma 2.13 from ll} to $\chi_\delta^{(k)}$ using its support properties to find:
\begin{align*}
    \norm{\poidsps B_2 u}{L^2}{}&\leq\norm{e^{\tau \psi}B_2 u}{L^2}{}\leq \norm{\tilde{b}\partial^\beta(\sigma_{R,\lambda})}{L^\infty}{}\norm{\chi_\delta^{(k)}u}{L^2}{}\leq C \lambda^{1/2}e^{\delta \tau}e^{\frac{\tau^2}{\lambda}e^{-c \lambda}}\norm{u}{H^1}{}\\
    &\leq C \mu^{1/2}e^{\delta \tau}e^{C \frac{\tau^2}{\mu}}e^{-c\mu}\norm{u}{H^1}{}.
\end{align*}

\bigskip

\noindent \textbf{Estimating $B_4$.} We use Lemma~\ref{lemma 2.3 from ll} applied to $(\chi_{\delta,\lambda})^{(k)}(\psi)$ and $\mathds{1}_{[3\delta/2,2\delta]}$ to find thanks to the localization of $\tchi_\delta^\prime(\psi)$:
$$
\norm{\poidsps B_4 u}{L^2}{}\leq \norm{e^{\tau \psi}B_4 u}{L^2}{}\leq C e^{2 \delta \tau}e^{-c\mu}\norm{u}{H^1}{}.
$$

\bigskip 

\noindent \textbf{First estimates on $B_+$} (defined in~\eqref{def of B+-}) \textbf{and $B_3$.} These are the most difficult terms since here the derivative does not localize \textit{exactly}. We have for $B_*$ with $*=+$ or $*=3$:
\begin{align*}
    \norm{\poidsps B_* u}{L^2}{}&=\norm{e^{-\epsilon \frac{|D_t|^2}{2  \tau}}e^{\tau \psi}B_* u}{L^2}{}\leq \norm{e^{-\epsilon \frac{|D_t|^2}{2  \tau}}M^\mu_\lambda e^{\tau \psi}B_* u}{L^2}{}+\norm{e^{-\epsilon \frac{|D_t|^2}{2  \tau}}(1-M^\mu_\lambda) e^{\tau \psi}B_* u}{L^2}{}\\
    &\leq \norm{M^\mu_\lambda e^{\tau \psi}B_* u}{L^2}{}+C\lambda^{1/2}\left(e^{-\frac{\epsilon \mu^2}{8 \tau}} +e^{- c \mu}\right)e^{C \frac{\tau^2}{\mu}}e^{\delta \tau}\norm{u}{H^1}{},
\end{align*}
where we have applied successively Lemma~\ref{lemma 2.14 from ll} and  Lemma~\ref{lemma 2.13 from ll}. We estimate now the first term in the above inequality, with $B_*=b_*\partial^\gamma$:
$$
\norm{M^\mu_\lambda e^{\tau \psi}B_* u}{L^2}{}\leq\norm{M^\mu_\lambda e^{\tau \psi}b_*(1-M_\lambda^{2\mu}) \partial^\gamma u}{L^2}{}+\norm{M^\mu_\lambda e^{\tau \psi}b_*M_\lambda^{2\mu} \partial^\gamma u}{L^2}{}.
$$
We apply then Lemma~\ref{lemma 2.16 from ll} which gives
$$
\norm{M^\mu_\lambda e^{\tau \psi}b_*(1-M_\lambda^{2\mu}) \partial^\gamma u}{L^2}{}\leq C \tau^N e^{C \frac{\tau^2}{\mu}}e^{2 \delta \tau - c \mu}\norm{u}{H^1}{}.
$$
Using the fact that
$$
\norm{M^\mu_\lambda e^{\tau \psi}b_*M_\lambda^{2\mu} \partial^\gamma u}{L^2}{}\leq \norm{ e^{\tau \psi}b_*M_\lambda^{2\mu} \partial^\gamma u}{L^2}{}
$$
we have thus obtained
$$
 \norm{\poidsps B_* u}{L^2}{}\leq \norm{ e^{\tau \psi}b_*M_\lambda^{2\mu} \partial^\gamma u}{L^2}{}+C\mu^{1/2}\tau^N\left(e^{-\frac{\epsilon\mu^2}{8 \tau}}+e^{\delta \tau}e^{-c\mu}\right)e^{C \frac{\tau^2}{\mu}}e^{\delta \tau}\norm{u}{H^1}{}.
$$
That is we “almost commuted”$M_\lambda^\mu$ with $e^{\tau \psi}B_*$. To finish the proof of Lemma~\ref{commutator estimate } we need therefore to estimate $ \norm{ e^{\tau \psi}b_*M_\lambda^{2\mu} \partial^\gamma u}{L^2}{}$, for $b_*=b_+$ and $b_*=b_3$.

This is done exactly as in \cite{Laurent_2018}. As we have already seen in the course of this proof, in our case $b_*$ is less regular. However, $u \in H^1$, and $b_* \in L^\infty $ satisfies the same localization properties as in \cite{Laurent_2018}. For the sake of completeness we sketch the estimate for $b_+$.

\bigskip

\noindent \textbf{Estimating $B_+$.} A generic term of $B_+$ has the form 
$$
b_+ \partial^\gamma=f \tilde{\tilde{b}}_\lambda\chi_{\delta, \lambda}^+(\psi)\tchi(\psi)\partial^\gamma,
$$
where $\tilde{\tilde{b}}=\partial^\beta(\sigma_R)$, $|\beta|\leq 1$, is supported in $B(0,2R)$ and $f$ is bounded. We decompose $\R^{n+1}$ as 
\begin{align*}
    \R^{n+1}& =O_1\cup O_2\cup O_3, \quad \textnormal{with} \\
    O_1&=\{\psi \notin [\delta/4, 2 \delta ] \}\cap B(0,5R/2),\\
    O_2&=B(0,5R/2)^c, \\
    O_3&= \{\psi\in [\delta/4,2\delta]\}\cap B(0,5R/2).
\end{align*}

For the region $O_1$ we use the fact that $\chi_\delta^+$ is supported in $[\delta/2, \delta]$ and then Lemma~\ref{lemma 2.3 from ll} with $f_2=\mathds{1}_{[\delta/4, 2\delta]^c}$. For the region $O_2$ we exploit as well the almost localization by using Lemma~\ref{lemma 2.3 from ll} and Lemma~\ref{lemma 2.13 from ll}.

For the region $O_3$ we start by noticing that thanks to the geometric convexification property \eqref{geom conv 2} one can find a smooth $\tilde{\theta}$ with $\tilde{\theta}=1$ on a neighborhood of $O_3$ and supported in $\{\phi>2 \rho\} \cap B(0,3R)$. We estimate then
$$
\norm{e^{\tau \psi}b_+ M^{2\mu}_\lambda \partial^\gamma u}{L^2(O_3)}{}\leq C e^{\delta \tau } \norm{M^{2\mu}_\lambda \partial^\gamma u}{L^2(O_3)}{}\leq C e^{\delta \tau } \norm{\tilde{\theta}_\lambda M^{2\mu}_\lambda \partial^\gamma u}{L^2}{}.
$$
The final step is to commute $\theta_\lambda$ with $M^{2\mu}_\lambda$. Let $\tilde{\tilde{\theta}} \in C^\infty_0$ be such that $\tilde{\tilde{\theta}}=1$ on a neighborhood of $\supp{\tilde{\theta}}$ and supported in $\{\phi>2\rho \cap B(0,3R)\}$. Now recall that from the assumption of Theorem~\ref{local quant estimate} we are given $\theta \in C^\infty_0$ with $\theta=1$ on $\{\phi> 2 \rho\} \cap B(0,3R)$ and consequently one has that $\theta=1$ in a neighborhood of $\supp{\tilde{\tilde{\theta}}}$. We use then
Lemma~\ref{lemma 2.6 from ll} which gives
$$
\norm{\tilde{\theta}_\lambda M_\lambda^{2\mu}\partial^\gamma u }{L^2}{}\leq C\norm{\tilde{\tilde{\theta}}M_\lambda^{2\mu}u}{H^1}{}+Ce^{-c \lambda}\norm{u}{H^1}{},
$$
and then Lemma~\ref{lemma 2.11 from ll}:
$$
\norm{\tilde{\tilde{\theta}}M_\lambda^{2\mu}u}{H^1}{}\leq \norm{M^{2\mu}_\lambda \theta_\lambda u}{H^1}{}+Ce^{-c \mu}\norm{u}{H^1}{}.
$$
Consequently, we have obtained in each region $O_*$ an estimate of the same type as is~\eqref{comm estimate}. This gives the estimate for $B_+$.

\bigskip

\noindent \textbf{Estimating $B_3$.} This is done in a similar manner to $B_+$.

\bigskip

We have finally obtained estimates of the type of \eqref{comm estimate} for all of the generic terms of the commutator. This proves Lemma~\ref{commutator estimate }.
\end{proof}

With the commutator estimate at hand we can now show:

\begin{lem}
\label{application of Carleman+ commutator}
There exists $R_0 >0$ such that for any $R \in (0,R_0)$ there exist $r, \rho>0$ such that for any $\theta \in C^\infty_0(\R^{n+1})$ such that $\theta(x)=1$ on a neighborhood of $\{\phi>2 \rho \}\cap B(0,3R),$ there exist $\mu_0, C, c, N>0$ such that:
\begin{align*}
    C\mu^{1/2}e^{C \frac{\tau^2}{\lambda}}e^{\delta\tau}\norm{Pu}{L^2(B(0,4R))}{}&+C\mu^{1/2}\tau^N \left(e^{-8\delta \tau }+e^{-\frac{\epsilon\mu^2}{8\tau}}+e^{\delta\tau -c\mu}\right) e^{C \frac{\tau^2}{\mu}}e^{\delta \tau}\norm{u}{H^1}{}\\
    &\hspace{4mm}+Ce^{2\delta\tau}\norm{M^{2\mu}_\lambda\theta_\lambda u}{H^1}{}\geq \tau \norm{\poidsps\sigma_{2R}\sigma_{R, \lambda}\tchi_{\delta}(\psi)\chi_{\delta,\lambda}(\psi)u}{H^1_\tau}{},
\end{align*}
for $u \in \mathcal{W}$ compactly supported, $\mu \geq \mu_0, \lambda \sim \mu$ and $\tau \geq \tau_0$.

\end{lem}

\begin{proof}
We apply the estimate of Proposition~\ref{starting point} to $w:=\sigma_{2R}\sigma_{R, \lambda}\tchi_{\delta}(\psi)\chi_{\delta,\lambda}(\psi)u$. This gives:

 \begin{equation*}
    C\left( \norm{Q_{\epsilon, \tau}^{\psi} Pw}{L^2(\Omega_{t,-}\cup \: \Omega_{t,+})}{2} +e^{-d \tau} \norm{e^{\tau \psi}w }{H^1_\tau}{2} + T_{\theta, \Theta}\right) \geq \tau \norm{Q_{\epsilon, \tau}^{\psi}w}{H^1_\tau}{2},
\end{equation*}
We need therefore to estimate the three terms appearing in the left hand side of the above inequality.

\begin{itemize}
    \item For the term $\norm{Q_{\epsilon, \tau}^{\psi} Pw}{L^2(\Omega_{t,-}\cup \: \Omega_{t,+})}{2}$ we simply combine estimate \eqref{elemantary estimate with commutator} with Lemma~\ref{commutator estimate }.

    \item For the term $T_{\theta,\Theta}$ we use the estimate \eqref{estim for transmission term}.

    \item It remains to deal with $e^{-d \tau}\norm{e^{\tau \psi}w }{H^1_\tau}{}$. Recall that $\chi_\delta$ is supported in $(-8\delta,\delta)$. This implies, using Lemma~\ref{lemma 2.13 from ll}:
\begin{align*}
  e^{-d \tau}\norm{e^{\tau \psi}w }{H^1_\tau}{}&=  e^{-d \tau} \norm{e^{\tau \psi}\sigma_{2R}\sigma_{R, \lambda}\tchi_{\delta}(\psi)\chi_{\delta,\lambda}(\psi)u }{H^1_\tau}{} \\
  &\leq C e^{- d \tau}\lambda^{1/2}\tau e^{\delta \tau}e^{\frac{\tau^2}{\lambda}}\norm{u}{H^1}{}\\
  &\leq \mu^{1/2}\tau e^{(\delta-d)\tau}e^{C\frac{\tau^2}{\mu}}\norm{u}{H^1}{}\\
  &\leq \mu^{1/2}\tau e^{-7\delta\tau}e^{C\frac{\tau^2}{\mu}}\norm{u}{H^1}{},
\end{align*}
where for the last inequality we used the fact $\delta \leq \frac{d}{8}$, thanks to Proposition~\ref{starting point}.
\end{itemize}

This finishes the proof of Lemma~\ref{application of Carleman+ commutator}.
\end{proof}

\bigskip

\noindent \textbf{The complex analysis argument}

\bigskip

The final step of the proof of Theorem~\ref{local quant estimate} consists in transferring the estimate provided by Lemma~\ref{application of Carleman+ commutator} from $$\norm{\poidsps\sigma_{2R}\sigma_{R, \lambda}\tchi_{\delta}(\psi)\chi_{\delta,\lambda}(\psi)u}{H^1_\tau}{}$$ to $\norm{M^{\beta\mu}_{c_1 \mu}\sigma_{r,c_1 \mu}u}{H^1}{}$. The presence of the microlocal weight $\poidsps$ makes this part highly non-trivial and one has to work by duality. In \cite{Hor:97, Tataru:95} the authors prove a qualitative unique continuation result. They use the following strategy:

\begin{enumerate}
    \item For any test function $f$ define the distribution $h_f= \psi_{*}(fu)$ by
    $$
    \langle h_f, w \rangle_{\mathcal{E}^\prime(\R), C^\infty(\R)}= \langle fu, w(\psi) \rangle_{\mathcal{E}^\prime(\R^{n+1}), C^\infty(\R^{n+1})}.
    $$
That measures $fu$ along the level sets of $\psi$.

\item Consider the Fourier transform of $h_f$ and use

\begin{itemize}

\item the Carleman estimate to bound the quantity $\widehat{h_f}(i\tau)$ for $\tau $ large

\item an \textit{a priori} estimate on $\widehat{h_f}(\zeta)$ for $\zeta \in \mathbb{C}$ which gives sub-exponential growth
\end{itemize}
Thanks to a Phragmén–Lindelöf theorem transfer the estimate provided by the Carleman estimate from the upper imaginary axis to the whole upper plane.

\item Use a Paley-Wiener theorem to deduce from the bound obtained for the Fourier transform of $h_f$ an information about its support.
    
\end{enumerate}

Now in \cite{Laurent_2018} both the Phragmén–Lindelöf and the Paley-Wiener theorem are replaced by some precise estimates for $\widehat{h_f}$. In our case the proof works in exactly the same way. The only difference is that here $\psi$ and $u$ are no longer smooth but they are  Lipschitz continuous. However this does not affect the proof.

\begin{lem}
\label{complex analysis lemma}
Under the above assumptions, there exists $\tilde{\tau_0}$ such that for any $\kappa, c_1>0$, there exist $\beta_0, C, c>0$, such that for any $0<\beta <\beta_0$, for all $\mu \geq \frac{\tilde{\tau_0}}{\beta}$ and $u \in \mathcal{W}$ compactly supported, one has:
\begin{align*}
    &Ce^{-c \mu}\left(e^{\kappa \mu}\left(\norm{M^{2\mu}_\lambda \theta_\lambda u}{H^1}{} +\norm{Pu}{L^2(B(0,4R))}{}\right)+\norm{u}{H^1}{} \right)
    \\& \hspace{8mm}\geq \norm{M^{\beta \mu}\sigma_{2R}\sigma_{R, \lambda}\tchi_{\delta}(\psi)\chi_{\delta,\lambda}(\psi)\eta_{\delta,\lambda}(\psi)u}{H^1}{},
\end{align*}
with $\lambda=2 c_1 \mu,$ and 
$$
\eta \in C^{\infty}_0((-4,1)), \quad \eta=1 \: \textnormal{in } \: [-1/2,1/2] \quad \textnormal{and}  \quad \eta_{\delta}(s):=\eta(s/\delta).
$$
\end{lem}

\begin{proof}
Recall that $\psi \in H^1$ (see \eqref{def of weight psi}). For any test function $f \in \mathscr{S}(\R^{n+1})$ we define then the distribution
$$
\langle h_f, w \rangle_{\mathcal{E}^\prime(\R), C^{\infty}(\R)}:= \langle(M^{\beta \mu}f \sigma_{2R}\sigma_{R, \lambda}\tchi_{\delta}(\psi)\chi_{\delta,\lambda}(\psi), w(\psi)\rangle_{H^{-1}(\R^{n+1}),H^1_0(\R^{n+1})}.
$$
We work with $w=\eta_{\delta, \lambda}$ and estimate the quantity $
\langle h_f, \eta_{\delta, \lambda} \rangle_{\mathcal{E}^\prime(\R), C^{\infty}(\R)}.
$
This done exactly as in \cite{Laurent_2018}. We use the formula for the Fourier transform of a compactly supported distribution to obtain an a priori estimate on $\hat{h}_f(\xi)$ for $\xi \in \R$. Then we use Lemma~\ref{application of Carleman+ commutator} to obtain an estimate for $\widehat{h_f}(\zeta)$ for $\zeta= i \tau$. All the complex analysis arguments that follow remain valid in our context. Indeed, these arguments do not involve the $t,x$ space but they are carried out in the complexification of our Carleman large parameter $\tau$.
\end{proof}

\bigskip

\noindent \textit{End of the proof of Theorem~\ref{local quant estimate}.} We are ready to finish the proof the local quantitative estimate. The last thing we need to do is to estimate the term
$$
\norm{M^{\beta \mu}\sigma_{2R}\sigma_{R, \lambda}\tchi_{\delta}(\psi)\chi_{\delta,\lambda}(\psi)\eta_{\delta,\lambda}(\psi)u}{H^1}{}
$$
appearing in the right hand side of Lemma~\ref{complex analysis lemma}. This is done as in \cite{Laurent_2018}. Indeed, all the operations are tangential and thanks to our Proposition~\ref{starting point} the geometric context is the same is in ~\cite{Laurent_2018}. We sketch the end of the proof in a concise way.

We have:
\begin{align*}
    \norm{M^{\beta \mu}_\lambda \sigma_{2R}\sigma_{R, \lambda}\tchi_{\delta}(\psi)\chi_{\delta,\lambda}(\psi)\eta_{\delta,\lambda}(\psi)u}{H^1}{} &\leq \norm{M^{\frac{\beta \mu}{2}}_\lambda (1-M^{\beta \mu})\sigma_{R, \lambda}\tchi_{\delta}(\psi)\chi_{\delta,\lambda}(\psi)\eta_{\delta,\lambda}(\psi)u}{H^1}{}\\
    &\hspace{4mm}+\norm{M^{\frac{\beta \mu}{2}}_\lambda M^{\beta \mu}\sigma_{R, \lambda}\tchi_{\delta}(\psi)\chi_{\delta,\lambda}(\psi)\eta_{\delta,\lambda}(\psi)u}{H^1}{}.
\end{align*}
To control the first term we use Lemma~\ref{lemma 2.3 from ll}. For the second one we use Lemma~\ref{complex analysis lemma}. We find for $0<\beta < \beta_0$, for all $\mu \geq \frac{\tilde{\tau_0}}{\beta}$ and $\lambda=2 c_1 \mu$:
$$
 \norm{M^{\beta \mu}_\lambda \sigma_{2R}\sigma_{R, \lambda}\tchi_{\delta}(\psi)\chi_{\delta,\lambda}(\psi)\eta_{\delta,\lambda}(\psi)u}{H^1}{} \leq Ce^{-c \mu}\left(e^{\kappa \mu}\left(\norm{M^{2\mu}_\lambda \theta_\lambda u}{H^1}{} +\norm{Pu}{L^2(B(0,4R))}{}\right)+\norm{u}{H^1}{} \right).
$$
Next we combine the above estimate with Lemma~\ref{lemma 2.11 from ll}:
\begin{align*}
\norm{M^{\frac{\beta \lambda}{4}}\sigma_{r, \lambda}u}{H^1}{}&\leq    Ce^{-c \mu}\left(e^{\kappa \mu}\left(\norm{M^{2\mu}_\lambda \theta_\lambda u}{H^1}{} +\norm{Pu}{L^2(B(0,4R))}{}\right)+\norm{u}{H^1}{} \right)\\
&\hspace{4mm}+ \norm{\sigma_{r,\lambda}M^{\frac{\beta \mu}{2}}_\lambda \left(1-  \sigma_{2R}\sigma_{R, \lambda}\tchi_{\delta}(\psi)\chi_{\delta,\lambda}(\psi)\eta_{\delta,\lambda}(\psi)\right)u}{H^1}{},
\end{align*}
where $r$ is given by Proposition~\ref{starting point}. In particular Proposition~\ref{starting point} implies thanks to the property \eqref{geom conv 3} that $\sigma_R=\chi_\delta(\psi)=\tchi_\delta(\psi)=\eta_\delta(\psi)=1$ on a neighborhood of $\supp{\sigma_r}$. We can then finish the proof exactly as in \cite{Laurent_2018}. Indeed, we take $\Pi \in C^{\infty}_0$ with $\Pi=1$ on a neighborhood of $\supp{\sigma_r}$ and such that  $\sigma_{2R}=\sigma_R=\chi_\delta(\psi)=\tchi_\delta(\psi)=\eta_\delta(\psi)=1$ on a neighborhood of $\supp{\Pi}$. Then
\begin{multline*}
\norm{\sigma_{r,\lambda}M^{\frac{\beta \mu}{2}}_\lambda \left(1-  \sigma_{2R}\sigma_{R, \lambda}\tchi_{\delta}(\psi)\chi_{\delta,\lambda}(\psi)\eta_{\delta,\lambda}(\psi)\right)u}{H^1}{}\\
\leq \norm{\sigma_{r,\lambda}M^{\frac{\beta \mu}{2}}_\lambda \left(1-  \sigma_{2R}\sigma_{R, \lambda}\tchi_{\delta}(\psi)\chi_{\delta,\lambda}(\psi)\eta_{\delta,\lambda}(\psi)\right)(1-\Pi)u}{H^1}{}\\
+\norm{\sigma_{r,\lambda}M^{\frac{\beta \mu}{2}}_\lambda \left(1-  \sigma_{2R}\sigma_{R, \lambda}\tchi_{\delta}(\psi)\chi_{\delta,\lambda}(\psi)\eta_{\delta,\lambda}(\psi)\right)\Pi u}{H^1}{}.
\end{multline*}
We control from above the first term by using Lemma~\ref{lemma 2.10 from ll} and for the second one we combine Lemmata~\ref{lemma 2.3 from ll} and~\ref{lemma 2.5 from ll}.

This finishes the proof of Theorem~\ref{local quant estimate} up to renaming the constants appearing in the statement of the theorem.
\end{proof}

\subsection{Propagation of information and applications}
\label{propagation of info}

\bigskip

\noindent In this section we make use $\mathcal{L}(\mathcal{M}, E)$. This is the “largest distance” of a subset $E \subset \mathcal{M}$ to a point of $\mathcal{M}$ and has been defined in \eqref{greatest distance}.

\bigskip

We introduce now the tools of \cite[Section 4]{Laurent_2018}  in order to explain how one can propagate information by applying successively the local quantitative estimate.

\begin{definition}
Fix an open subset $\Omega$ of  $\R^{n+1}=\R_t \times \R^n_x$ and two \textit{finite} collections $(V_j)_{j\in J}$ and $(U_i)_{i \in I}$ of bounded open sets in $\R^{n+1}$. We say that $(V_j)_{j\in J}$ is \textit{under the dependence} of $(U_i)_{i \in I}$, denoted

$$
(V_j)_{j\in J}  \trianglelefteq  (U_i)_{i \in I},
$$
if for any $\theta_i \in C^\infty_0(\R^{n+1})$ such that $\theta_i (t,x)=1$ on a neighborhood of $\overline{U_i}$, for any $\tilde{\theta}_j \in C^\infty_0(V_j)$ and for all $\kappa, \alpha >0$, there exist $C, \kappa^\prime, \beta, \mu_0$ such that for all $\mu \geq \mu_0$ and $u \in \mathcal{W} $ compactly supported, one has:

$$
C\left(\sum_{i \in I} \norm {M_\mu^{\alpha \mu} \theta_{i,\mu} u}{H^1}{}+\norm{Pu}{L^2(\Omega_{t,+} \cup \Omega_{t,-})}{} \right)+Ce^{-\kappa^\prime \mu} \norm{u}{H^1}{}\geq
\sum_{j\in J} \norm{M_\mu^{\beta \mu} \tilde{\theta}_{j,\mu} u}{H^1}{},
$$
where $P$ is as defined in Section \ref{local setting}. 
\end{definition}

The motivation behind this definition becomes apparent when one looks at the local quantitative estimate of Theorem~\ref{local quant estimate}. If one forgets about the
localization and regularization indexes then the definition says simply that $V$ depends on $U$ if information on $U$ controls information on $V$, and this comes with a precise estimate. 

The Carleman estimate we have obtained (Theorem~\ref{theorem}) as well as Carleman estimates in general provide already some sort of quantitative estimate which propagates information. However, what is of fundamental importance here is that the relation of dependence as defined above (in fact in \cite{Laurent_2018}) can be \textbf{propagated} (in an optimal way). Indeed, the crucial property that one needs in order to iterate such a dependence relation is $\textit{transitivity}$. That is if a set $A$ depends on a set $B$ in the sense of the definition above, and $B$ depends on a set $C$ then one would like to say that $A$ depends on $C$. That is why the slightly stronger notion of \textit{strong dependence} is introduced:

\bigskip

\begin{definition}
\label{def of triangle strict}
Fix an open subset $\Omega$ of  $\R^{n+1}=\R_t \times \R^n_x$ and two \textit{finite} collections $(V_j)_{j\in J}$ and $(U_i)_{i \in I}$ of bounded open sets in $\R^{n+1}$. We say that $(V_j)_{j\in J}$ is \textit{under the strong dependence} of $(U_i)_{i \in I}$, denoted by

$$
(V_j)_{j\in J}  \vartriangleleft  (U_i)_{i \in I},
$$
if there exist $\tilde{U}_i \Subset U_i $ such that $(V_j)_{j\in J}  \trianglelefteq  (\tilde{U}_i)_{i \in I}$.

\end{definition}

\bigskip

To facilitate the lecture we re-write here the basic properties of the relation $\vartriangleleft$ summarised in Proposition 4.5 of \cite{Laurent_2018}. We shall use these properties for the proof of Theorem~\ref{semi global } in order to iterate in an abstract way the local estimate of Theorem~\ref{local quant estimate}.

\begin{prop}[Proposition 4.5 in \cite{Laurent_2018}]
\label{prop 4.5 in ll} One has:
\begin{enumerate}
    \item $(V_j)_{j\in J}  \vartriangleleft  (U_i)_{i \in I}$ implies $(V_j)_{j\in J}  \trianglelefteq  (U_i)_{i \in I}.$
    
    \item If $(V_j)_{j\in J}  \vartriangleleft  (U_i)_{i \in I}$ with $U_i=U$ for all $i \in I$, then $(V_j)_{j\in J}  \vartriangleleft  U$.
    
    \item If $V_i \Subset U_i $ for any $i \in I$, then $(V_i)_{i\in I}  \vartriangleleft  (U_i)_{i \in I}$.
    
    \item If $V_i \Subset U_i $ for any $i \in I$, then $\cup_{i \in I} V_i  \vartriangleleft  (U_i)_{i \in I}$.
    
    \item If $V_i  \vartriangleleft  U_i$ for any $i \in I$, then $(V_i)_{i\in I}  \vartriangleleft  (U_i)_{i \in I}$. In particular, if $U_i \vartriangleleft U$ for any $i \in I$, then $(U_i)_{i \in I} \vartriangleleft U$.
    
    \item The relation is transitive, that is:
    
    \begin{equation*}
    [(V_j)_{j\in J}  \vartriangleleft  (U_i)_{i \in I} \: \textnormal{and} \: (U_i)_{i\in I}  \vartriangleleft  (W_k)_{k \in K}    ] \Rightarrow (V_j)_{j\in J}  \vartriangleleft  (W_k)_{k \in K}.
    \end{equation*}
\end{enumerate}
\end{prop}

\begin{remark}
    Notice that the relation $\vartriangleleft$ depends on the operator $P$, therefore technically speaking Proposition~\ref{prop 4.5 in ll} is not exactly the same as Proposition 4.5 in~\cite{Laurent_2018} since in our case the coefficients of $P$ present a jump discontinuity. However this does not have any impact in the proof. 
\end{remark}

We can now formulate the result of Theorem ~\ref{local quant estimate} in terms of relations of dependence. Indeed, one has:

\begin{cor}
	\label{cor of local quant estimate}
 In the geometric situation of Theorem~\ref{local quant estimate}	let $(t_0,x_0) \in \Sigma$ given locally by $\Sigma=\{\phi=0 \}$. Then there exists $R_0$ such that for any $R\in (0,R)$, there exist $r,\rho >0$ such that
	$$
	B((t_0,x_0),r) \vartriangleleft \{\phi > \rho\} \cap B((t_0,x_0),4R).
	$$
\end{cor}

\begin{remark}
	In fact the constants $R,\rho,r$ depend only on $x_0$ and not on $t_0$. This comes from the fact that the coefficients of $P$ and the interface $\Sigma$ are independent of $t$.
\end{remark}
\begin{proof}
	The proof is as in Corollary 4.6 in \cite{Laurent_2018}.
\end{proof}

With the local quantitative estimate of Theorem~\ref{local quant estimate} at our disposal we are now ready to propagate these estimates to obtain a global one. We will start by propagating the low frequency estimates only (such as the estimate provided by Theorem~\ref{local quant estimate}). This will be done by using some abstract iteration properties of the relation $\vartriangleleft$ as defined in~\ref{def of triangle strict}. We will then follow a path from a point $x_0 \in \mathcal{M}$ to another one $x_1 \in \mathcal{M}$. As long as the path stays either in $\Omega_{+}$ or in $\Omega_{-}$ the propagation is guaranteed by the result of \cite{Laurent_2018}. If the path crosses the interface $S$ then we use Theorem~\ref{local quant estimate} which allows to propagate uniqueness from one side of $S$ to the other losing only an $\epsilon$ of time. Then we continue the propagation using again \cite{Laurent_2018}.  We state the main theorem:

\begin{thm}
\label{semi global }
Let $(\mathcal{M},g)$ be a smooth compact connected $n$-dimensional Riemannian manifold with (or without) boundary and $S$ a $(n-1)$-dimensional smooth submanifold of $\mathcal{M}$. We suppose that $\mathcal{M}\backslash S= \Omega_{-}\cup \Omega_{+}$ with $\Omega_{-}\cap \Omega_{+}= \emptyset$. Consider $P$ as defined in \eqref{definition of P }. For any nonempty open subset $\omega$ of $\mathcal{M}\backslash S$ and any $T>\mathcal{L}(\mathcal{M},\omega)$, there exist $\eta, C, \kappa, \mu_0$ such that for any $u\in H^1((-T,T)\times \mathcal{M})$ and $f \in L^2((-T,T)\times \mathcal{M})$ solving
\begin{equation}
\label{system}
    \begin{cases}
    Pu=f & \textnormal{in} \:(-T,T) \times \Omega_{-}\cup\Omega_{+}\\
    u_{|S_-}=u_{|S_+} & \textnormal{in}\: (-T,T)\times S \\
    (c\partial_\nu u)_{|S_-}=(c \partial_\nu u)_{|S_+} & \textnormal{in}\: (-T,T)\times S \\
    u=0 & \textnormal{in}\: (-T,T)\times \partial\mathcal{M},
    \end{cases}
\end{equation}
one has, for any $\mu\geq \mu_0$,

$$
\norm{u}{L^2((-\eta,\eta)\times \mathcal{M})}{}\leq Ce^{\kappa \mu}\left( \norm{u}{L^2((-T,T)\times\omega)}{}+\norm{f}{L^2(-T,T)\times \mathcal{M}}{} \right)+\frac{C}{\mu}\norm{u}{H^1((-T,T)\times\mathcal{M})}{}.
$$
\end{thm}

The following lemma will be used for the proof of Theorem~\ref{semi global } in order to transport the information locally from one side of the interface to the other. 

\begin{lem}
\label{local passage}
Consider a point $x \in S=\{\phi=0\}$. Let $R_0, R, r, \rho$ be the associated constants given by Corollary~\ref{cor of local quant estimate}. Then for any $T>0$ , for any $\epsilon>0$ and any subset $U \subset \R^{n+1}$ such that
$$
[-T,T]\times\left(  \{\phi > \frac{\rho}{2}\}  \cap B(x,r),4R)\right) \Subset U. 
$$
there exists $r_\epsilon>0$ with 
$$
[-T + 2 \epsilon, T-2 \epsilon]\times B(x,r_{\epsilon}) \vartriangleleft U.
$$
\end{lem}

\begin{proof}
Let us consider a finite covering of 
$$
[-T+\epsilon, T- \epsilon]\times B(x,r/2)\subset\bigcup_{i \in I}B((t_i,x),r), \quad I \: \textnormal{finite}.
$$
Now recall that from Corollary~\ref{cor of local quant estimate} and definitions of the above quantities we have as well:
\begin{equation}
\label{application of cor1}
B((t_i,x),r)\vartriangleleft \{\phi>\rho\}\cap B((t_i,x),4R), \quad i \in I.    
\end{equation}

We also have that
\begin{equation*}
    \bigcup_{i \in I}\bigg(\{\phi>\rho\}\cap B((t_i,x),4R)\bigg) \Subset [-T,T]\times\left(  \{\phi > \frac{\rho}{2}\}  \cap B(x,r),4R)\right)
\end{equation*}
and consequently thanks to the assumption made on $U$ we obtain that
$$
\{\phi>\rho\}\cap B((t_i,x),4R)\Subset U, \quad \textnormal{for} \: i \in I.
$$
We apply then item 3 of Proposition~\ref{prop 4.5 in ll} which gives
$$
\bigg(\{\phi>\rho\}\cap B((t_i,x),4R)\bigg)_{i\in I} \vartriangleleft U.
$$
By transitivity of $\vartriangleleft$ by \eqref{application of cor1} we get
\begin{equation}
\label{u contains 2}
\bigg(B((t_i,x),r)\bigg)_{i \in I}\vartriangleleft U. 
\end{equation}
We now use the following geometric fact, which is a consequence of the triangle inequality: For a given $\epsilon>0$ there exists $r_{\epsilon}$ such that we have
\begin{equation*}
[-T + 2 \epsilon, T-2 \epsilon]\times B(x,r_{\epsilon}) \Subset \bigcup_{i \in I} B\left((t_i,x),\frac{r}{2}\right).
\end{equation*}
Consequently, item 3 of Proposition~\ref{prop 4.5 in ll} gives:
\begin{equation}
\label{final inclusion property}
[-T + 2 \epsilon, T-2 \epsilon]\times B(x,r_{\epsilon}) \vartriangleleft \bigcup_{i \in I} B\left((t_i,x),\frac{r}{2}\right).
\end{equation}
The compact inclusion $B\left((t_i,x),\frac{r}{2}\right) \Subset B((t_i,x),r)$ implies according to item 4 of Proposition~\ref{prop 4.5 in ll}:
\begin{equation}
\label{application compact incl}
\bigcup_{i \in I} B\left((t_i,x),\frac{r}{2}\right) \vartriangleleft \bigg(B((t_i,x),r)\bigg)_{i \in I}.
\end{equation}

Finally, we combine \eqref{u contains 2}, \eqref{final inclusion property}, \eqref{application compact incl} using as well the transitivity (property 6 in Proposition~\ref{prop 4.5 in ll})) to find 
$$
[-T + 2 \epsilon, T-2 \epsilon]\times B(x,r_{\epsilon}) \vartriangleleft U,
$$
which finishes the proof of the Lemma.
\end{proof}

We are now ready to give the proof of Theorem~\ref{semi global }.

\begin{proof}[Proof of Theorem~\ref{semi global }]
The proof is divided in two steps:

\bigskip

\noindent \textbf{Step 1: Abstract propagation of low frequency information}

\bigskip

For technical reasons related to the second step of the proof we consider an open set $\tilde{\omega}$ with $\tilde{\omega}\Subset \omega.$

We want to propagate uniqueness from $\tilde{\omega}$ to a neighborhood of an arbitrary point of $\mathcal{M}$, say $y_0 \in \mathcal{M}$. From the definition of $\mathcal{L}(\mathcal{M}, \tilde{\omega})$ we can find a point $x_0 \in \tilde{\omega}$ and an admissible path $\gamma: [0,1] \rightarrow \mathcal{M}$ of length $l$ such that $T>\mathcal{L}(\mathcal{M}, \tilde{\omega})>l$, $\gamma(0)=x_0$ and $\gamma(1)=y_0$. 

By assumption (see Section~\ref{not and def}), the path $\gamma$ intersects the interface $S$ a finite number of times $N$. We call $x_{S,j},\: j\in \{1,2,...,N\}$ the intersection points. Moreover, the conditions made on the family of admissible paths imply that we can use the same coordinates as in \cite[proof of Theorem 6.3]{Laurent_2018} ,~\cite[pp 21-22]{Leb:Analytic}) in a neighborhood of this path. In particular, in these coordinates the path is straighten out, that is $\gamma(s)=(0,sl)$. We now apply Corollary~\ref{cor of local quant estimate} to $x_{S,1}\in S $ which gives us some constants $r_{S,1}, \rho_{S,1}$. We then choose a point $\tilde{x}_0$ with
$$
\tilde{x}_0 \in \{0<\phi_{S,1} < \rho/4 \} \cap B(x_{S,1},r_{S,1})\cap \gamma\left((\gamma^{-1}(x_0),\gamma^{-1}(x_{S,1}-))\right).
$$
\begin{figure}
    \centering
    \includegraphics{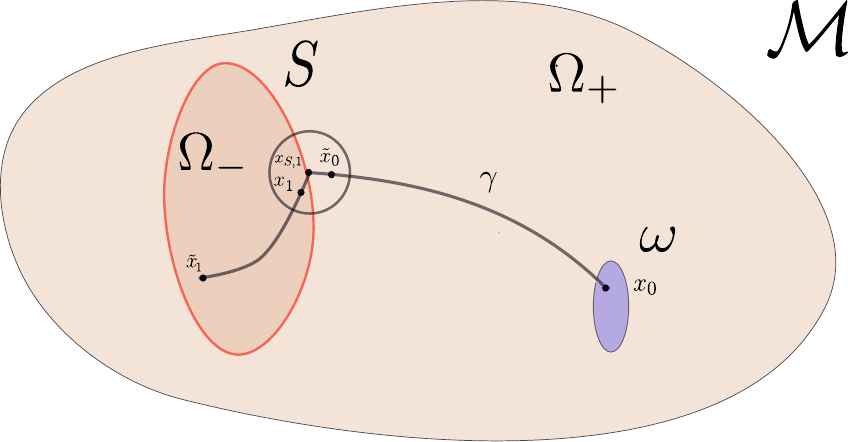}
    \caption{First step of the iteration process for the proof of Theorem~\ref{semi global }. Information is propagated from $x_0$ to $x_1$. The process is then repeated starting from $x_1$.}
    \label{two paths}
\end{figure}
Let us now look at the path joining $x_0$ to $\tilde{x}_0$. This path is entirely either in $\Omega_{+}$ or $\Omega_{-}$ and in particular one can apply the results of \cite{Laurent_2018}. More precisely one can construct an appropriate foliation such that one can apply Theorem 4.7 of \cite{Laurent_2018}. The construction is exactly as in the proof of Theorem 6.3 in \cite{Laurent_2018}. The difference is that here we only use the foliation to apply Theorem 4.7 which solely concerns \textit{low frequency} information. That is because in our case, this is only the first step of the propagation. Indeed, since we will need to continue the iteration we can not consider high frequencies yet. 

Applying Theorem 4.7 of \cite{Laurent_2018} gives us a set $U \subset \R_t \times \Omega_{+}$ such that 
$$
U \vartriangleleft  [-T,T]\times  \tilde{\omega},
$$
and the precise construction of \cite[Proof of theorem 6.3]{Laurent_2018} implies that $U$ contains a set of the form 
\begin{equation}
\label{u contains}
[-T_0,T_0]\times\left(  \{\phi_{S,1} > \frac{\rho}{2}\}  \cap B(x_{S,1},r_{S,1}),4R)\right) \Subset U.    
\end{equation}
The time $T_0$ is equal to $T_0=T-\tilde{T}_0$ where $\tilde{T}_0$ is the time given by \cite{Laurent_2018} to propagate the information from $x_0$  to $\tilde{x}_0$. This is any time greater than the length of the path joining those two points.

We can now apply Lemma~\ref{local passage} which ensures that for arbitrary $\epsilon>0$ there exists $r_\epsilon$ with
$$
[-T_0 + 2 \epsilon, T_0-2 \epsilon]\times B(x_{S_1},r_{\epsilon}) \vartriangleleft U.
$$
Since by construction of $U$ we have $U \vartriangleleft [-T,T]\times \tilde{\omega}$, using once again the transitivity of $\vartriangleleft$ we have finally shown that
$$
[-T_0 + 2 \epsilon, T_0-2 \epsilon]\times B(x_{S,1},r_{\epsilon}) \vartriangleleft[-T,T]\times \tilde{\omega}.
$$
Recalling that $x_{S,1} \in S$ the above property says that we managed to pass on (a possibly very small) neighborhood of the other side of the interface, and this by losing an arbitrarily small time. In particular we have shown that we can find a $x_1 \in \gamma\left(\gamma^{-1}(x_{S,1}),1]\right)$ such that 
\begin{equation}
\label{pass the interface}
    [-T_0 + 2 \epsilon, T_0-2 \epsilon]\times V(x_1) \vartriangleleft [-T,T]\times   \tilde{\omega},
\end{equation}
where $V(x_1)$ is a (small) neighborhood of $x_1$. We can now repeat this propagation procedure starting from the point $x_1$. We consider a point $\tilde{x}_1$ which is on the same side of the interface as $x_1$, that is $\tilde{x}_1 \in  \gamma\left((\gamma^{-1}(x_{S,1}),\gamma^{-1}(x_{S,2}))\right)$ and sufficiently close to the interface and we propagate the information from $x_1$ to $\tilde{x}_1$ using Theorem 4.7 of~\cite{Laurent_2018}. Then we pass on the other side of the interface using Lemma~\ref{local passage}. Iterating this process $N$ times and using the transitivity of $\vartriangleleft$ gives 
$$
[-T_N + 2 N\epsilon, T_N-2 N\epsilon]\times V(y_0) \vartriangleleft  [-T,T]\times \tilde{\omega},
$$
with $T_N=T-\sum_{0\leq j\leq N}\tilde{T}_j$ and $\tilde{T}_j$ any number strictly larger than the length of the path joining $x_j$ to $\tilde{x}_j$. Notice that the definition of $ \mathcal{L}(\mathcal{M}, \tilde{\omega})$ (see \eqref{greatest distance}) implies thanks to our assumption $T>\mathcal{L}(\mathcal{M}, \tilde{\omega})$ that $T_N>0$ and since $\epsilon$ can be chosen arbitrarily small we can have
$$
\eta_{y_0}:=T_N-2 N \epsilon>0.
$$

\begin{figure}
	\centering
	
	\begin{tikzpicture}
\draw (-8,0) -- (6,0) node[right]{$x$} ;
\filldraw[black] (5,0) circle (1pt) node[anchor=south west]{$x_0$};
\draw (5,-4) -- (5,4) ;
\draw (1.4,-3) -- (1.4,3) ;
\filldraw[black] (1.4,0) circle (1pt) node[anchor=south west]{$\tilde{x}_0$};

\node at (5.2, 4) {$T$};

\node at (5.2, -4.1) {$-T$};

\node at (1.7, 3) {$T_0$};
\node at (1.7, -3.1) {$-T_0$};

\draw (0.8,-2.9) -- (0.8,2.9) ;

\node at (0.3, 3.1) {$T_0-2\epsilon$};
\node at (0.3, -3.1) {$-T_0+2\epsilon$};

\filldraw[black] (1.1,0) circle (1pt) node[anchor=south]{$x_{S,1}$};

\filldraw[black] (0.8,0) circle (1pt) node[anchor=south east]{$x_{1}$};

\draw (-2,-1.5) -- (-2,1.5) ;

\node at (-2,1.6) {$T_1$};

\node at (-2,-1.7) {$-T_1$};

\filldraw[black] (-2,0) circle (1pt) node[anchor=south west]{$\tilde{x}_1$};

\draw (-5,-0.8) -- (-5,0.8) ;
\draw (-5.6,-0.7) -- (-5.6,0.7) ;
\filldraw[black] (-5.3,0) circle (1pt) node[anchor=south]{$x_{S,j}$};

\filldraw[black] (-5,0) circle (1pt) node[anchor=south west]{$\tilde{x}_{j-1}$};

\filldraw[black] (-5.6,0) circle (1pt) node[anchor=south east]{$x_j$};

\draw (-7.5,-0.3) -- (-7.5,0.3) ;
\filldraw[black] (-7.5,0) circle (1pt) node[anchor=south east]{$y_0$};

\node at (-7.5,0.5) {$\eta_{y_0}$};

\node at (-7.5,-0.5) {$-\eta_{y_0}$};

\draw [very thick, blue] (4.7,0) -- (5.3,0)node[anchor=north]{$\omega$};

\draw [thick, red] (5,0) -- (-7.5,0)node[anchor=north east]{$\gamma$};

 \draw [-stealth, thick, green](5.8,3) [out=150,in=30] to (2,3);

 \draw [-stealth, thick, orange](1.5,0.6) [out=150,in=30] to (0.6,0.6);

 \draw [-stealth, thick, green](0.6,1) [out=150,in=30] to (-2,1);

 \draw [-stealth, thick, orange](-4.9,0.45) [out=150,in=30] to (-5.75,0.45);

\end{tikzpicture}
\caption{The iteration process of the proof of Theorem~\ref{semi global } in space-time. We transport information from a point $x_0 \in \omega$ to $y_0\in \mathcal{M}$ by following the path $\gamma$. The points $x_{S,j}$ are the intersection points of $\gamma$ with $S$. The green arrows correspond to propagation of information in the smooth context where we use the results of~\cite{Laurent_2018}. The orange arrows propagate the information through the interface using Theorem~\ref{local quant estimate} and this, by losing an arbitrarily small time. }
\label{Propagation of information with time figure}
\end{figure}

To sum up, for an arbitrary $y_0 \in \mathcal{M}$ we were able to find an $\eta_{y_0}>0$ and a neighborhood $V(y_0)$ of $y_0$ such that 
\begin{equation}
    \label{property with eta}
[-\eta_{y_0},\eta_{y_0}]\times V(y_0) \vartriangleleft [-T,T]\times  \tilde{\omega}.  
\end{equation}

We have thus completed the iteration procedure. 
\bigskip

\noindent \textbf{Step 2: Unfolding the definition of $\vartriangleleft$}

\bigskip

The last part of the proof will be to write down the definition of $\vartriangleleft$ and obtain the desired estimate using a compactness argument for $\mathcal{M}$.

Define $U_{y_0}$ as 
$$
U_{y_0}:=[-\eta_{y_0},\eta_{y_0}]\times V(y_0).
$$
Consider a sufficiently small neighborhood $\tilde{V}(y_0)$ of $y_0$ such that $K \Subset U_{y_0}$, where we define $K$ as $$K:= \left[-\frac{\eta_{y_0}}{2}, \frac{\eta_{y_0}}{2}\right] \times \tilde{V}(y_0).$$ Pick $\chi \in C^\infty_0(U_{y_0})$ such that $\chi=1$ on a neighborhood $U_\chi$ of $K$ and $\phi \in C^\infty_0 ((-T,T)\times \omega)$ with $\phi=1$ on a neighborhood of $\tilde{\omega}$. The definition of $\vartriangleleft$ and  \eqref{property with eta} imply that for any $\kappa>0$, there exist $C, \beta, \kappa^\prime, \mu_0>0$ such that for $\mu \geq \mu_0$~\footnote{Here and all along the proof we write directly, with a slight abuse of notation, the estimates in an invariant way. In fact one writes down the estimates given by the definition of $\vartriangleleft$ in some appropriate coordinates and then passes into the global ones. This does not pose any problem since the estimates we consider are invariant by changes of coordinates.},
\begin{equation}
    \label{unfolding triangle}
 Ce^{\kappa \mu}\left(\norm{M^\mu_\mu \phi_\mu u}{H^1}{}+\norm{Pu}{L^2(-T,T)\times \mathcal{M}}{}\right)+Ce^{- \kappa^\prime \mu} \norm{u}{H^1}{}
\geq \norm{M_\mu^{\beta \mu}\chi_\mu u}{H^1}{}.
\end{equation}

We want to use the almost localization of $\phi_\mu$ in order to control the left hand side of the above quantity.  To do so we use Lemma~\ref{lemma from ll hypo} which gives us
$$
\norm{M^\mu_\mu \phi_\mu u}{H^1}{} \leq C\mu \norm{u}{L^2((-T,T)\times \omega)}{}+Ce^{-c \mu}\norm{u}{H^1}{}.
$$
We want to inject the estimate above in \eqref{unfolding triangle}. We need to be sure that the term $\norm{u}{H^1}{}$ will be multiplied by a negative exponential. We recall that $\kappa>0$ can be  chosen arbitrarily small and we impose $\kappa < c/2$. Since $\mu\leq Ce^{\kappa \mu}$ we obtain with a new constant $c^\prime>0$:
\begin{equation}
\label{estimate with localization of phi}
    \norm{M_\mu^{\beta \mu}\chi_\mu u}{H^1}{}\leq Ce^{2 \kappa \mu}\left(\norm{u}{L^2((-T,T)\times \omega)}{}+\norm{Pu}{L^2(-T,T)\times \mathcal{M}}{}\right)+Ce^{-c^\prime \mu}\norm{u}{H^1}{}.
\end{equation}
We need to exploit the almost localization of $\chi_\mu$ to control from below the left hand side of \eqref{estimate with localization of phi}. We take $\tchi \in C^\infty_0(U_\chi)$ with $\tchi=1$ in a neighborhood of $K$. We use Lemma~\ref{lemma 2.3 from ll} to control:
\begin{align}
\label{estim for final estim}
    \norm{\tchi u }{L^2}{}&\leq \norm{\tchi\chi_\mu u }{L^2}{}+\norm{\tchi(1-\chi_\mu) u }{L^2}{}\leq C \norm{\chi_\mu u }{L^2}{}+Ce^{-c \mu} \norm{u}{H^1}{}\nonumber\\ 
    &\lesssim \norm{M^{\beta \mu}_\mu \chi_\mu u}{L^2}{}+\norm{(1-M^{\beta \mu}_\mu) \chi_\mu u}{L^2}{}+e^{-c \mu}\norm{u}{H^1}{}.
\end{align}
We control the second term as follows:
\begin{align*}
\norm{(1-M^{\beta \mu}_\mu) \chi_\mu u}{L^2}{}&=\norm{\left(1-m_\mu\left(\frac{\xi_t}{\beta \mu}\right)\right) \widehat{\chi_\mu u}}{L^2}{}=\norm{\frac{\left(1-m_\mu\left(\frac{\xi_t}{\beta \mu}\right)\right)}{|\xi_t|+|\xi_x|} \left(|\xi_t|+|\xi_x|\right)\widehat{\chi_\mu u}}{L^2}{}\\
&\lesssim \norm{\frac{\left(1-m_\mu\left(\frac{\xi_t}{\beta \mu}\right)\right)}{|\xi_t|+|\xi_x|}}{L^\infty}{} \norm{u}{H^1}{}.
\end{align*}
In the region $|\xi_t| \geq \frac{\beta \mu}{2}$ we simply have
$$
\frac{\left(1-m_\mu\left(\frac{\xi_t}{\beta \mu}\right)\right)}{|\xi_t|+|\xi_x|} \leq \frac{C}{\mu}.
$$
In the region $|\xi_t |\leq \frac{\beta \mu}{2}$ we use the support of $m$ and Lemma~\ref{lemma 2.3 from ll} to find
$$
\frac{\left(1-m_\mu\left(\frac{\xi_t}{\beta \mu}\right)\right)}{|\xi_t|+|\xi_x|} \leq C e^{- c \mu},
$$
hence in particular
$$
\norm{\frac{\left(1-m_\mu\left(\frac{\xi_t}{\beta \mu}\right)\right)}{|\xi_t|+|\xi_x|}}{L^\infty}{} \leq \frac{C}{\mu}.
$$
Combining estimates \eqref{estim for final estim} and \eqref{estimate with localization of phi} yields:

$$
\norm{\tchi u}{L^2}{}\leq Ce^{2 \kappa \mu}\left(\norm{u}{L^2((-T,T)\times \omega)}{}+\norm{Pu}{L^2(-T,T)\times \mathcal{M}}{}\right)+\frac{C}{\mu}\norm{u}{H^1}{}.
$$
Since $\tchi=1$ on $K$ we have:
$$
\norm{u}{L^2(K)}{}=\norm{\tchi u }{L^2(K)}{}\leq \norm{\tchi u }{L^2}{}.
$$
This yields, thanks to the previous estimate and the explicit definition of $K$ the final estimate:
\begin{equation*}
    \norm{u}{L^2\left([-\frac{\eta_{y_0}}{2}, \frac{\eta_{y_0}}{2}] \times \tilde{V}(y_0)\right)}{}\leq Ce^{2 \kappa \mu}\left(\norm{u}{L^2((-T,T)\times \omega)}{}+\norm{Pu}{L^2(-T,T)\times \mathcal{M}}{}\right)+\frac{C}{\mu}\norm{u}{H^1}{},
\end{equation*}
where $\tilde{V}_{y_0}$ is a (small) neighborhood of $y_0$, $y_0$ is an arbitrary point of $\mathcal{M}$ and $\eta_{y_0}>0$ is an associated strictly positive time. One can cover the manifold $\mathcal{M}$ by such neighborhoods and by compactness we can extract a finite covering 
$$
\mathcal{M} \subset \bigcup_{j \in J \: \textnormal{finite}} \tilde{V}(y_j),
$$
such that 
\begin{equation*}
    \norm{u}{L^2\left([-\frac{\eta_{y_j}}{2}, \frac{\eta_{y_j}}{2}] \times \tilde{V}(y_j)\right)}{}\leq C_je^{2 \kappa_j \mu}\left(\norm{u}{L^2((-T,T)\times \omega)}{}+\norm{Pu}{L^2(-T,T)\times \mathcal{M}}{}\right)+\frac{C_j}{\mu}\norm{u}{H^1}{},
\end{equation*}
for $\mu \geq \mu_0$. Let $\eta:= \min \eta_{y_j}/2$, $C:= \max C_j$, $c:= \min c_j^\prime$. We have:
\begin{align*}
 \norm{u}{L^2\left([-\eta, \eta] \times \mathcal{M}\right)}{}&\leq \sum_{j \in J} \norm{u}{L^2\left([-\frac{\eta_{y_j}}{2}, \frac{\eta_{y_j}}{2}] \times \tilde{V}(y_j)\right)}{}\\
 &\leq |J|\bigg( Ce^{2 \kappa \mu}\left(\norm{u}{L^2((-T,T)\times \omega)}{}+\norm{Pu}{L^2(-T,T)\times \mathcal{M}}{}\right)+\frac{C}{\mu}\norm{u}{H^1}{}\bigg).
\end{align*}
 
The proof of Theorem~\ref{semi global } is now complete.
\end{proof}

In the preceding theorem we chose to present a proof in the case where our observation domain $\omega$ is an open subset of $\mathcal{M}$. The point is that the difficulty of our problem comes from the interface inside $\mathcal{M}$ where the metric may jump. The Theorem and its proof show how we can propagate the information when crossing the interface in a way that is compatible with the quantitative results of \cite{Laurent_2018}. Since in \cite{Laurent_2018} the boundary case has been treated too, we can as well formulate the analogous result in the case of boundary observation. More precisely one has:

\begin{thm}
\label{boundary observation}
Under the assumptions of Theorem~\ref{semi global } assume additionally that $\partial \mathcal{M}$ is non empty and consider $\Gamma$ a non empty open subset of $\partial \Gamma$. Then for any $T> \mathcal{L}(\mathcal{M}, \Gamma)$, there exist $\eta, C, \kappa, \mu_0>0$ such that for any $u \in H^1((-T,T)\times \mathcal{M})$ and $f \in L^2((-T,T)\times \mathcal{M})$ solving~\ref{system}, we have for any $\mu \geq\mu_0$
$$
\norm{u}{L^2((-\eta,\eta)\times \mathcal{M})}{}\leq Ce^{\kappa \mu}\left( \norm{\partial_{\nu_{\Gamma}}u}{L^2((-T,T)\times \Gamma)}{}+\norm{f}{L^2(-T,T)\times \mathcal{M}}{} \right)+\frac{C}{\mu}\norm{u}{H^1((-T,T)\times\mathcal{M})}{}.
$$
\end{thm}

\bigskip

One can now combine the two preceding theorems with classical energy estimates for solutions of the wave equation that allow to relate $\norm{u}{H^1}{}$ and $\norm{u}{L^2}{}$ with the energy of its initial data $(u,\partial_t u)_{|t=0}$ to obtain the following slightly more general version of Theorem~\ref{theorem quant}: 

\begin{thm}
\label{initial data}
Let $(\mathcal{M},g)$ be a smooth compact connected $n$-dimensional Riemannian manifold with (or without) boundary and $S$ an $(n-1)$-dimensional smooth submanifold of $\mathcal{M}$. We write $\mathcal{M}\backslash S= \Omega_{-}\cup \Omega_{+}$. Consider 
$P$ as defined in \eqref{definition of P }. For any nonempty open subset $\omega$ of $\mathcal{M}\backslash S$ and any $T>\mathcal{L}(\mathcal{M},\omega)$, there exist $C, \kappa, \mu_0$ such that for any $(u_0,u_1) \in H^1_0( \mathcal{M}) \times L^2(\mathcal{M})$, $f \in L^2((-T,T)\times \mathcal{M})$ and $u$ solving
\begin{equation}
\label{system with data}
    \begin{cases}
    Pu=f & \textnormal{in} \:(-T,T) \times \Omega_{-}\cup\Omega_{+}\\
    u_{|S_-}=u_{|S_+} & \textnormal{in}\: (-T,T)\times S \\
    (c\partial_\nu u)_{|S_-}=(c \partial_\nu u)_{|S_+} & \textnormal{in}\: (-T,T)\times S \\
    u=0 & \textnormal{in}\: (-T,T)\times \partial\mathcal{M} \\
    \left(u,\partial_t u \right)_{|t=0}=(u_0,u_1) & \textnormal{in} \: \mathcal{M},
    \end{cases}
\end{equation}
one has, for any $\mu\geq \mu_0$,

$$
\norm{(u_0,u_1)}{L^2\times H^{-1}}{}\leq Ce^{\kappa \mu}\left( \norm{u}{L^2((-T,T)\times \omega)}{}+\norm{f}{L^2(-T,T)\times \mathcal{M}}{} \right)+\frac{C}{\mu}\norm{(u_0,u_1)}{H^1\times L^2}{}.
$$
If moreover $\mathcal{M}\neq \emptyset$ and $\Gamma$ is a non empty open subset of $\partial \mathcal{M}$, for any $T> \mathcal{L}(\mathcal{M},\Gamma)$, there exist $C, \kappa, \mu_0>0$ such that for any $(u_0,u_1) \in H^1_0( \mathcal{M}) \times L^2(\mathcal{M})$, $f \in L^2((-T,T)\times \mathcal{M})$ and $u$ solving \eqref{system with data}, we have
$$
\norm{(u_0,u_1)}{L^2\times H^{-1}}{}\leq Ce^{\kappa \mu}\left( \norm{\partial_{\nu_{\Gamma}}u}{L^2((-T,T) \times \Gamma)}{}+\norm{f}{L^2(-T,T)\times \mathcal{M}}{} \right)+\frac{C}{\mu}\norm{(u_0,u_1)}{H^1\times L^2}{}.
$$
\end{thm}

\appendix

\section{A few facts on pseudodifferential calculus}

\label{a few facts on pseudo}

We collect here some facts and notations concerning the symbolic calculus which is an essential ingredient for the proof of the Carleman estimate. We follow here the exposition of~\cite{le2013carleman} and~\cite[Chapter~2]{rousseau2022elliptic}.

\subsection{Differential operators}
\label{differential operators}
For $(t,x) \in \R_t \times \R^n_x$ we write $(t,x)=(t,x^\prime,x_n)$. The variable $x_n$ is normal to the interface.  We use as well the notation $D_{x_j}:=\frac{1}{i} \partial_{x_j} $.

We denote by $\mathcal{D}^m_\tau$ the set of differential operators depending on $\tau$, that is operators of the form
$$
P(t,x,D_t,D_x,\tau)=\sum_{j+|\alpha| \leq m}a_{j,\alpha}(t,x)\tau^j D_{t,x}^\alpha.
$$
Their principal symbols are defined as
$$
\sigma(P)=\sum_{j+|\alpha| = m}a_{j,\alpha}(t,x)\tau^j (\xi_t \xi)^\alpha.
$$

The set of tangential operators depending on the large parameter $\tau$ is denoted by $\mathcal{D}^m_{\top, \tau}$ and contains operators of the form
$$
P(t,x,D_t,D_{x^\prime},\tau)=\sum_{j+|\alpha| \leq m}a_{j,\alpha}(t,x)\tau^j (D_t D_{x^\prime})^\alpha
$$
with principal symbols defined as
$$
\sigma(P)=\sum_{j+|\alpha| = m}a_{j,\alpha}(t,x)\tau^j (\xi_t \xi^\prime)^\alpha.
$$

\subsection{Standard tangential classes}
\label{standard pseudo }

For $m \in \R$ we define the class \textit{tangential} symbols $\mathcal{S}^m$ as the smooth functions on $\R^{n+1}\times \R^n$ such that for all $(\alpha, \beta) \in \mathbb{N}^{n+1} \times \mathbb{N}^n$,
$$
\sup_{( t,x,\xi_t, \xip)} (1+|\xip|^2+ |\xi_t|^2)^{\frac{-m+ |\beta|}{2}}|(\partial^{\alpha}_{t,x} \partial^{\beta}_{\xi_t,\xip})a( t,x,\xi_t, \xip)| < \infty.
$$
We shall mainly work with the \textit{Weyl quantization} which associates to $a \in \mathcal{S}^m$ an operator denoted by $\op(a)$ defined by
\begin{align*}
(\op(a)u)(t,x^\prime,x_n)&=\\
&\hspace{-4mm}(2\pi)^{-n}\iint_ {\R^{2n}}e^{i\left((t,x^\prime)-(s,y^\prime)\right) \cdot (\xi_t,\xip)}a\left(\frac{t+s}{2},\frac{x^\prime+y^\prime}{2},x_n,\xip\right)u(y^\prime,s,x_n)dsdy^\prime d\xip d\xi_t.    
\end{align*}

These integrals may not be defined in the classical sense (Lebesgue integration). They are however well defined as \textit{oscillatory integrals}. 

We denote by $\Psi^{m}$ the set of these pseudodifferential operators. We define as well
$$
\mathcal{S}^{-\infty}:=\bigcap_{m\in \R} \mathcal{S}^m, \quad \Psi^{-\infty}:=\bigcap_{m\in \R} \Psi^m.
$$
Notice that even though the operators above are tangential we do not use any special notation since all pseudodifferential operators we consider are tangential.

A basic feature of the Weyl quantization is that we have the exact equality:
$$
(\op(a))^*=\op(\Bar{a}),
$$
where we denote by $^*$ the adjoint operator on $L^2$. In particular, operators associated to real valued symbols are (formally) self-adjoint. 

For $a \in \mathcal{S}^m$ we call \textit{principal symbol}, $\sigma(a)$, the equivalence class of $a$ in $\mathcal{S}^m/\mathcal{S}^{m-1}.$

Many times we refer to the \textit{Sobolev regularity} property of pseudodifferential calculus, namely:
$$
\op(a): L^2(\R_{x_n}; H^{s+m}(\R^n_{t,x^\prime}))\rightarrow L^2(\R_{x_n}; H^{s}(\R^n_{t,x^\prime})) \: \textnormal{continuously},
$$
where $a \in \mathcal{S}^m$.

Consider now $a_1 \in \mathcal{S}^{m_1}$ and $a_2 \in \mathcal{S}^{m_2}$. Then there exists a $c\in \mathcal{S}^{m_1+m_2}$ such that we have
$$
\op(a_1)\op(a_2)=\op(c),
$$
and we denote $c:=a_1\sharp a_2$ where $\sharp$ is called the \textit{Moyal product}. One has, for any $N \in\mathbb{N}$ the following asymptotic formula:
\begin{equation}
\label{asympt formula}
(a_1 \sharp a_2 )(t,x,\xi_t,\xi)-\sum_{j<N}\left(i \omega(D_{t,x^\prime},D_{\xi_t,\xip};D_{y^\prime},D_{\eta^\prime})\right)^j a_1(t,x,\xi_t,\xi)a_2(y,\eta)|_{y=(t,x), \eta=(\xi_t,\xi)}  \in \mathcal{S}^{m_1+m_2-N},
\end{equation}

with $\omega(a,b;c,d)=c\cdot b- a\cdot d$. This formula implies:

\begin{enumerate}
    \item $\op(a_1)\op(a_2)=\op(a_1a_2)+\op(r_1),\quad r_1\in \mathcal{S}^{m_1+m_2-1}. $
    
    \item $[\op(a_1),\op(a_2)]=\op\left( \frac{1}{i}\{a_1,a_2\}\right)+\op(r_3),\quad r_3 \in \mathcal{S}^{m_1+m_2-2}.$
\end{enumerate}

\subsection{Tangential classes with a large parameter}

Since we want to show a Carleman estimate which involves a large parameter, the natural class in our context is that of \textit{pseudodifferential operators with a large parameter.} For $\tau\geq 1$ we define 
$$
\lambda^2_\tau= \tau^2+|\xip|^2+|\xi_t|^2,
$$
The class denoted by $\mathcal{S}_{\tau}^m$ contains the functions $a \in C^{\infty}( t,x,\xi_t, \xip, \tau)$ satisfying for all $(\alpha, \beta) \in \mathbb{N}^{n+1} \times \mathbb{N}^n$:
$$
\sup_{\substack{( t,x,\xi_t, \xip) \\ \tau \geq 1}} \lambda_\tau^{-m+ |\beta|}(\partial^{\alpha}_{t,x} \partial^{\beta}_{\xi_t,\xip})a( t,x,\xi_t, \xip, \tau)| < \infty.
$$
We set $\Psi^m_\tau:=\{ \op(a), \: a\in \mathcal{S}^m_\tau\}$ and
$$
\mathcal{S}^{-\infty}_\tau:=\bigcap_{m\in \R} \mathcal{S}^m_\tau, \quad \Psi^{-\infty}_\tau:=\bigcap_{m\in \R} \Psi^m_\tau.
$$

We denote by $(\cdot, \cdot)$ the inner product on $\lrn$ defined by $(f,g)=\int_{\R^{n+1}} f \Bar{g}$ and by $(\cdot, \cdot)_{\pm}$ its restriction on $L^2(\R^{n+1}_{\pm})$.

We introduce the following Sobolev norms, defined in the tangential variables:
$$
\normsurf{u(x_n,\cdot)}{H^s}{}=\normsurf{\op(\lambda^s) u(x_n, \cdot)}{L^2(\R^{n})}{}, \quad \normsurf{u(x_n,\cdot)}{H^s_{\tau}}{}=\normsurf{\op(\lambda^s_{\tau}) u(x_n, \cdot)}{L^2(\R^{n})}{}.
$$
The above norms define the (usual) Sobolev space 
$H^s$ and the Sobolev space including a large parameter $H^s_\tau.$ We use many times that for $s=1$ one has the equivalence
$$
\norm{\cdot}{H^1_\tau}{2}\sim \tau^2 \norm{\cdot}{L^2}{2}+\norm{\nabla \cdot}{L^2}{2}.
$$

All the properties listed in Section \ref{standard pseudo } in the classical case remain valid in the context of the large parameter. In particular we have the \textit{Sobolev regularity} property, for $a\in \mathcal{S}_m^\tau$: 

$$
\op(a): L^2(\R_{x_n}; H^{s+m}_\tau(\R^n_{t,x^\prime}))\rightarrow L^2(\R_{x_n}; H^{s}_\tau(\R^n_{t,x^\prime})),
$$
continuously and uniformly in $\tau\geq 1$. This yields that for $a\in \mathcal{S}^m_\tau$ we have 
$$
\left| \left( \op(a)u,u\right) \right|\lesssim \norm{u}{L^2(\R;H^{m/2}_\tau)}{2}.
$$

Many times in the article we absorb error terms by taking  \textit{ $\tau$ sufficiently large}. By that we invoke the following property: If $m^\prime>m$ then 
$$
\norm{\cdot}{H^m_\tau}{}\leq C \tau^{-(m^\prime-m)}\norm{\cdot}{H^{m^\prime}_\tau}{},
$$
for $\tau\geq 1$.

The main tool to transfer some positivity properties of the symbol to some estimate for the corresponding operator is \textit{Gårding's inequality}. We shall use it in the context of operators involving a large parameter. For a proof we refer to \cite[Chapter 2]{rousseau2022elliptic}. There the standard quantization is used, the same proof works however for the Weyl quantization.

\begin{lem}[\textbf{Gårding's inequality with a large parameter}]
\label{Garding with a large}
Consider $a\in \mathcal{S}^m_\tau$ with principal symbol $a_m$. Suppose that there exist $C>0$ and $R>0$ such that 
$$
\operatorname{Re}a_m( t,x,\xi_t, \xip,\tau)\geq C \lambda^m_{\tau}, \quad x\in \R^{n+1}, \: (\xi_t,\xip)\in \R^n, \: \tau \geq 1, \: |(\xi_t,\xip,\tau)|\geq R,
$$
then there exist $C^\prime$ and $\tau_0$ such that
$$
\operatorname{Re}\left(\op(a)u,u\right)\geq C^\prime \norm{u}{L^2(\R; H^{m/2}_\tau)}{2},
$$
for $u \in \mathscr{S}(\R^{n+1})$ and $\tau \geq \tau_0$.
\end{lem}

Before stating the last lemma of this section let us recall a definition.

\begin{definition}
The \textbf{essential support} of a symbol $\mathcal{S}^m_\tau$, denoted by $\textnormal{essupp}(a)$ is the complement of the largest open set of $\R^{n+1}_{t,x}\times\R^n_{\xi_t,\xip}\times \{\tau \geq 1\}$ where the estimates for $\mathcal{S}^{-\infty}_\tau$ hold. More precisely, a point $(t_0,x_0,s_0, \xip_0) \in \R^{n+1}\times (\R^n \backslash \{0\})$ \textbf{does not} lie in the essential support of $a$ if there exists a neighborhood $\mathcal{U}$ of $(t_0, x_0)$ and a conic neighborhood $\mathcal{V}$ of $(s_0, \xip_0)$ such that for all $m \in \R$ and all $(\alpha, \beta) \in \mathbb{N}^{n+1} \times \mathbb{N}^n$ one has
$$
\sup_{\substack{( t,x) \in \mathcal{U}, (\xi_t, \xip) \in \mathcal{V} \\ \tau \geq 1}} \lambda_\tau^{-m+ |\beta|}(\partial^{\alpha}_{t,x} \partial^{\beta}_{\xi_t,\xip})a( t,x,\xi_t, \xip, \tau)| < \infty.
$$

\end{definition}

Although the natural classes for us to work with are $\mathcal{S}^m_\tau$ involving the large parameter $\tau$, for technical reasons we also have to deal with symbols in $\mathcal{S}^m$. Since $\mathcal{S}^m \not \subset \mathcal{S}^m_\tau$ one has to make sure that the chosen objects belong to the appropriate spaces. The following Lemma  \cite[Lemma A.4]{le2013carleman} will be then very useful:

\begin{lem}
\label{lemma with different classes}
Let $m, m^\prime \in \R$, $a_1( t,x,\xi_t, \xip) \in \mathcal{S}^{m}$ and $a_2( t,x,\xi_t, \xip,\tau)\in \mathcal{S}^{m^\prime}_\tau$ such that the essential support of $a_2$ is contained in a region where $|(\xi_t,\xip)|\gtrsim \tau$. Then
$$
\op(a_1)\op(a_2)=\op(b_1), \quad \op(a_2)\op(a_1)=\op(b_2),
$$
with $b_1,b_2 \in \mathcal{S}_\tau ^{m+m^\prime}$. Moreover one has the same asymptotic formula as~\eqref{asympt formula} with the remainder in $\mathcal{S}^{m+m^\prime-N}_\tau$. 
\end{lem}

\section{Some lemmata used in the quantitative estimates}
In this Section we collect some estimates coming essentially from $\cite{Laurent_2018}$ concerning the regularization and localization operators that we introduced for the quantitative estimates. 

\bigskip

\noindent We define for a function $f \in L^\infty(\R^{n+1})$ and $\lambda>0$:
$$
f_\lambda(t,x):=e^{-\frac{|D_t|^2}{\lambda}}f=\left(\frac{\lambda}{4 \pi}  \right)^{\frac{1}{2}}\int_{\R} f(s,x)e^{-\frac{\lambda}{4}|t-s|^2}ds.
$$
We use many times the fact that
$$
\norm{f_\lambda}{L^2}{}\leq \norm{e^{- \frac{|\cdot|^2}{\lambda}}}{L^\infty(\R_t)}{}\norm{\mathcal{F}_t(f)(\xi_t,x)}{L^2}{}=\norm{f}{L^2}{}.
$$
Notice also that we have 
$$
f \geq 0 \Longrightarrow f_{\lambda}\geq 0,
$$
and consequently
\begin{equation}
\label{property croissance}
f\geq g \Longrightarrow f_\lambda \geq g_\lambda.    
\end{equation}

\bigskip

We now recall several Lemmas from \cite{Laurent_2018} that we use in the main part of this article.

\begin{lem}[Lemma 2.3 in \cite{Laurent_2018}]
\label{lemma 2.3 from ll}
For any $d>0$, there exist $C,c >0$ such that for any $f_1, f_2 \in L^{\infty}(\R^{n+1})$ such that $\textnormal{dist}(\supp{f_1},\supp{f_2})\geq d$ and all $\lambda\geq 0$, we have
$$
\norm{f_{1,\lambda}f_2}{L^{\infty}}{}\leq C e^{-c \lambda}\norm{f_1}{L^\infty}{} \norm{f_2}{L^\infty}{}, \quad \norm{f_{1, \lambda}f_{2,\lambda}}{L^\infty}{} \leq C e^{-c \lambda}\norm{f_1}{L^\infty}{} \norm{f_2}{L^\infty}{}.
$$

\end{lem}

\begin{lem}[Lemma 2.4 in \cite{Laurent_2018}]
\label{lemma 2.4 from ll}
Let $f_2 \in C^{\infty}(\R^{n+1})$ with all derivatives bounded, and $d>0$. Then for every $k \in \mathbb{N}$, there exist $C,c >0$ such that for all $f_1 \in H^k(\R^{n+1})$ such that $\textnormal{dist}(\supp{f_1},\supp{f_2})\geq d$ and all $\lambda \geq 0$ we have 
$$
\norm{f_{1, \lambda}f_2}{H^k}{} \leq Ce^{- c\lambda}\norm{f_1}{H^k}{}.
$$

\end{lem}

\begin{lem}
\label{lemma 2.5 from ll}
Let $\psi: \R^{n+1} \rightarrow \R$ be a Lipschitz continuous function, $f_1 \in C^{\infty}(\R^{n+1})$ with bounded derivatives and $f_2 \in C^{\infty}_0(\R^{n+1})$ such that $\textnormal{dist}(\supp{f_1(\psi), \supp{f_2}})>0$. Then, for $k \in \{0,1\}$ there exist $C,c >0$ such that for all $\lambda>0$, we have
$$
\norm{f_{1,\lambda}(\psi)f_2}{H^k \rightarrow H^k}{}\leq Ce^{-c \lambda}.
$$
\end{lem}

Lemma \ref{lemma 2.5 from ll} is essentially Lemma 2.5 from \cite{Laurent_2018}. In its statement the Lemma requires for $\psi$ to be smooth. However, since we only need to control derivatives of order at most one Lipschitz regularity is sufficient.

\begin{lem}[Lemma 2.6 in \cite{Laurent_2018}]
\label{lemma 2.6 from ll}
Let $f_1, f_2 \in C^{\infty}_0(\R^{n+1})$ such that $f_1=1$ in a neighborhood of $\supp{f_2}$. Then there exist $C,c >0 $ such that for all $\lambda>0$, and all $u \in H^1(\R^{n+1})$, we have 
$$
\norm{f_{2,\lambda}\partial^\alpha u}{L^2}{} \leq C \norm{f_{1, \lambda}u}{H^1}{}+Ce^{-c \lambda}\norm{u}{H^1}{}, \quad \textnormal{for all} \: |\alpha|\leq 1.
$$

\end{lem}

\noindent We recall that the operators $M^\mu_\lambda$ have been defined in Section~\ref{some definitions for the local quant estimate}.

\begin{lem}[Lemma 2.10 in \cite{Laurent_2018}]
\label{lemma 2.10 from ll}
Let $f_1$ and $f_2$ be in $C^\infty$ bounded as well as their derivatives with $\textnormal{dist}(\supp{f_1},\supp{f_2})\geq d>0$. Then for any $k \in \mathbb{N}$, there exist $C,c >0$ such that for all $\mu>0$ and $\lambda>0$, we have
$$
\norm{f_{1,\lambda}M^\mu_\lambda f_{2,\lambda}}{H^k\rightarrow H^k}{}\leq C^{-c \frac{\mu^2}{\lambda}}+Ce^{-c \lambda}, \quad \norm{f_{1,\lambda}M^\mu_\lambda f_{2}}{H^k\rightarrow H^k}{}\leq C^{-c \frac{\mu^2}{\lambda}}+Ce^{-c \lambda}
$$
\end{lem}

\begin{lem}[Lemma 2.11 in \cite{Laurent_2018}]
\label{lemma 2.11 from ll}
Let $f \in C^\infty_0(\R^{n+1})$. Then there exist $C, c>0$ such that for all $\mu>0$, $\lambda>0$ and $u  \in H^1(\R^{n+1})$, one has
$$
\norm{f_\lambda M^\mu_\lambda u}{H^1}{} \leq \norm{f_\lambda M^{2\mu}_\lambda u }{H^1}{}+C\left(e^{-c \frac{\mu^2}{\lambda}}+e^{-c \lambda}\right)\norm{u}{H^1}{}.
$$
\end{lem}

\begin{lem}[Lemma 2.13 in \cite{Laurent_2018}]
\label{lemma 2.13 from ll}
There exists $C>0$ such that for all $D  \in \mathbb{R}$, $\tilde{\chi} \in L^\infty(\R) $ such that $\supp{\tilde{\chi}} \subset (-\infty, D],$ for all $\lambda, \tau >0$, we have
$$
\norm{e^{\tau \psi } \tilde{\chi}_\lambda(\psi)}{L^\infty}{}\leq C \norm{\tchi}{L^\infty}{}\langle \lambda \rangle^{1/2}e^{D \tau}e^{\frac{\tau^2}{\lambda}},
$$
for all $\psi \in C^0(\R^{n+1}; \R).$

\end{lem}

\begin{lem}[Lemma 2.14 in \cite{Laurent_2018}]
\label{lemma 2.14 from ll}
There exist $C, c$ such that, for any $\epsilon, \tau, \lambda, \mu>0$, for any $k \in \mathbb{N}$, we have:
$$
\norm{e^{- \frac{\epsilon |D_t|^2}{2\tau}}(1-M_\mu^\lambda)}{H^k\rightarrow H^k}{}\leq e^{-\frac{\epsilon \mu^2}{8 \tau}}+Ce^{-c \lambda}.
$$

\end{lem}

\begin{lem}
\label{lemma 2.16 from ll}
Let $\psi$ be a locally Lipschitz continuous real valued function on $\R_t \times \R^n_x$, which is a quadratic polynomial in $t$, let $R_\sigma>0$, and $\sigma \in C^\infty_0(B(0,R_{\sigma}))$. Let $\chi \in C^\infty_0(\R)$ with $\supp(\chi)\subset (-\infty,1),$ and $\tchi\in C^\infty_0(\R)$ such that $\tchi=1$ on a neighborhood of $(-\infty,3/2)$, $\supp{\tchi}\subset (-\infty,2)$, and set $\chi_\delta(s):=\chi(s/\delta),$ $\tchi_\delta(s):=\tchi(s/\delta)$. Let $f$ be bounded, compactly supported and real analytic in the variable $t$ in a neighborhood of $\overline{B}_{\R_t}(0,R_\sigma)$ and define 
$$
g:=e^{\tau \psi}\chi_{\delta, \lambda}(\psi)\tchi_\delta(\psi)f \sigma_\lambda.
$$
Then one has the following estimate: for all $c, \delta>0$ there exist $c_0, C, N>0$ such that for any $\tau, \mu \geq 1$ and $\frac{\mu}{c}  \leq \lambda \leq c \mu,$ we have
$$
\norm{M_\lambda^{\mu/2}g(1-M_\lambda^\mu)}{L^2\rightarrow L^2}{}\leq C \tau^N e^{\frac{\tau^2}{\lambda}}e^{2 \delta \tau}e^{-c_0 \mu}.
$$
\end{lem}

\begin{proof}
This is essentially Lemma 2.17 in \cite{Laurent_2018}. Its proof is based upon Lemma 2.15 from \cite{Laurent_2018} and there the assumptions on $\psi$ and $f$ are more restrictive. However, since we only need to use a version of~\cite[Lemma 2.17]{Laurent_2018} for the case $k=0$ we see that the proof works as well in the less restrictive case of Lipschitz regularity for $\psi$ (our function $\psi$ remains a quadratic polynomial in $t$) and boundedness for $f$.
\end{proof}
Lemma~\ref{lemma 2.16 from ll} is used Section~\ref{proof of local quant estimate} for the first estimates on the terms $B_*$. In our case $f$ will be equal to functions which will be either independent of $t$ or simply a polynomial in $t$, therefore the real analyticity property is preserved as well.

The following lemma is taken from~\cite{LL:17Hypo} and allows to replace in the quantitative estimates the observation term $\norm{u}{H^1(\omega)}{}$ by the weaker $\norm{u}{L^2(\omega)}{}$. It is used in the proof of the semi-global estimate of Theorem~\ref{semi global }.

\begin{lem}[Lemma 5.2 in \cite{LL:17Hypo}]
\label{lemma from ll hypo}
Let $\Omega$ be a bounded set of $\R^{n+1}=\R_t\times\R^n_x$. Let $P$ be a differential operator of order $2$, defined in a neighborhood of $\Omega$, with real principal symbol and coefficients independent of the variable $t$. Suppose as well that $P$ is elliptic in $\{\xi_t=0\}$. Let $\omega \Subset \Omega$ and $\theta \in C^\infty_0(\omega)$. Then there exists $C>0$ such that for all $u \in C^\infty_0(\R^{n+1})$ and $\mu \geq 1$, we have
$$
\norm{M^\mu_\mu u}{H^1}{}\leq C\mu \norm{u}{L^2(\omega)}{}+C\norm{Pu}{L^2(\Omega)}{}+Ce^{- c \mu}\norm{u}{H^1}{}.
$$
\end{lem}

\small \bibliographystyle{alpha}
\bibliography{bibl}

\end{document}